\documentclass[12pt]{amsart}
\usepackage{amssymb,amscd}
\usepackage{pb-diagram}
\usepackage{verbatim}
\usepackage{graphicx}
\usepackage{rotating}
\usepackage{amssymb,amscd,enumerate}
\def\a{\alpha}
\def\b{\beta}

\def\G{\Gamma}

\def\lb\{{\left\{}
\def\la{\lambda}
\def\La{\Lambda}
\def\lla{\longleftarrow}
\def\lm{\limits}
\def\lra{\longrightarrow}
\def\dllra{\Longleftrightarrow}
\def\llra{\longleftrightarrow}
\def\n{\nabla}
\def\ngth{\negthickspace}
\def\ola{\overleftarrow}
\def\Om{\Omega}
\def\om{\omega}
\def\op{\oplus}
\def\oper{\operatorname}
\def\oplm{\operatornamewithlimits}
\def\ora{\overrightarrow}
\def\ov{\overline}
\def\ova{\overarrow}
\def\ox{\otimes}
\def\p{\partial}
\def\rb\}{\right\}}
\def\s{\sigma}
\def\sbq{\subseteq}
\def\spq{\supseteq}
\def\sqp{\sqsupset}
\def\supth{{\text{th}}}
\def\T{\Theta}
\def\th{\theta}
\def\tl{\tilde}
\def\thra{\twoheadrightarrow}
\def\un{\underline}
\def\ups{\upsilon}
\def\vp{\varphi}
\def\wh{\widehat}
\def\wt{\widetilde}
\def\x{\times}
\def\z{\zeta}
\def\({\left(}
\def\){\right)}
\def\[{\left[}
\def\]{\right]}
\def\<{\left<}
\def\>{\right>}

\def\tec{Teichm\"uller\ }
\def\sconr{\hbox{\medspace\vrule width 0.4pt height 4.7pt depth
0.4pt \vrule width 5pt height 0pt depth 0.4pt\medspace}}

\def\SA{\mathcal A}
\def\SB{\mathcal B}
\def\SC{\mathcal C}
\def\SD{\mathcal D}
\def\SE{\mathcal E}
\def\SF{\mathcal F}
\def\SG{\mathcal G}
\def\SH{\mathcal H}
\def\SI{\mathcal I}
\def\SJ{\mathcal J}
\def\SK{\mathcal K}
\def\SL{\mathcal L}
\def\SM{\mathcal M}
\def\SN{\mathcal N}
\def\SO{\mathcal O}
\def\SP{\mathcal P}
\def\SQ{\mathcal Q}
\def\SR{\mathcal R}
\def\SS{\mathcal S}
\def\ST{\mathcal T}
\def\SU{\mathcal U}
\def\SV{\mathcal V}
\def\SW{\mathcal W}
\def\SX{\mathcal X}
\def\SY{\mathcal Y}
\def\SZ{\mathcal Z}

\newcommand{\BA}{\ensuremath{\mathbf A}}
\newcommand{\BB}{\ensuremath{\mathbf B}}
\newcommand{\BC}{\ensuremath{\mathbf C}}
\newcommand{\BD}{\ensuremath{\mathbf D}}
\newcommand{\BE}{\ensuremath{\mathbf E}}
\newcommand{\BF}{\ensuremath{\mathbf F}}
\newcommand{\BG}{\ensuremath{\mathbf G}}
\newcommand{\BH}{\ensuremath{\mathbf H}}
\newcommand{\BI}{\ensuremath{\mathbf I}}
\newcommand{\BJ}{\ensuremath{\mathbf J}}
\newcommand{\BK}{\ensuremath{\mathbf K}}
\newcommand{\BL}{\ensuremath{\mathbf L}}
\newcommand{\BM}{\ensuremath{\mathbf M}}
\newcommand{\BN}{\ensuremath{\mathbf N}}
\newcommand{\BO}{\ensuremath{\mathbf O}}
\newcommand{\BP}{\ensuremath{\mathbf P}}
\newcommand{\BQ}{\ensuremath{\mathbf Q}}
\newcommand{\BR}{\ensuremath{\mathbf R}}
\newcommand{\BS}{\ensuremath{\mathbf S}}
\newcommand{\BT}{\ensuremath{\mathbf T}}
\newcommand{\BU}{\ensuremath{\mathbf U}}
\newcommand{\BV}{\ensuremath{\mathbf V}}
\newcommand{\BW}{\ensuremath{\mathbf W}}
\newcommand{\BX}{\ensuremath{\mathbf X}}
\newcommand{\BY}{\ensuremath{\mathbf Y}}
\newcommand{\BZ}{\ensuremath{\mathbf Z}}

\newcommand{\rank}{\operatorname{rank}}

\def\bba{{\mathbb A}}
\def\bbb{{\mathbb B}}
\def\bbc{{\mathbb C}}
\def\bbd{{\mathbb D}}
\def\bbe{{\mathbb E}}
\def\bbf{{\mathbb F}}
\def\bbg{{\mathbb G}}
\def\bbh{{\mathbb H}}
\def\bbi{{\mathbb I}}
\def\bbj{{\mathbb J}}
\def\bbk{{\mathbb K}}
\def\bbl{{\mathbb L}}
\def\bbm{{\mathbb M}}
\def\bbn{{\mathbb N}}
\def\bbo{{\mathbb O}}
\def\bbp{{\mathbb P}}
\def\bbq{{\mathbb Q}}
\def\bbr{{\mathbb R}}
\def\bbs{{\mathbb S}}
\def\bbt{{\mathbb T}}
\def\bbu{{\mathbb U}}
\def\bbv{{\mathbb V}}
\def\bbw{{\mathbb W}}
\def\bbx{{\mathbb X}}
\def\bby{{\mathbb Y}}
\def\bbz{{\mathbb Z}}

\newtheorem*{Whitney towers}{Theorem~\ref{Whitney towers}}
\newtheorem*{h-towers}{Theorems ~\ref{half} \& \ref{$(n)$-solvable}}

\newtheorem*{surgery curves}{Theorem~\ref{surgery curves}}
\newtheorem*{cg=0}{Theorem~\ref{vanish}}

\newtheorem{thm}{Theorem}[section]
\newtheorem{lem}[thm]{Lemma}
\newtheorem{cor}[thm]{Corollary}

\newtheorem{prop}[thm]{Proposition}
\newtheorem{defn}[thm]{Definition}
\newtheorem{rem}[thm]{Remark}

\newtheorem{ex}[thm]{Example}

\numberwithin{equation}{section}
\numberwithin{figure}{section}

\newcommand{\spec}{\operatorname{spec}}
\newcommand{\Herm}{\operatorname{Herm}}
\newcommand{\GL}{\operatorname{GL}}
\newcommand{\Ker}{\operatorname{Ker}}
\newcommand{\End}{\operatorname{End}}
\newcommand{\Hom}{\operatorname{Hom}}
\newcommand{\Ext}{\operatorname{Ext}}
\newcommand{\Rep}{\operatorname{Rep}}
\newcommand{\Tors}{\operatorname{Tors}}
\newcommand{\dom}{\operatorname{dom}}
\newcommand{\tr}{\operatorname{tr}}
\newcommand{\id}{\operatorname{id}}
\newcommand{\spin}{\operatorname{Spin}}

\newcommand{\Z}{\mathbb{Z}}
\newcommand{\N}{\mathbb{N}}
\newcommand{\C}{\mathbb{C}}
\newcommand{\Q}{\mathbb{Q}}
\newcommand{\K}{\mathbb{K}}
\newcommand{\R}{\mathbb{R}}
\newcommand{\F}{\mathbb{F}}

\newcommand{\RR}{\mathcal{R}}
\newcommand{\UU}{\mathcal{U}}
\newcommand{\NN}{\mathcal{N}}
\newcommand{\DD}{\mathcal{D}}
\newcommand{\KK}{\mathcal{K}}
\newcommand{\FF}{\mathcal{F}}
\newcommand{\MM}{\mathcal{M}}
\newcommand{\LL}{\mathcal{L}}
\newcommand{\CC}{\mathcal{C}}

\newcommand{\QQ}{\mathcal{Q}}
\newcommand{\PP}{\mathcal{P}}

\newcommand{\sra}{\twoheadrightarrow}
\newcommand{\ira}{\rightarrowtail}
\newcommand{\sd}{\rtimes}
\newcommand{\ra}{\longrightarrow}

\newcommand{\lr}{\longleftrightarrow}
\def\x{\times}
\def\p{\partial}
\def\ov{\overline}
\def\Om{\Omega}
\def\s{\sigma}
\def\lra{\longrightarrow}
\renewcommand{\SS}{\mathcal{S}}
\renewcommand{\AA}{\mathcal{A}}
\renewcommand{\l}{\ell}
\renewcommand{\a}{\alpha}
\renewcommand{\i}{\iota}
\renewcommand{\b}{\beta}

\newcommand{\torsionp}{\Z_{(p)}/\Z}
\newcommand{\defeq}{\stackrel{\mathrm{def}}{=}}
\newcommand{\2}{\Z[\pi/\pi^{(2)}]}

\newcommand{\gu}{\G_n^U}
\newcommand{\go}{\G_0^U}
\newcommand{\1}{\pi_1}
\newcommand{\rk}{\text{rank}}

\renewcommand{\dgeverylabel}{\displaystyle}

\title{Knot Concordance and Higher-Order Blanchfield Duality}
\author{Tim D. Cochran$^{\dag}$}
\address{Department of Mathematics, Rice University, Houston, Texas, 77005-1892}
\email{cochran@rice.edu}

\author{Shelly Harvey$^{\dag\dag}$}
\address{Department of Mathematics, Rice University, Houston, Texas, 77005-1892}
\email{shelly@rice.edu}

\author{Constance Leidy}
\address{Wesleyan University, Wesleyan Station, Middletown, CT 06459}
\email{cleidy@wesleyan.edu}

\thanks{\noindent $^{\dag}$The author was partially supported by the NSF DMS-0406573 and DMS-0706929}
\thanks{ $^{\dag\dag}$The author was partially supported
by NSF DMS-0539044 and The Alfred P. Sloan Foundation}

\begin{document}
\begin{abstract} In 1997, T. Cochran, K. Orr, and P. Teichner \cite{COT} defined a filtration
of the classical knot concordance group $\mathcal{C}$,
$$
\cdots \subseteq \mathcal{F}_{n} \subseteq \cdots \subseteq
\mathcal{F}_1\subseteq \mathcal{F}_{0.5} \subseteq \mathcal{F}_{0} \subseteq \mathcal{C}.
$$
The filtration is important because of its strong connection to the classification of topological $4$-manifolds. Here we introduce new techniques for studying $\mathcal{C}$ and use them to prove that, for each $n \in \mathbb{N}_0$, the group $\mathcal{F}_{n}/\mathcal{F}_{n.5}$ has infinite rank. We establish the same result for the corresponding filtration of the smooth concordance group. We also resolve a long-standing question as to
whether certain natural families of knots, first considered by Casson-Gordon, and Gilmer, contain slice knots.
\end{abstract}
\maketitle

\section{Introduction}\label{sec:Introduction}

A (classical) \textbf{knot} $J$ is the image of a tame embedding of an oriented circle in $S^3$. A \textbf{slice knot} is a knot that bounds an embedding of a $2$-disk in $B^4$. We wish to consider both the \emph{smooth} category and the \emph{topological} category (in the latter case all embeddings are required to be flat). The question of which knots are slice knots was first considered by Kervaire and Milnor in the early $60's$ in their study of isolated singularities of $2$-spheres in $4$-manifolds in the context of a surgery-theoretic scheme for classifying $4$-dimensional manifolds. Indeed, certain concordance problems are known to be \emph{equivalent} to whether the surgery techniques that were so successful in higher-dimensions, ``work'' for topological $4$-manifolds ~\cite{CF}. Thus the question of which knots are slice knots lies at the heart of the topological classification of $4$-dimensional manifolds. Moreover the question of which knots are topologically slice but not smoothly slice may be viewed as ``atomic'' for the question of which topological $4$-manifolds admit distinct smooth structures.

There is an equivalence relation on knots wherein slice knots are equivalent to the trivial knot. Two knots, $J_0\hookrightarrow S^3\times \{0\}$ and $J_1\hookrightarrow S^3\times \{1\}$, are \textbf{concordant} if there exists a proper embedding of an annulus into $S^3\times [0,1]$ that restricts to the knots on $S^3\times \{0,1\}$. A knot is concordant to a trivial knot if and only if it is a slice knot. The connected sum operation endows the set of all concordance classes of knots with the structure of an abelian group, called the \textbf{topological knot concordance group}, $\mathcal{C}$, which is a quotient of its smooth analogue $\mathcal{C}^s$. For excellent surveys see ~\cite{Go1} and ~\cite{Li1}.

In this paper we introduce new techniques for showing knots are not topologically slice (and hence also not smoothly slice). As one application we resolve a long-standing question about whether certain natural families of knots contain non-slice knots (some of these results were announced in ~\cite{CHL1}). As another major application we establish that each quotient, $\mathcal{F}_{n}/\mathcal{F}_{n.5}$, in the Cochran-Orr-Teichner filtration $\{\mathcal{F}_{n}\}$ of $\mathcal{C}$, has infinite rank (the same result is shown for the filtration of $\mathcal{C}^s$). This was previously known only for $n=0,1,$ and $2$. Our proof of the latter avoids two ad hoc technical tools employed by Cochran-Teichner, one of which was a deep analytical bound of Cheeger-Gromov for their von Neumann $\rho$ invariants.

In the late $60$'s Levine \cite{L5} (see also ~\cite{Sto}) defined an epimorphism from $ \mathcal{C}$ to $\mathbb{Z}^\infty \oplus \mathbb{Z}_2^\infty \oplus \mathbb{Z}_4^\infty$, given by the Arf invariant, certain discriminants and twisted signatures associated to the infinite cyclic cover of the knot complement. A knot for which these invariants vanish is called an \textbf{algebraically slice knot}. Thus the question at that time was ``Is every algebraically slice knot actually a slice knot?'' A simple way to create potential counterexamples is to begin with a known slice knot, $R$, such as the $9_{46}$ knot shown on the left-hand side of Figure~\ref{fig:ribbonCG}, and ``tie the bands into some knot $J_0$'', as shown schematically on the right-hand side of Figure~\ref{fig:ribbonCG}. An example of a band tied into a trefoil knot is shown in Figure~\ref{fig:knottedband}. All of these genus one knots are algebraically slice since they have the same Seifert matrix as the slice knot $R$. Similar knots have appeared in the majority of papers on this subject (for example ~\cite{Li1}\cite{Li5}\cite{Li7}\cite{Li10}\cite{GL1}).

\begin{figure}[htbp]
\setlength{\unitlength}{1pt}
\begin{picture}(327,151)
\put(0,0){\includegraphics{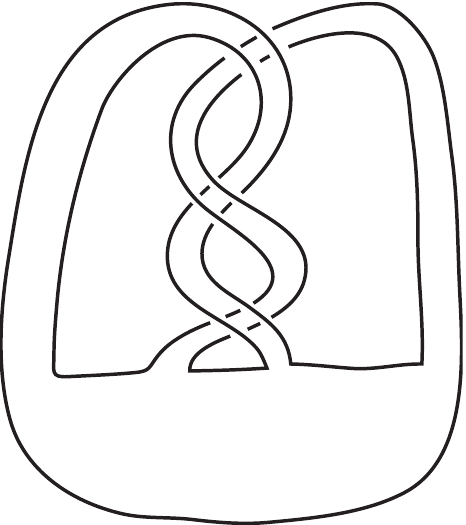}}
\put(184,0){\includegraphics{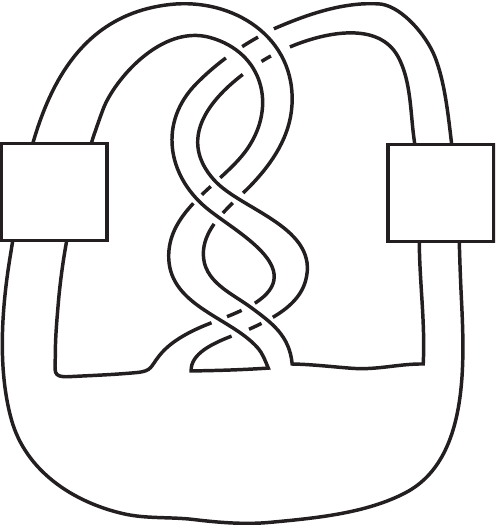}}
\put(196,93){$J_0$}
\put(307,93){$J_0$}
\put(154,93){$J_1\equiv$}
\put(-30,93){$R\equiv$}
\end{picture}
\caption{Algebraically Slice Knots $J_1$ Patterned on the Slice Knot $R$}\label{fig:ribbonCG}
\end{figure}
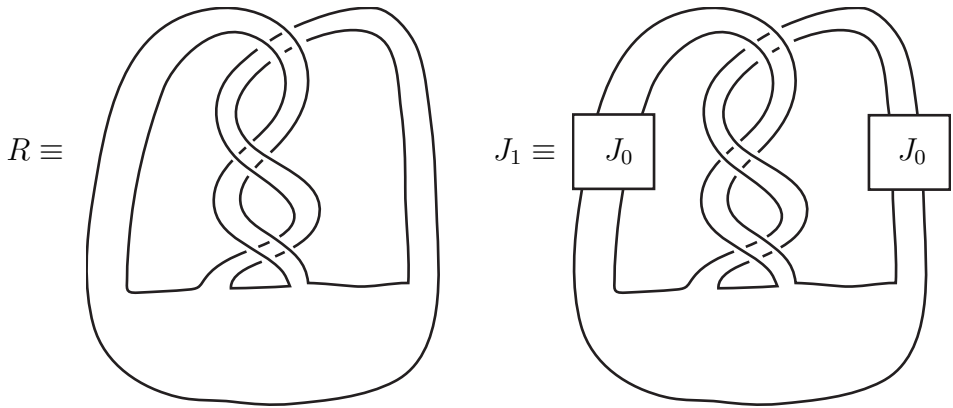
\begin{figure}[htbp]
\setlength{\unitlength}{1pt}
\begin{picture}(165,151)
\put(0,0){\includegraphics{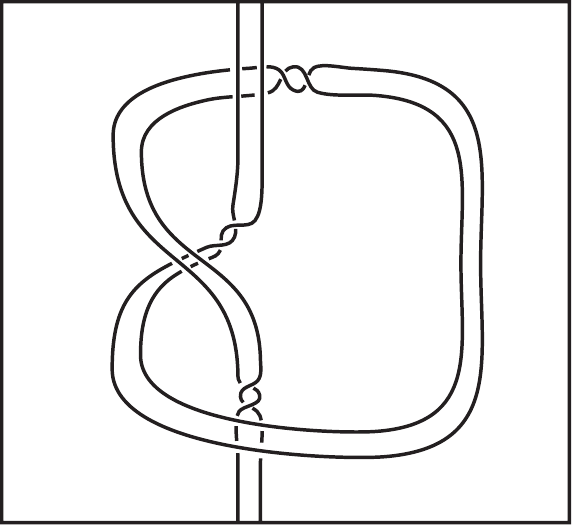}}
\end{picture}
\caption{Tying a band into a trefoil knot}\label{fig:knottedband}
\end{figure}
In the early $70$'s Casson and Gordon defined new knot concordance invariants via dihedral covers \cite{CG1}~\cite{CG2}. These ``higher-order signature invariants'' were used to show that some of the knots $J_1$ of Figure~\ref{fig:ribbonCG} are not slice knots. P. Gilmer showed that these higher-order signature invariants for $J_1$ are equal to certain combinations of classical signatures of $J_0$ and thus the latter constituted higher-order obstructions to $J_1$ being a slice knot \cite{Gi3}\cite{Gi5}~(see ~\cite{Li6} for $2$-torsion invariants). These invariants were also used to show that the subgroup of algebraically slice knots has infinite rank ~\cite{Ji1}. Hence the question arose:``What if $J_0$ itself were algebraically slice?'' Thus shortly after the work of Casson and Gordon the self-referencing family of knots shown in Figure~\ref{fig:family} was considered by Casson, Gordon, Gilmer and others ~\cite{Gi1}. An example with $n=3$ and $J_0=U$, the unknot, is shown in Figure~\ref{fig:R3}.

\begin{figure}[htbp]
\setlength{\unitlength}{1pt}
\begin{picture}(143,151)
\put(0,0){\includegraphics{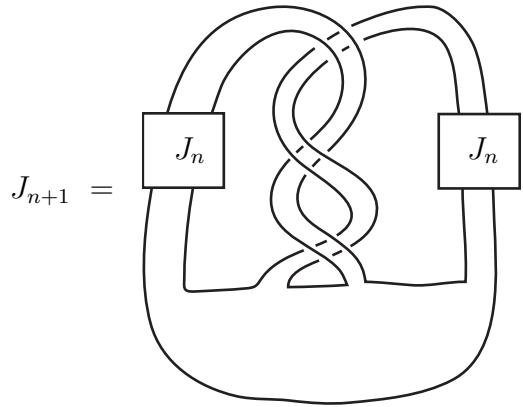}}
\put(-50,78){$J_{n+1}~=$}
\put(12,93){$J_{n}$}
\put(123,93){$J_{n}$}
\end{picture}
\caption{The recursive family $J_{n+1}, n\geq 0$}\label{fig:family}
\end{figure}

\begin{figure}[htbp]
\setlength{\unitlength}{1pt}
\begin{picture}(192,266)
\put(0,0){\includegraphics{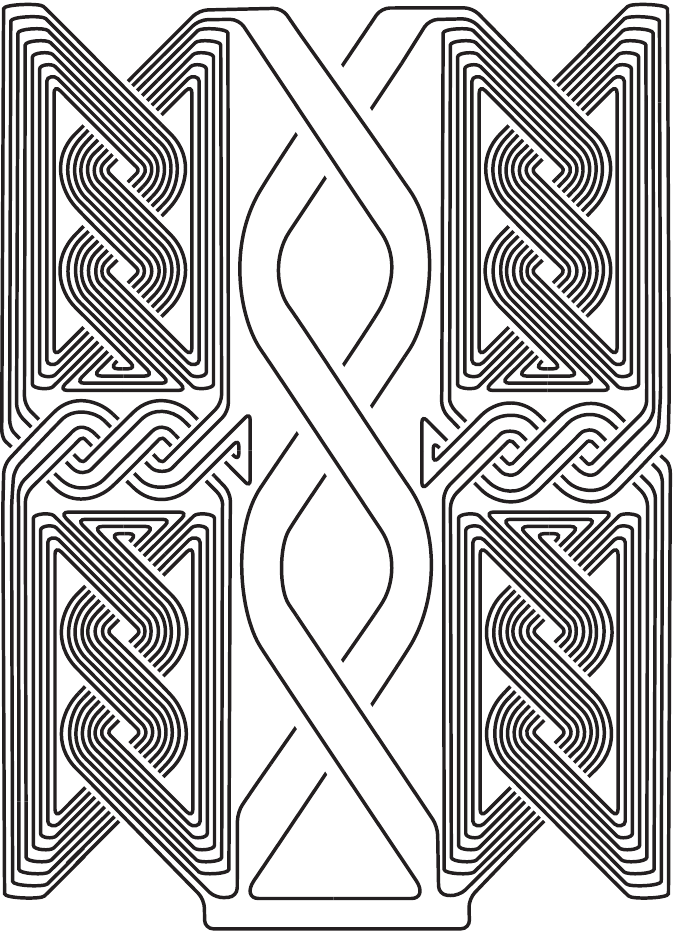}}
\end{picture}
\caption{The Ribbon Knot $J_3$ for $J_0=$ the unknot}\label{fig:R3}
\end{figure}
All of the invariants above vanish for $J_n$ if $n\geq 2$ and it is not difficult to see that if $J_0$ is itself a slice knot then each $J_n$ is a slice knot. It was asked whether or not $J_n$ is always a slice knot. In fact, Gilmer proved (unpublished) that $J_2$, for certain $J_0$, is not a \emph{ribbon knot} ~\cite{Gi1}. However the status of the knots $J_n$ has remained open for $25$ years. Much more recently, Cochran, Orr and Teichner, Friedl, and Kim used higher-order signatures associated to solvable covers of the knot complement to find non-slice knots that could not be detected by the invariants of Levine or Casson-Gordon \cite{COT}\cite{COT2}\cite{Ki1}\cite{Fr2}. In fact the techniques of \cite{COT}, ~\cite{COT2}, ~\cite{CT} and ~\cite{CK} were limited to knots of genus at least $2$ (note each $J_n$ has genus $1$) because of their use of localization techniques.

Recall that to each knot $K$ and each point on the unit circle in $\mathbb{C}$, Levine associated a signature. This endows each knot with an integral-valued  \textbf{signature function} defined on the circle. Let $\mathbf{\rho_0(K)}$ denote the integral of this function over the unit circle, normalized to have length $1$. This should be viewed as the average of the Levine signatures for $K$.

We prove:

\newtheorem*{thm:main}{Theorem~\ref{thm:main}}
\begin{thm:main} There is a constant $C$ such that if $|\rho_0(J_0)|>C$, then for each $n\ge 0$, $J_n$ is of infinite order in the topological and smooth knot concordance groups. Furthermore, there is constant $D$ such that if $J_2$ is a slice knot then $\rho_0(J_0)\in\{0,D\}$ (Theorem~\ref{thm:J2notslice}). \end{thm:main}

\noindent This was classically known only for $n=0,1$, using the Levine signatures and Casson-Gordon invariants respectively. The constant $D$ is a specific real number associated to the $9_{46}$ knot that may in fact be $0$.

In 1997, T. Cochran, K. Orr, and P. Teichner \cite{COT} defined an important filtration
of the classical knot concordance group $\mathcal{C}$,
$$
\cdots \subseteq \mathcal{F}_{n} \subseteq \cdots \subseteq
\mathcal{F}_1\subseteq \mathcal{F}_{0.5} \subseteq \mathcal{F}_{0} \subseteq \mathcal{C}.
$$
The elements of $\mathcal{F}_{n}$ are called the \textbf{$(n)$-solvable knots}. This filtration is geometrically significant because it measures the successive failure of the Whitney trick for $2$-disks in $4$-manifolds and hence is closely related to Freedman's topological classification scheme for $4$-dimensional manifolds. The filtration is also natural because it exhibits all of the previously known concordance invariants in its associated graded quotients of low degree: $\mathcal{F}_{0}$ is precisely the set of
knots with Arf invariant zero, $\mathcal{F}_{0.5}$ is precisely Levine's subgroup of algebraically slice knots, and $\mathcal{F}_{1.5}$ contains all knots with vanishing Casson-Gordon invariants. The filtration was also shown to be non-trivial: \cite{COT2} established that the abelian group $\mathcal{F}_{2}/\mathcal{F}_{2.5}$ has infinite rank; Cochran-Teichner showed in \cite{CT} that each of the groups $\mathcal{F}_{n}/\mathcal{F}_{n.5}$ has rank at least $1$.

Our second major result (known previously for $n=0,1,2$) is:

\begin{thm} For each $n \in \mathbb{N}_0$, the group $\mathcal{F}_{n}/\mathcal{F}_{n.5}$ has infinite rank.
\end{thm}

We note that the construction of our examples is done completely in the smooth category  so that we also establish the corresponding statements about the Cochran-Orr-Teichner filtration of the \emph{smooth} knot
 concordance group (In fact it can be shown that the natural map induces an \emph{isomorphisms} $\mathcal{F}_{n}^{smooth}/\mathcal{F}_{n.5}^{smooth}\cong \mathcal{F}_{n}/\mathcal{F}_{n.5}!$). Our technique also recovers the result of Cochran-Teichner, while eliminating two highly technical steps from their proof. In particular our proof does not rely on the analytical bound of Cheeger-Gromov.
 Moreover we use the knots $J_n$ (for suitably chosen $J_0$) to prove this. This family is simpler than the examples of Cochran and Teichner. In fact the families $J_n$ are distinct even up to concordance from
 the examples of Cochran and Teichner (this result will appear in another paper). We employ the Cheeger-Gromov von Neumann $\rho$-invariants and higher-order Alexander modules that were introduced in ~\cite{COT}. Our new technique is to expand
  upon previous results of Leidy concerning higher-order Blanchfield linking forms \emph{without localizing the coefficient system} ~\cite{Lei1}~\cite{Lei3}. This is used to show that certain elements of $\pi_1$ of a slice knot exterior cannot
 lie in the kernel of the map into any slice disk(s) exterior. Another new feature is the essential use of equivalence relations that are much weaker than concordance and $(n)$-solvability.

These techniques provide other new information about the order of knots in the concordance group. For example, consider the family of knots below where $J_{n-1}$, $n\geq 2$, is one of the the algebraically slice knots
 above.  For any such $K_{n}$, $K_{n}\# K_{n}$ is algebraically slice and has vanishing Casson-Gordon invariants. Therefore $K_{n}$ cannot be distinguished from an order $2$ knot by these invariants.
\begin{figure}[htbp]
\setlength{\unitlength}{1pt}
\begin{picture}(185,118)
\put(0,0){\includegraphics{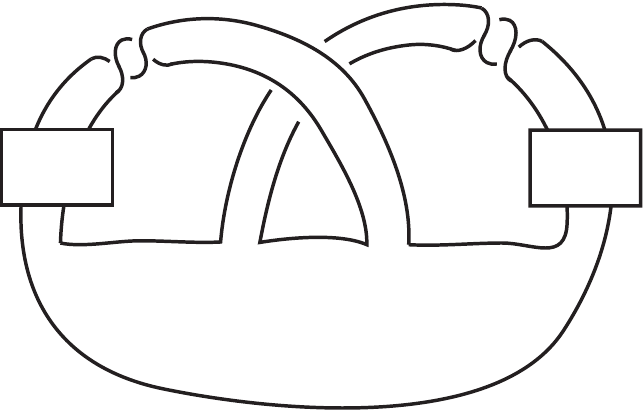}}
\put(-40,60){$K_{n}=$}
\put(5,67){$J_{n-1}$}
\put(158,67){$J_{n-1}$}
\end{picture}
\caption{Knots potentially of order 2 in the concordance group}\label{fig:torsionfigeight}
\end{figure}
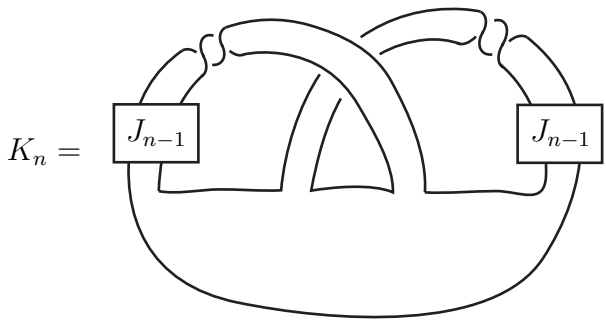
However we show:
\newtheorem*{cor:torsion}{Corollary~\ref{cor:torsion}}
\begin{cor:torsion} For any $n$ there is a constant $D$ such that if $|\rho_{0}(J_0)|>D$  then $K_{n}$ is of infinite order in the smooth and topological concordance groups.
\end{cor:torsion}

The specific families of knots of Figure~\ref{fig:family} are important because of their simplicity and their history. However, they are merely particular instances of a more general `doubling' phenomenon to which our techniques may be applied. In order to state these results, we review a method we will use to construct examples. Let $R$ be a knot in $S^3$ and $\{\eta_1,\eta_2,\ldots,\eta_m\}$ be an oriented trivial link in $S^3$, that misses $R$, bounding a collection of disks that meet $R$ transversely as shown on the left-hand side of Figure~\ref{fig:infection}. Suppose $\{K_1,K_2,\ldots,K_m\}$ is an $m$-tuple of auxiliary knots. Let $R(\eta_1,\ldots,\eta_m\,K_1,\ldots,K_m)$ denote the result of the operation pictured in Figure~\ref{fig:infection}, that is, for each $\eta_j$, take the embedded disk in $S^3$ bounded by $\eta_j$; cut off $R$ along the disk; grab the cut strands, tie them into the knot $K_j$ (with no twisting) and reglue as shown in Figure~\ref{fig:infection}.

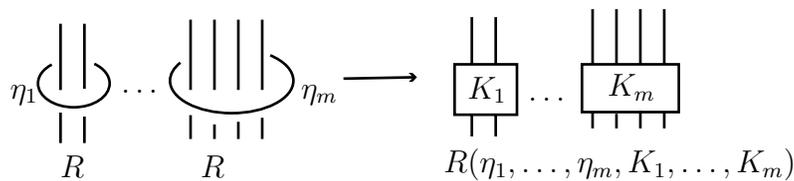
\begin{figure}[htbp]
\setlength{\unitlength}{1pt}
\begin{picture}(262,71)
\put(10,37){$\eta_1$} \put(120,37){$\eta_m$} \put(52,39){$\dots$}
\put(206,36){$\dots$} \put(183,37){$K_1$} \put(236,38){$K_m$}
\put(174,9){$R(\eta_1,\dots,\eta_m,K_1,\dots,K_m)$}
\put(29,7){$R$} \put(82,7){$R$}
\put(20,20){\includegraphics{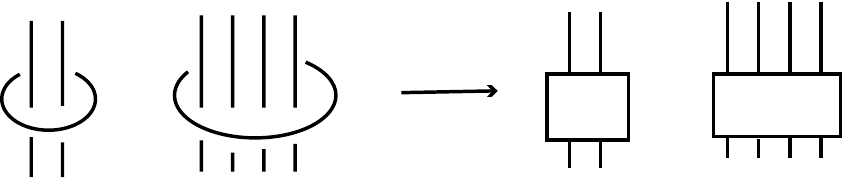}}
\end{picture}
\caption{$R(\eta_1,\dots,\eta_m,K_1,\dots,K_m)$:
Infection of $R$ by $K_j$ along $\eta_j$}\label{fig:infection}
\end{figure}
\noindent We will call this the \textbf{result of infection performed on the knot} $\mathbf{R}$ \textbf{using the infection knots} $\mathbf{K_j}$ \textbf{along the curves} $\mathbf{\eta_j}$. This construction can also be
described in the following way. For each $\eta_j$, remove a tubular neighborhood of $\eta_j$ in $S^3$ and glue in the exterior of a tubular neighborhood of $K_j$ along their common boundary, which is a
torus, in such a way that the longitude of $\eta_j$ is identified with the meridian of $K_j$ and the meridian of $\eta_j$ with the reverse of the longitude of $K_j$. The resulting space can be seen to be homeomorphic to $S^3$ and the image of $R$ is the new knot. In the case that $m=1$ this is the same as the classical satellite construction. In general it can be considered to be a ``generalized satellite construction'', widely utilized in the study of knot concordance. In the case that $m=1$ and $lk(\eta,R)=0$ it is precisely the same as forming a satellite of $J$ with winding number zero. This yields an operator
$$
R_{\eta}:\mathcal{C}\to \mathcal{C}.
$$
that has been studied (e.g. ~\cite{LiM}). For general $m$ with $lk(\eta_j,R)=0$, it can be considered as a \textbf{generalized doubling operator}, $R_{\eta_j}$, parameterized by $(R,\{\eta_{j}\})$
$$
R_{\eta_j}:~\mathcal{C}\times\dots\times\mathcal{C}\to \mathcal{C}.
$$
If, for simplicity, we assume that all ``input knots'' are identical then such an operator is a function
$$
R_{\eta_j}:~\mathcal{C}\to \mathcal{C}.
$$

\noindent A primary example is the ``$R$-doubling'' operation of going from the left-hand side of Figure~\ref{fig:ribbonCG} to the right-hand side.
Here $R$ is the $9_{46}$ knot and $\{\eta_1,\eta_2\}=\{\alpha,\beta\}$ are as shown on the left-hand side of Figure~\ref{fig:Rdoubling}.
 The image of a knot $K$ under the operator $R_{\alpha,\beta}$ is denoted by $R(K)$ and is shown on the right-hand side of Figure~\ref{fig:Rdoubling}.
  Note that our previously defined knot $J_1$ is the same as $R(J_0)$ and that $K_{n}$ of Figure~\ref{fig:torsionfigeight} is $\bar{R}(J_{n-1})$ where $\bar{R}$ is the figure-eight knot.

\begin{figure}[htbp]
\setlength{\unitlength}{1pt}
\begin{picture}(357,151)
\put(0,0){\includegraphics{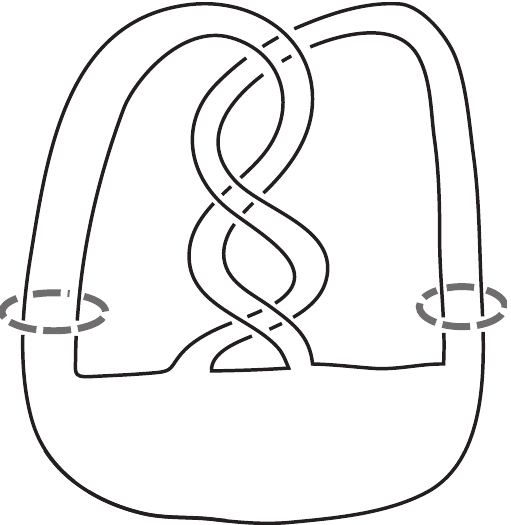}}
\put(214,0){\includegraphics{family_scaled.pdf}}
\put(-10,60){$\alpha$}
\put(148,60){$\beta$}
\put(226,93){$K$}
\put(337,93){$K$}
\put(168,93){$R(K)\equiv$}
\put(-40,93){$R_{\alpha,\beta}~\equiv$}
\end{picture}
\caption{$R$-doubling}\label{fig:Rdoubling}
\end{figure}

Most of the results of this paper concern to what extent these functions are injective. Because of the condition on ``winding numbers'', $lk(\eta_j,R)=0$,
if $R$ is a slice knot, the images of such operators $R$ contain only knots for which the classical invariants vanish. Thus iterations of these operators, \textbf{iterated generalized doubling}, produce increasingly subtle knots. This claim is quantified by the following.

\newtheorem*{thm:nsolvable}{Theorem~\ref{thm:nsolvable}}
\begin{thm:nsolvable}[see also {\cite[proof of Proposition 3.1]{COT2}}] If $R_i$, $1\leq i\leq n$, are slice knots and
 $\eta_{ij}\in \pi_1(S^3-R_i)^{(1)}$ then
$$
R_n\circ\dots\circ R_2\circ R_1(\mathcal{F}_{0})\subset \mathcal{F}_{n},
$$
where we  abbreviate $(R_i)_{\eta_{ij}}$ by $R_i$.
\end{thm:nsolvable}

\noindent For example the knot $J_n$ is the result of $n$ iterations of the $R_{\alpha,\beta}$ operator shown above
$$
\mathcal{C}\overset{R}\longrightarrow \mathcal{C}\to\dots\to \mathcal{C}\overset{R}\rightarrow\mathcal{C}
$$
applied to some initial knot $J_0=K$. More generally let us define an \textbf{n-times iterated generalized doubling} to be such a composition of operators using possibly different slice knots $R_i$, and different curves $\eta_{i1},\dots,\eta_{im_i}$.

Then our main proof establishes:

\newtheorem*{thm:main3}{Theorem~\ref{thm:main3}}
\begin{thm:main3}  Suppose $R_i$, $1\leq i\leq n$, is a set of (not necessarily distinct) slice knots. Suppose that, for each fixed $i$, $\{\eta_{i1},...,\eta_{im_i}\}$ is a trivial link of circles in $\pi_1(S^3-R_i)^{(1)}$ such that for some $ij$ and $ik$ (possibly equal) $\mathcal{B}\ell_0^i(\eta_{ij},\eta_{ik})\neq 0$, where $\mathcal{B}\ell_0^i$ is the classical Blanchfield form of $R_i$. Then there exists a constant $C$ such that if $K$ is any knot with Arf$(K)=0$ and $|\rho_0(K)|>C$, the result, $R_{n}\circ\dots\circ R_{1}(K)$, of n-times iterated generalized doubling, is of infinite order in the smooth and topological concordance groups, and moreover represents an element of infinite order in $\mathcal{F}_{n}/\mathcal{F}_{n.5}$.
\end{thm:main3}

Note that any set $\{\eta_{i1},\dots,\eta_{im_i}\}$ that generates a submodule of the Alexander module of $R_i$ of more than half rank necessarily satisfies the condition of Theorem~\ref{thm:main3}, because of the non-singularity of the Blanchfield form.

\section{Higher-Order Signatures and How to Calculate Them}\label{signatures}

In this section we review the von Neumann $\rho$-invariants and explain to what extent they are concordance invariants. We also show how to calculate them for knots or links that are obtained from the infections defined in Section~\ref{sec:Introduction}.

The use of variations of Hirzebruch-Atiyah-Singer signature defects associated to covering spaces is a theme common to most of the work in the field of knot and link concordance since the 1970's. In particular, Casson and Gordon initiated their use in cyclic covers ~\cite{CG1}~\cite{CG2}; Farber, Levine and Letsche initiated the use of signature defects associated to general (finite) unitary representations ~\cite{L6}~\cite{Let}; and Cochran-Orr-Teichner initiated the use of signatures associated to the left regular representations ~\cite{COT}. See ~\cite{Fr2} for a beautiful comparison of these approaches in the metabelian case.

Given a compact, oriented 3-manifold $M$, a discrete group $\G$, and a representation $\phi : \pi_1(M)
\to \G$, the \textbf{von Neumann
$\mathbf{\rho}$-invariant} was defined by Cheeger and Gromov by choosing a Riemannian metric and using $\eta$-invariants associated to $M$ and its covering space induced by $\phi$. It can be thought of as an oriented homeomorphism invariant associated to an arbitrary regular covering space of $M$ ~\cite{ChGr1}. If $(M,\phi) = \partial
(W,\psi)$ for some compact, oriented 4-manifold $W$ and $\psi : \pi_1(W) \to \G$, then it is known that $\rho(M,\phi) =
\s^{(2)}_\G(W,\psi) - \s(W)$ where $\s^{(2)}_\G(W,\psi)$ is the
\textbf{$\mathbf{L^{(2)}}$-signature} (von Neumann signature) of the intersection form defined on
$H_2(W;\mathbb{Z}\G)$ twisted by $\psi$ and $\sigma(W)$ is the ordinary
signature of $W$ ~\cite{LS}. In the case that $\G$ is a poly-(torsion-free-abelian) group (abbreviated \textbf{PTFA group} throughout), it follows that $\mathbb{Z}\G$ is a right Ore domain that embeds into its (skew) quotient field of fractions $\mathcal{K}\G$ ~\cite[pp.591-592, ~Lemma 3.6ii p.611]{P}. In this case $\s^{(2)}_\G$  is a function of the Witt class of the equivariant intersection form on $H_2(W;\mathcal{K}\G)$ ~\cite[Section 5]{COT}. In the special case that this form is non-singular (such as $\beta_1(M)=1$), it can be thought of as a homomorphism from
$L^0(\mathcal{K}\G)$ to $\mathbb{R}$.

All of the coefficient systems $\G$ in this paper will be of the form $\pi/\pi^{(n)}_r$ where $\pi$ is the fundamental group of a space (usually a $4$-manifold) and $\pi^{(n)}_r$ is the $n^{th}$-term of the \textbf{rational derived series}. The latter was first considered systematically by Harvey. It is defined by
$$
\pi^{(0)}_r\equiv \pi,~~~ \pi^{(n+1)}_r\equiv \{x\in \pi^{(n)}_r|\exists k\neq 0, x^k\in [\pi_r^{(n)},\pi_r^{(n)}]\}.
$$
Note that $n^{th}$-term of the usual derived series $\pi^{(n)}$ is contained in the $n^{th}$-term of the rational derived series. For free groups and knot groups, they coincide. It was shown in ~\cite[Section 3]{Ha1} that $\pi/\pi^{(n)}_r$ is a PTFA group.

The utility of the von Neumann signatures lies in the fact that they obstruct knots from being slice knots. It was shown in ~\cite[Theorem 4.2]{COT} that, under certain situations, higher-order von Neumann signatures vanish for slice knots, generalizing the classical result of Murasugi and the results of Casson-Gordon. Here we state their result for slice knots.

First,

\begin{thm}[Cochran-Orr-Teichner~{\cite[Theorem 4.2]{COT}}] \label{thm:oldsliceobstr}If a knot $K$ is topologically slice in a rational homology $4$-ball, $M_K$ is the zero surgery on $K$ and
$\phi:\pi_1(M_K)\to \G$ is a PTFA coefficient system that extends to the fundamental group of the exterior of the slicing disk, then $\rho(M_K,\phi)=0$.
\end{thm}

Moreover, Cochran-Orr-Teichner showed that this same result holds if $\G^{(n+1)}=\{1\}$ and $K\in \mathcal{F}_{(n.5)}$. This filtration will be defined in Section \ref{sec:bordism} where we also greatly generalize their signature theorem.

Some other useful properties of von Neumann
$\rho$-invariants are given below. One can find
detailed explanations of most of these in \cite[Section 5]{COT}.


\begin{prop}
\label{prop:rho invariants}Let $M$ be a closed, oriented $3$-manifold and $\phi : \pi_1(M) \to \G$ be a PTFA coefficient system.
\begin{itemize}
\item [(1)] If $(M,\phi) = \partial (W,\psi)$ for some compact
oriented 4-manifold $W$ such that the equivariant intersection form on $H_2(W;\mathcal{K}\G)/j_*(H_2(\partial W;\mathcal{K}\G))$ admits a half-rank summand on
which the form vanishes, then
$\s^{(2)}_\G(W,\psi)=0$ (see ~\cite[Lemma 3.1 and Remark 3.2]{Ha2} for a proper explanation of this for manifolds with $\beta_1>1$). Thus if  $\s(W) = 0$ then $\rho(M,\phi) = 0$.

\item [(2)] If $\phi$ factors through $\phi' : \pi_1(M) \to \G'$ where
$\G'$ is a subgroup of $\G$, then $\rho(M,\phi') = \rho(M,\phi)$.

\item [(3)] If $\phi$ is trivial (the zero map), then $\rho(M,\phi) = 0$.

\item [(4)] If $M=M_K$ is zero surgery on a knot $K$ and $\phi:\pi_1(M)\to \mathbb{Z}$ is the abelianization, then $\rho(M,\phi)$ is equal to the integral over the circle of the Levine (classical) signature function of $K$, normalized so that the length of the circle is $1$ ~\cite[Prop. 5.1]{COT2}. \textbf{This real number will be denoted $\rho_0(K)$}.

\end{itemize}
\end{prop}

We will establish an elementary lemma that reveals the additivity of the $\rho$-invariant under infection. It is slightly more general than \cite[Proposition 3.2]{COT2}. The use of a Mayer-Vietoris sequence to analyze the effect of a satellite construction on signature defects is common to essentially all of the previous work in this field (see for example ~\cite{Lith1}).

Suppose $L=R(\eta_i,K_i)$ is obtained by infection as described in Section~\ref{sec:Introduction}. Let the zero surgeries on $R$, $L$, and $K_i$ be denoted $M_R$, $M_L$, $M_i$ respectively. Suppose $\phi:\pi_1(M_L)\to \G$ is a map to an arbitrary PTFA group $\G$ such that, for each $i$, $\ell_i$, the longitude of $K_i$, lies in the kernel of $\phi$. Since $S^3-K_i$ is a submanifold of $M_L$, $\phi$ induces  a map on $\pi_1(S^3-K_i)$. Since $l_i$ lies in the kernel of $\phi$ this map extends uniquely to a map that we call $\phi_i$  on $\pi_1(M_i)$. Similarly, $\phi$ induces a map on $\pi_1(M_R-\coprod \eta_i)$. Since $M_R$ is obtained from $(M_R-\coprod \eta_i)$ by adding $m$ $2$-cells along the meridians of the $\eta_i$, $\mu_{\eta_i}$ and $m$ $3-$cells, and since $\mu_{\eta_i}=l_i^{-1}$ and $\phi_i(l_i)=1$, $\phi$ extends uniquely to $\phi_R$. Thus $\phi$ induces unique maps $\phi_i$ and $\phi_R$ on $\pi_1(M_i)$ and $\pi_1(M_R)$ (characterized by the fact that they agree with $\phi$ on $\pi_1(S^3-K_i)$ and $\pi_1(M_R-\coprod \eta_i)$ respectively).

There is a very important case when the hypothesis above that $\phi(\ell_i)=1$ is always satisfied. Namely suppose $\G^{(n+1)}=1$ and $\eta_i\in \pi_1(M_R)^{(n)}$.  Since a longitudinal push-off of $\eta_i$, called $\ell_{\eta_i}$ or $\eta_i^+$, is isotopic to $\eta_i$ in the solid torus $\eta_i\times D^2\subset M_R$, $\ell_{\eta_i}\in \pi_1(M_R)^{(n)}$ as well. By ~\cite[Theorem 8.1]{C} or ~\cite{Lei3} it follows that $\ell_{\eta_i}\in \pi_1(M_L)^{(n)}$. Since $\mu_i$, the meridian of $K_i$, is identified to $\ell_{\eta_i}$, $\mu_i \in \pi_1(M_L)^{(n)}$ so $\phi(\mu_i)\in \G^{(n)}$  for each $i$.  Thus $\phi_i(\pi_1(S^3- K_i)^{(1)})\subset\G^{(n+1)}=1$ and in particular the longitude of each $K_i$ lies in the kernel of $\phi$.

\begin{lem}\label{lem:additivity} In the notation of the two previous paragraphs (assuming $\phi(\ell_i)=1$ for all $i$),
$$
\rho(M_L,\phi) - \rho(M_R,\phi_R) = \sum^m_{i=1}\rho(M_i,\phi_i).
$$
Moreover if $\pi_1(S^3-K_i)^{(1)}\subset$ kernel($\phi_i$) then either $\rho(M_i,\phi_i)=\rho_0(K_i)$, or $\rho(M_i,\phi_i)=0$, according as $\phi_R(\eta_i)\neq 1$ or $\phi_R(\eta_i)= 1$. Specifically, if $\G^{(n+1)}=1$ and $\eta_i\in \pi_1(M_R)^{(n)}$ then this is the case.
\end{lem}

\begin{proof} Let $E$ be the $4$-manifold obtained from $M_R\times [0,1] \coprod -M_i\times [0,1]$ by identifying, for each $i$, the copy of $\eta_i\times D^2$ in $M_R\times \{1\}$ with the tubular neighborhood of $K_i$ in $M_i\times \{0\}$ as in Figure~\ref{fig:mickey}.
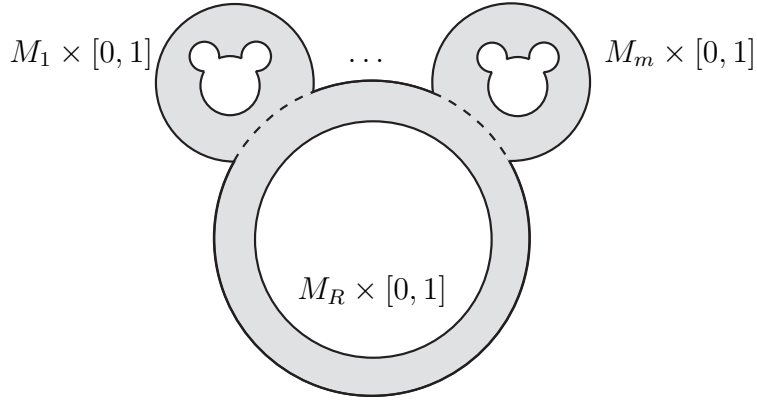
\begin{figure}[htbp]
\setlength{\unitlength}{1pt}
\begin{picture}(150,150)
\put(0,0){\includegraphics[height=150pt]{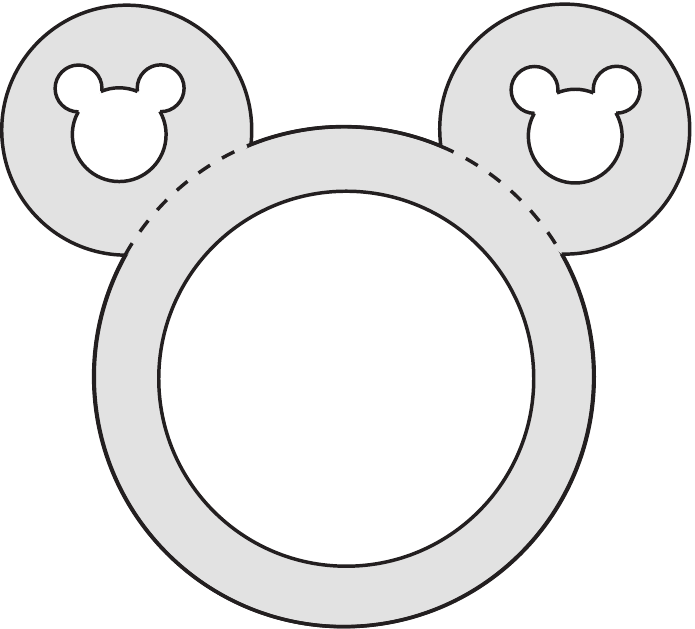}}
\put(54,37){$M_R\times [0,1]$}
\put(-55,127){$M_1\times [0,1]$}
\put(170,127){$M_m\times [0,1]$}
\put(73,127){$\dots$}
\end{picture}
\caption{The cobordism $E$}\label{fig:mickey}
\end{figure}
The dashed arcs in the figure represent the solid tori $\eta_i\times D^2$. Observe that the `outer' boundary component of $E$ is $M_L$.
Note that $E$ deformation retracts to $\overline{E}= M_L\cup (\coprod_i(\eta_i\times D^2))$, where each solid torus is attached to $M_L$ along its boundary. Hence $\overline{E}$ is obtained from $M_L$ by adding $m$ $2$-cells along the loops $\mu_{\eta_i}=l_i$, and $m$ $3$-cells. Thus, by our assumption, $\phi$ extends uniquely to $\overline{\phi}:\pi_1(\overline{E})\to \G$ and hence $\ov\phi:\pi_1(E)\to\G$. Clearly the restrictions of $\overline{\phi}$ to $\pi_1(M_i)$ and $\pi_1(M_R\times \{0\})$ agree with $\phi_i$ and $\phi_R$ respectively. It follows that that
$$
\rho(M_L,\phi) - \rho(M_R,\phi_R) = \sum^m_{i=1}\rho(M_i,\phi_i) + \sigma^{(2)}_\G(E,\overline{\phi})- \sigma(E).
$$
Now we claim that both the ordinary signature of $E$, $\s(E)$, as well as the $L^2$-signature $\s^{(2)}_\G(E)$, vanish. The first part of the proposition will follow immediately.

\begin{lem}\label{lem:mickeysig} With respect to \emph{any} coefficient system, $\phi:\pi_1(E)\to \Gamma$,  the signature of the equivariant intersection form on $H_2(E;\mathbb{Z}\G)$ is zero.
\end{lem}
\begin{proof}[Proof of Lemma~\ref{lem:mickeysig}] We show that all of the (twisted) second homology of $E$ comes from its boundary. This immediately implies the claimed result.

Consider the Mayer-Vietoris sequence with coefficients twisted by $\phi$:
$$
H_2(M_R\x~I)\oplus_i H_2(M_{i}\x~I)\ra H_2(E)\ra H_1(\amalg\eta_i\x D^2)\ra H_1(M_R\x I)\oplus_i H_1(M_i\x I).
$$
We claim that each of the inclusion-induced maps
$$
H_1(\eta_i\x D^2)\ra H_1(M_i)
$$
is injective. If $\phi(\eta_i)=1$ then, since $\eta_i$ is equated to the meridian of $K_i$, $\phi(\mu_{K_i})=1$. Since $\mu_{K_i}$ normally generates $\pi_1(M_i)$, it follows that the coefficient systems on $\eta_i\x D^2$ and $M_i$ are trivial and hence the injectivity follows from the injectivity with $\mathbb{Z}$-coefficients, which is obvious since $\mu_{K_i}$ generates $H_1(M_i)$. Suppose now that $\phi(\eta_i)\neq 1$. Since $\eta_i\x D^2$ is homotopy equivalent to a circle, it suffices to consider the cell structure on $S^1$ with one $1$-cell. Then the boundary map in the $\mathbb{Z}[\pi_1(S^1)]$ cellular chain complex for $S^1$ is multiplication by $t-1$ so the boundary map in the equivariant chain complex
$$
C_1\otimes \mathbb{Z}\G \overset{\partial\otimes id}\longrightarrow C_0\otimes \mathbb{Z}\G
$$
is easily seen to be left multiplication by $\phi(\eta_i)-1$. Since $\phi(\eta_i)\neq 1$ and $\mathbb{Z}\G$ is a domain, this map is injective. Thus $H_1(\eta_i\x D^2;\mathbb{Z}\G)=~0$ so injectivity holds.

Now using the Mayer-Vietoris sequence, any element of $H_2(E)$ comes from $H_2(M_R\x\{0\})\oplus_i H_2(M_i\x\{0\})$, in particular from $H_2(\p E)$. Thus the intersection form on $H_2(E)$ is identically zero and any signature vanishes.
\end{proof}

This completes the proof of the first part of Lemma \ref{lem:additivity}.

If $\pi_1(S^3-K_i)^{(1)}\subset$ kernel($\phi_i$) then $\phi_i$ factors through the abelianization of $H_1(S^3\backslash K_i)$ and so by parts $2,3$ and $4$ of Proposition~\ref{prop:rho invariants}, we are done. In particular if $\G^{(n+1)}=1$ and $\eta_i \in \pi_1(M_R)^{(n)}$, then $\phi_i(\mu_i)\in\G^{(n)}$ for each $i$ as we have shown in the paragraph above the lemma, so $\phi_i(\pi_1(S^3\backslash K_i)^{(1)})\subset\G^{(n+1)}=1$. Thus each $\phi_i$ factors through the abelianization.
\end{proof}

We want to collect, in the form of a lemma,  the technical properties of the cobordism $E$ that we have established in the proofs above. These will be used often in later sections.

\begin{lem}\label{lem:mickeyfacts} With regard to $E$ as above, the inclusion maps induce
\begin{itemize}
\item [(1)] an epimorphism $\pi_1(M_L)\to \pi_1(E)$ whose kernel is the normal closure of the longitudes of the infecting knots $K_i$ viewed as curves $\ell_i\subset S^3-K_i\subset M_L$;
\item [(2)] isomorphisms $H_1(M_L)\to H_1(E)$ and $H_1(M_R)\to H_1(E)$;
\item [(3)] and isomorphisms $H_2(E)\cong H_2(M_L)\oplus_i H_2(M_{K_i})\cong H_2(M_R)\oplus_i H_2(M_{K_i})$.
\item [(4)] The longitudinal push-off of $\eta_i$, $\ell_{\eta_i}\subset M_L$ is isotopic in $E$ to $\eta_i\subset M_R$ and to the meridian of $K_i$, $\mu_i\subset M_{K_i}$.
\item [(5)] The longitude of $K_i$, $\ell_i\subset M_{K_i}$ is isotopic in $E$ to the reverse of the meridian of $\eta_i$, $(\mu_{\eta_i})^{-1}\subset M_L$ and to the longitude of $K_i$ in $S^3-K_i\subset M_L$ and to the reverse of the meridian of $\eta_i$, $(\mu_{\eta_i})^{-1}\subset M_R$ (the latter bounds a disk in $M_R$).
\end{itemize}
\end{lem}
\begin{proof} We saw above that $E\sim\overline{E}$ is obtained from $M_L$ by adding $m$ $2$-cells along the loops $\mu_{\eta_i}=\ell_i$, and then adding $m$ $3$-cells that go algebraically zero over these $2$-cells. Property $(1)$ and the first part of properties $(2)$ and $(3)$ follow. The second parts of properties $(2)$ and $(3)$ follow from a Mayer-Vietoris argument as in the proof just above.
 Properties $(4)$ and $(5)$ are obvious from the definitions of infection and of $E$. \end{proof}

\section{First-order $L^{(2)}$-signatures}\label{sec:firstordersigs}

For a knot $K$ the $\rho$-invariant, $\rho_0(K)$, associated to the abelianization of $\pi_1(M_K)$, has played a central role in knot concordance since it is the average of classical signatures. Call this a \textbf{zero-order signature}. In this section we define first-order signatures for a knot $K$ and make some elementary observations. These signatures are essentially the $L^{(2)}$ analogues of Casson-Gordon invariants, though not necessarily associated to characters corresponding to metabolizers. They will play a small but central role in our proofs.

Suppose $K$ is a knot in $S^3$, $G=\pi_1(M_K)$ and $\mathcal{A}_0=\mathcal{A}_0(K)$ is its classical rational Alexander module. Note that since the longitudes of $K$ lie in $\pi_1(S^3-K)^{(2)}$,
$$
\mathcal{A}_0\equiv G^{(1)}/G^{(2)}\otimes_{\mathbb{Z}[t,t^{-1}]}\mathbb{Q}[t,t^{-1}]
$$
Each submodule $P\subset \mathcal{A}_0$ corresponds to a unique metabelian quotient of $G$,
$$
\phi_P:G\to G/\tilde{P},
$$
by setting
$$
\tilde{P}\equiv \{x~| x\in \text{kernel}(G^{(1)}\to G^{(1)}/G^{(2)}\to \mathcal{A}_0/P)\}.
$$
\noindent Note that $G^{(2)}\subset \tilde{P}$ so $G/\tilde{P}$ is metabelian. In summary, to any such submodule $P$ there corresponds a real number, the Cheeger-Gromov invariant, $\rho(M_K, \phi_P:G\to G/\tilde{P})$.

\begin{defn}\label{defn:highordersignatures} The \textbf{first-order $\mathbf{L^{(2)}}$-signatures} of a knot $K$ are the real numbers
$\rho(M_K, \phi_P)$ where $P\subset \mathcal{A}_0(K)$ satisfies $P\subset P^\perp$ with respect to the classical Blanchfield form $\mathcal{B}\ell_0$ on $K$ (i.e. $\mathcal{B}\ell_0(p,p')=0$ for all $p,p'\in P$). Slightly more generally, in light of property $2$ of Proposition~\ref{prop:rho invariants}, we say that $\rho(M_K,\phi)$ for $\phi:\pi_1(M_K)\to\G$ is a \textbf{first-order signature of} $K$ if
\begin{itemize}
\item [1.] $\phi$ factors through $G/G^{(2)}$, where $G=\pi_1(M_K)$;
\item [2.]$\text{kernel}(\phi)= \text{kernel}(G^{(1)}\to G^{(1)}/G^{(2)}\to \mathcal{A}_0/P)$
for some submodule $P\subset \mathcal{A}_0(K)$ such that $P\subset P^\perp$ with respect to the classical Blanchfield form on $K$.
\end{itemize}
\end{defn}
The first-order signatures that correspond to metabolizers, that is submodules $P$ for which $P=P^\perp$, have been previously studied and are closely related to Casson-Gordon-Gilmer invariants ~\cite{Let}~\cite{Fr2}~\cite{Fr3}~\cite{Ki1}. Since $P=0$ always satisfies $P\subset P^\perp$, we give a special name to the signature corresponding to this case.

\begin{defn}\label{defn:rho1} $\mathbf{\rho^1(K)}$ of a knot $K$ is the first-order $L^{(2)}$-signature given by the Cheeger-Gromov invariant $\rho(M_K, \phi:G\to G/G^{(2)})$.
\end{defn}

Similar to Casson-Gordon invariants, if $K$ is topologically slice in a rational homology 4-ball, then \emph{one} of the first-order signatures of $K$ must be zero. We will prove a more general statement in Proposition~\ref{prop:firstordervanish}. However none of the first-order signatures is itself a concordance invariant. In particular there exist ribbon knots with $\rho^1\neq 0$ as we shall see below.

A genus one algebraically slice knot has precisely two metabolizers, $P_1$, $P_2$ for the Seifert form and so has precisely $3$ first-order signatures, two corresponding to $P_1$ and $P_2$ and the third corresponding to $P_3=0$.

\begin{ex}\label{ex:first-ordersigs} Consider the knot $K$ in Figure~\ref{fig:examplehighersigs}. This knot is obtained from a ribbon knot $R$ by two infections on the band meridians $\alpha, \beta$ (as in the left-hand side of Figure~\ref{fig:Rdoubling}). Thus $\{\alpha, \beta \}$ is a basis of $\mathcal{A}_0(K)= \mathcal{A}_0(R)$.
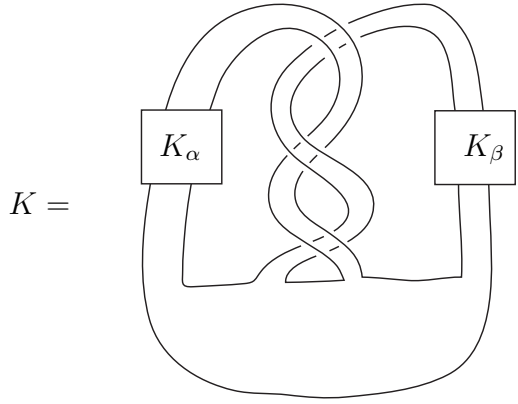
\begin{figure}[htbp]
\setlength{\unitlength}{1pt}
\begin{picture}(200,160)
\put(0,0){\includegraphics[height=150pt]{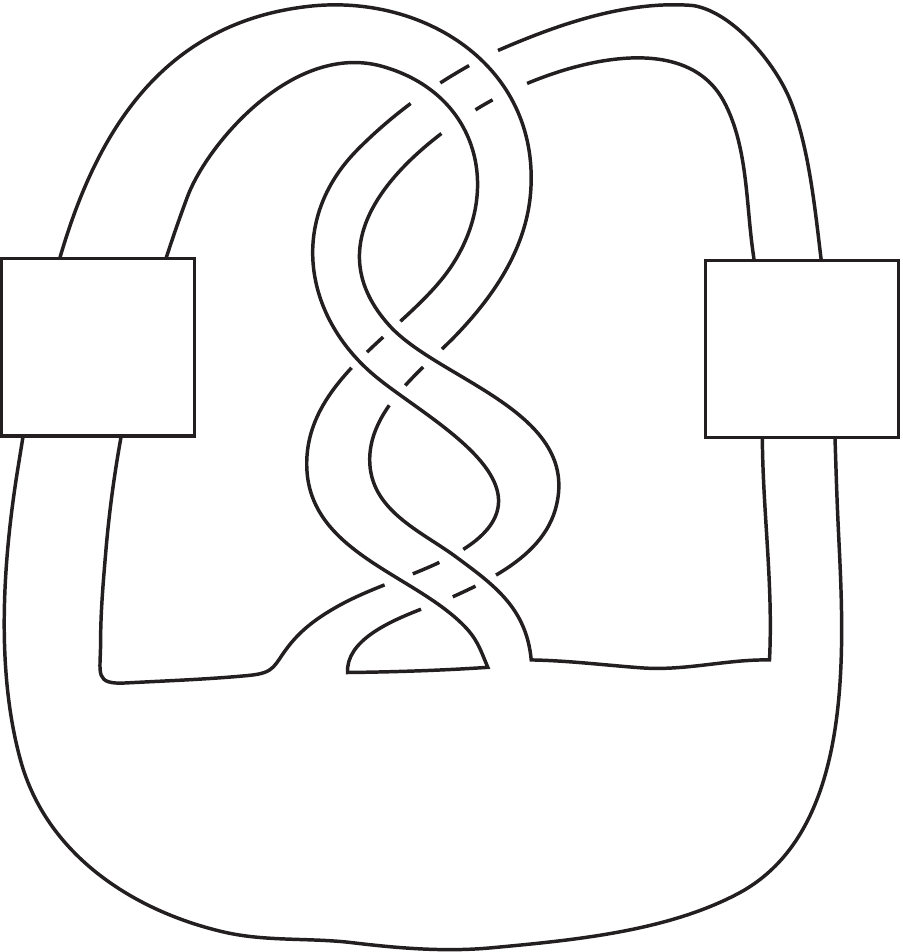}}
\put(-50,70){$K=$}
\put(7,92){$K_\alpha$}
\put(122,92){$K_\beta$}
\end{picture}
\caption{A genus $1$ algebraically slice knot $K$}\label{fig:examplehighersigs}
\end{figure}
There are $3$ submodules $P$ for which $P\subset P^\perp$, namely $P_0=0$, $P_\alpha=\left<\alpha\right>$ and $P_\beta=\left<\beta\right>$. We may apply Lemma~\ref{lem:additivity} to show
$$
\rho(M_K, \phi_P)=\rho(M_{R}, \phi_P)+\epsilon^\alpha _P\rho_0(K_\alpha)+\epsilon^\beta _P\rho_0(K_\beta)
$$
where $\epsilon^\alpha_P$ is $0$ or $1$ according as $\phi_P(\alpha)=1$ or not (similarly for $\epsilon^\beta_P$). For our example $\phi_{P_\alpha}(\alpha)=1$ and $\phi_{P_\alpha}(\beta)\neq 1$. Similarly $\phi_{P_\beta}(\beta)=1$ and $\phi_{P_\beta}(\alpha)\neq 1$. By contrast $\phi_{P_0}(\alpha)\neq 1$ and $\phi_{P_0}(\beta)\neq 1$. Moreover $P_\alpha$ corresponds to the kernel $\tilde{P}_\alpha$, of $\pi_1(S^3-R)\to \pi_1(B^4-\Delta_\alpha)/\pi_1(B^4-\Delta_\alpha)^{(2)}$ for the ribbon disk $\Delta_\alpha$ for $R$ obtained by ``cutting the $\alpha$-band''. (Similarly for $P_\beta$.) Thus in both cases the maps $\phi_P$ on $M_{R_1}$ extend over ribbon disk exteriors. Consequently $\rho(M_{R}, \phi_P)=0$ for $P=P_\alpha$ and $P=P_\beta$, by Theorem \ref{thm:oldsliceobstr}. Of course $\rho(M_R, \phi_{P_0})=\rho^1(R)$ by definition. Putting this all together we see that the first-order signatures of the knot $K$ are $\{\rho_0(K_\alpha),\rho_0(K_\beta),\rho^1(R)+\rho_0(K_\alpha)+\rho_0(K_\beta)\}$. Note that if we choose $K_\alpha$ to be the unknot and choose $K_\beta$ so that $\rho_0(K_\beta)\neq -\rho^1(R)$ then $K$ is a ribbon knot with $\rho^1\neq 0$.
\end{ex}

We remark that $\rho^1$ vanishes for a $(\pm)$-amphichiral knot by Proposition~\ref{prop:amphi} but it is not true that all the first-order signatures vanish for an amphichiral knot.

\begin{prop}\label{prop:amphi} If a $3$-manifold $M$ admits an orientation-reversing homeomorphism, then $\rho(M,\phi)=0$ for any $\phi$ whose kernel is a characteristic subgroup of $\pi_1(M)$.
\end{prop}
\begin{proof}[Proof of Proposition~\ref{prop:amphi}] Suppose $h:-M\to M$ is an orientation preserving homeomorphism. Then for any $\phi$,
$$
\rho(M,\phi)=\rho(-M, \phi\circ h_*)=-\rho(M,\phi\circ h_*).
$$
Since the $\rho$ invariant depends only on the kernel of $\phi$, which, being characteristic, is the same as the kernel of $\phi\circ h_*$, the last term equals $-\rho(M,\phi)$. Since the $\rho$ invariant is real-valued, it is zero.
\end{proof}

A genus one knot that is \emph{not} zero in the rational algebraic concordance group (that is there is no metabolizer for the rational Blanchfield form) has precisely one first-order signature, namely $\rho^1(K)$ since any \emph{proper} submodule $P$ of the rational Alexander module satisfying $P\subset P^\perp$ would have to be a (rational) metabolizer.

\begin{ex}
The knot $K$ in Figure~\ref{fig:examplehighersigsfigeight} is of order two in the rational algebraic concordance group and therefore $\rho^1$ is the only first-order signature. Using Lemma~\ref{lem:additivity}, we see that $\rho^1(K)=\rho^1(\text{figure-eight})+2\rho_0(K')=2\rho_0(K')$. Therefore if $K'$ is chosen so that $\rho_0(K')\neq 0$ then $K$ is not slice in a rational homology ball.
\begin{figure}[htbp]
\begin{picture}(185,118)
\put(0,0){\includegraphics{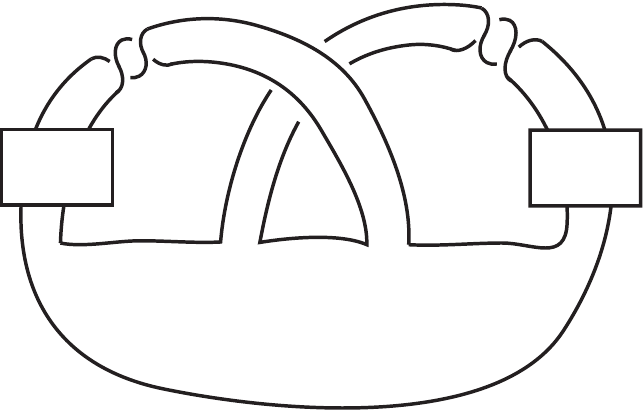}}
\put(10,67){$K'$}
\put(163,67){$K'$}
\end{picture}
\caption{} \label{fig:examplehighersigsfigeight}
\end{figure}
\end{ex}


The definition of the first-order signatures is not quite the same as that implicit in the work of Casson-Gordon-Gilmer and in more generality in ~\cite[Theorem 4.6]{COT}. One would hope that one need only consider those $P$ such that $P=P^\perp$. However this is false in the context of rational concordance. The knots in Figure~\ref{fig:examplehighersigsfigeight} are in general not slice in a rational homology ball, but this fact is \emph{not} detected by signatures associated to metabolizers of the classical rational Blanchfield form. But this \emph{is} detected by $\rho^1$. Note that the figure-eight knot is slice in a rational homology $4$-ball in such a way that the Alexander module of the figure-eight knot \emph{injects} into $\pi/\pi^{(2)}_r$ where $\pi$ is the fundamental group of the complement of the slicing disk!

\section{$J_2(K)$}\label{J2}

Recall that, for $J_2(K)$, as in Figure~\ref{fig:family}, all classical invariants as well as those of Casson-Gordon vanish. In this section, as a warm-up for more general results, we prove that higher-order signatures yield further obstructions to $J_2(K)$ being a slice knot. We set $J_{0}=K$ and let $J_{2}=J_{2}(K)=R\circ R(K)$ where $R$ is the $9_{46}$ knot.

\begin{thm}\label{thm:J2notslice} If $J_2(K)$ is a slice knot then $\rho_0(K)\in \{0,-\frac{1}{2}\rho^1(R)\}$ where $\rho^1(R)$ is the real number from Definition \ref{defn:rho1}.
\end{thm}
We obtain this as a corollary of the following more general result. Consider a knot $J$ as shown in Figure \ref{fig:genJ2}. Note that if $\bar{J}$ is algebraically slice then $J$ has vanishing Casson-Gordon invariants, so is indistinguishable from a slice knot by all previously known techniques.

\begin{thm}\label{thm:generalJ2notslice} If the knot $J$ of Figure~\ref{fig:genJ2} is a slice knot (or even $(2.5)$-solvable) then one of the first-order signatures of $\bar{J}$ vanishes.
\end{thm}
\begin{figure}[htbp]
\setlength{\unitlength}{1pt}
\begin{picture}(200,160)
\put(0,0){\includegraphics[height=150pt]{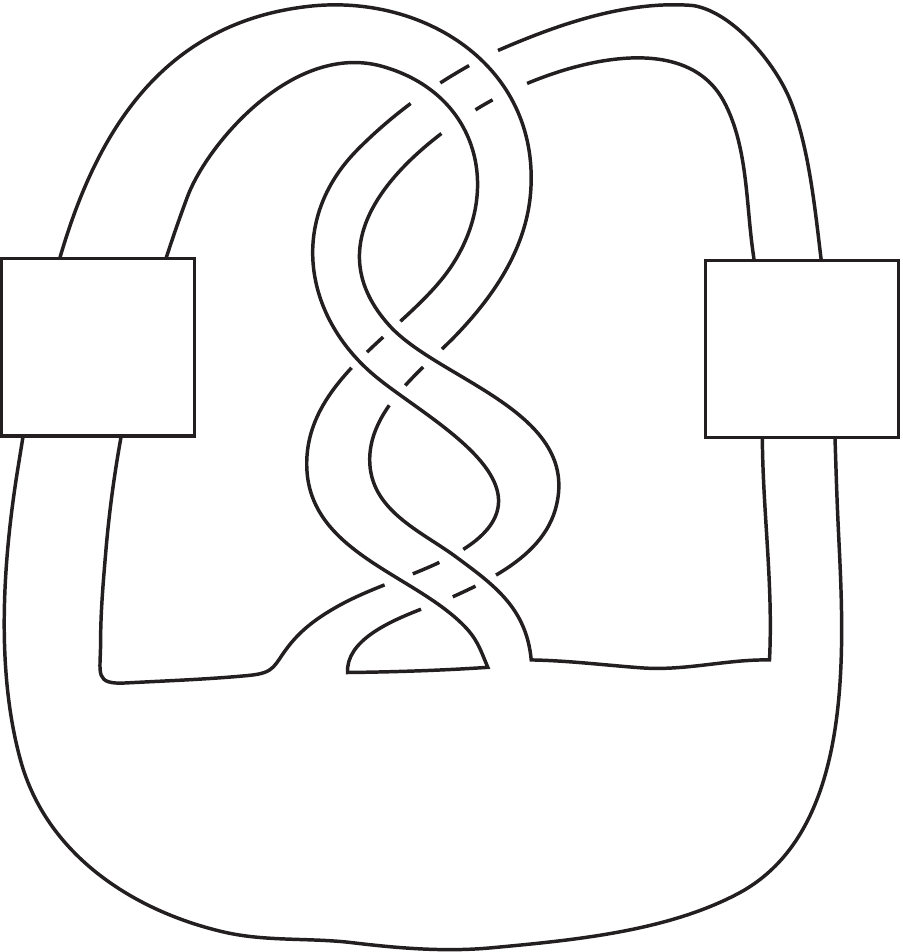}}
\put(-50,70){$J~=$}
\put(7,92){$\bar{J}$}
\put(122,92){$\bar{J}$}
\end{picture}
\caption{}\label{fig:genJ2}
\end{figure}

\begin{proof}[Proof of Theorem~\ref{thm:J2notslice}] Suppose $J_2(K)$ is slice. Since $J_1(K)$ is algebraically slice, we can apply Theorem~\ref{thm:generalJ2notslice} with $J=J_2(K)$ and $\bar{J}=J_1(K)$ to conclude that one of the first order signatures of $J_1(K)$ vanishes. In Example~\ref{ex:first-ordersigs} we saw that these signatures are $\{\rho_0(K),\rho_0(K), \rho^1(R)+2\rho_0(K)\}$. Thus either $\rho_0(K)=0$ or $\rho_0(K)=-\frac{1}{2}\rho^1(R)$.
\end{proof}

\begin{proof}[Proof of Theorem~\ref{thm:generalJ2notslice}] Suppose $J$ is slice and let $V$ denote the exterior of a slice disk. Thus $\partial V=M_{J}$. Let $\pi=\pi_1(V)$. Then $H_1(M_J)\cong H_1(V)\cong \pi/\pi^{(1)}\cong \mathbb{Z}$. Consider the coefficient system $\pi\to \pi/\pi^{(1)}\cong \mathbb{Z}$ and the inclusion-induced map:
\begin{equation}\label{eq:6.1}
j_*: H_1(M_J;\mathbb{Q}[t,t^{-1}])\to H_1(V;\mathbb{Q}[t,t^{-1}])
\end{equation}
which is merely the map on the classical rational Alexander modules. If $V$ is a slice disk exterior it is well known that the kernel $P_0$ of $j_*$ is self-annihilating with respect to the classical Blanchfield form on $J$, i.e. $P_0=P_0^\perp$. The knot $J$ has the same Blanchfield form as the ribbon knot $R=9_{46}$ so $P_0$ is either the submodule generated by $\alpha$ or the submodule generated by $\beta$. Because $J$ is symmetric, without loss of generality we assume $P_0=\<\alpha\>$. Furthermore, it is them known, by work of D. Cooper (unpublished) and ~\cite[Theorem 5.2]{COT2} that the zero$^{th}$ order signature of $\bar{J}$ vanishes, i.e. $\rho_0(\bar{J})=0$. By definition
$$
\pi^{(1)}/\pi_r^{(2)}=(\pi^{(1)}/[\pi^{(1)},\pi^{(1)}])/(\mathbb{Z}-\text{torsion}).
$$
(Note that since $\pi/\pi^{(1)}\cong \mathbb{Z}$, $\pi^{(1)}_r = \pi^{(1)}$.)
Thus there is a monomorphism
$$
i:\pi^{(1)}/\pi_r^{(2)}\hookrightarrow \pi^{(1)}/[\pi^{(1)},\pi^{(1)}])\otimes_\mathbb{Z}\mathbb{Q}.
$$
The latter has a strictly homological interpretation as the first homology with $\mathbb{Q}$ coefficients of the covering space of $V$ whose fundamental group is $\pi^{(1)}$. In other words
$$
\pi^{(1)}/[\pi^{(1)},\pi^{(1)}])\otimes_\mathbb{Z}\mathbb{Q}\cong H_1(V;\mathbb{Q}[\pi/\pi^{(1)}]).
$$
Therefore we have the following commutative diagram where $i$ is injective.
\begin{equation*}
\begin{CD}
\pi_1(M_J)^{(1)}      @>\equiv>>    \pi_1(M_J)^{(1)}  @>j_*>>   \pi^{(1)}  @>>>
\pi^{(1)}/\pi^{(2)}_r \\
  @VVV   @VVV        @VVV       @VViV\\
\mathcal{A}_0(J)     @>\cong>>  H_1(M_J;\mathbb{Q}[t,t^{-1}])    @>j_*>> H_1(V;\mathbb{Q}[t,t^{-1}]) @>\cong>>
  (\pi^{(1)}_r/[\pi^{(1)}_r,\pi^{(1)}_r])\otimes_\mathbb{Z} \mathbb{Q}\\
\end{CD}
\end{equation*}
Since the kernel of the bottom horizontal composition is $\<\alpha\>$, the kernel of the top horizontal composition is $\<\alpha\>$, and therefore it follows that
\begin{equation}\label{eq:6.2}
j_*(\alpha)\in \pi^{(2)}_r ~~\text{and} ~~j_*(\beta)\neq 1\in \pi^{(1)}/\pi^{(2)}_r.
\end{equation}

\noindent Since $J$ is obtained from $R$ by two infections along $\alpha$ and $\beta$, from Lemma~\ref{lem:mickeyfacts} there is a corresponding cobordism $E$ with $4$ boundary components $M_J$, $M_{R}$ and two copies of $M_{\bar{J}}$. $R$ has an obvious ribbon disk that corresponds to ``cutting the $\alpha$ band''. Let $\mathcal{R}$ denote the exterior in $B^4$ of this ribbon disk. Then $\partial \mathcal{R}=M_{R}$.
\begin{rem}\label{rem:Rfacts} The salient features of $\mathcal{R}$ are
\begin{itemize}
\item [1.] $\pi_1(M_{R})\to \pi_1(\mathcal{R})$ is an epimorphism whose kernel is generated by the normal closure of $\alpha$;
\item [2.] $H_1(M_{R})\to H_1(\mathcal{R})$ is an isomorphism;
\item [3.] $H_2(\mathcal{R})=0$.
\end{itemize}
\end{rem}
Construct a $4$-manifold called $W$ by first identifying $-E$ with $V$ along $M_{J}$ and then capping off the boundary component $M_{R}$ using $-\mathcal{R}$, as shown in Figure~\ref{fig:WforJ2}.
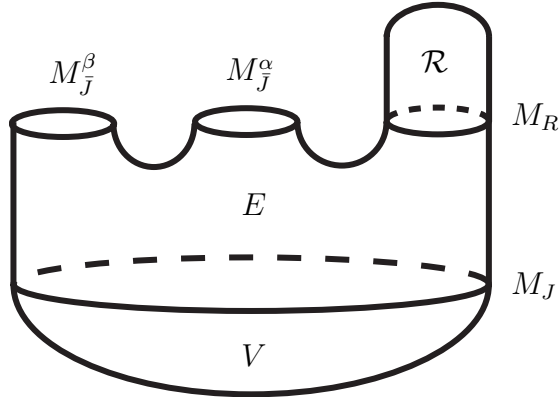
\begin{figure}[htbp]
\setlength{\unitlength}{1pt}
\begin{picture}(200,160)
\put(0,0){\includegraphics[height=150pt]{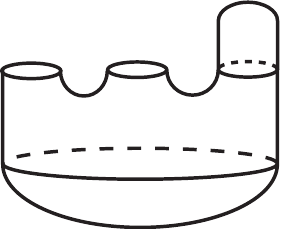}}
\put(82,120){$M^{\alpha}_{\bar{J}}$}
\put(15,119){$M^{\beta}_{\bar{J}}$}
\put(156,123){$\mathcal{R}$}
\put(190,103){$M_{R}$}
\put(190,39){$M_J$}
\put(88,12){$V$}
\put(88,69){$E$}
\end{picture}
\caption{The cobordism $W$}\label{fig:WforJ2}
\end{figure}
The boundary components of $W$ will be called $M^\alpha_{\bar{J}}$ and $M^\beta_{\bar{J}}$. Let $\tilde{\pi}=\pi_1(W)$ and $\G=\tilde{\pi}/\tilde{\pi}^{(3)}_r$. Denote the projection $\tilde{\pi}\to \G$ by $\phi$ and its restriction to the boundary components by $\phi_\alpha$ and $\phi_\beta$. We have
$$
\rho(M^{\alpha}_{\bar{J}},\phi_\alpha)+\rho(M^\beta_{\bar{J}},\phi_\beta)=\sigma^{(2)}_\G(W)-\sigma(W).
$$
By the additivity of both types of signature, the signature defect on the right-hand side is a sum of the signature defects of $V$, $\mathcal{R}$, and $E$. But these vanish; the first two by Theorem~\ref{thm:oldsliceobstr} and the last by Lemma~\ref{lem:mickeysig}. Thus
\begin{equation}\label{con1}
\rho(M^{\alpha}_{\bar{J}},\phi_\alpha)+\rho(M^\beta_{\bar{J}},\phi_\beta)=0.
\end{equation}

It follows from (\ref{eq:6.2}) that $j_*(\alpha)\in \tilde{\pi}^{(2)}_r.$ Note that $\pi_1(M^\alpha_{\bar{J}})$ is normally generated by the meridian which is isotopic in $E$ to a push-off of the curve $\alpha\subset M_J$ (see property 4 of \ref{lem:mickeyfacts}). Therefore, $j_*(\pi_1(M^\alpha_{\bar{J}}))\subset \tilde{\pi}^{(2)}_r $, and hence, $j_*(\pi_1(M^\alpha_{\bar{J}})^{(1)})\subset \tilde{\pi}^{(3)}_r$. Thus $\phi_\alpha$ factors through the abelianization of $\pi_1(M^\alpha_{\bar{J}})$. Hence by properties (2)-(4) of Proposition~\ref{prop:rho invariants}, $\rho(M^\alpha_{\bar{J}},\phi_\alpha)$ is either zero or $\rho_0(\bar{J})$. However, as remarked earlier, $\rho_0(\bar{J})=0$. Thus $\rho(M^{\alpha}_{\bar{J}},\phi_\alpha)=0$.

From (\ref{con1}), we now have $\rho(M^\beta_{\bar{J}},\phi_\beta)=0.$ This finishes the proof of Theorem~\ref{thm:generalJ2notslice} once we establish that this is indeed a first order signature of $\bar{J}$. Specifically, from Definition \ref{defn:highordersignatures} we must show that:
\begin{itemize}
\item [1.] $\phi_\beta$ factors through $G/G^{(2)}$, where $G=\pi_1(M^\beta_{\bar{J}})$;
\item [2.]$\ker{\phi_\beta}= \ker{(G^{(1)}\to G^{(1)}/G^{(2)}\to \mathcal{A}_0/P)}$
for some submodule $P$ of the Alexander module $\mathcal{A}_0(\bar{J})$ such that $P\subset P^\perp$ with respect to the classical Blanchfield form on $\bar{J}$.
\end{itemize}

Since $j_*(\beta)\in {\pi}^{(1)}$, it follows that $j_*(\beta)\in \tilde{\pi}^{(1)}$. Also $\pi_1(M^\beta_{\bar{J}})$ is normally generated by the meridian which is isotopic in $E$ to a push-off of the $\beta\subset M_J$ (see property 4 of \ref{lem:mickeyfacts}). Therefore,
\begin{equation}\label{con2}
j_*(\pi_1(M^\beta_{\bar{J}}))\subset \tilde{\pi}^{(1)}
\end{equation}
Hence, $j_*(\pi_1(M^\beta_{\bar{J}})^{(2)})\subset \tilde{\pi}^{(3)}_r$. This establishes property 1 above.

To show property 2, we begin by showing that the inclusion $V\to W$ induces an isomorphism
\begin{equation}\label{eq:6.3}
\pi/\pi^{(2)}_r\to \tilde{\pi}/\tilde{\pi}^{(2)}_r.
\end{equation}
The map $\pi_1(V)\to \pi_1(V\cup E)$ is a surjection whose kernel is the normal closure of the set of longitudes $\{\ell_\alpha,\ell_\beta\}$ of the copies of $S^3-\bar{J}\subset M_J$ (property $(1)$ of Lemma~\ref{lem:mickeyfacts}). These longitudes lie in the second derived subgroups of their respective knot groups. The groups $\pi_1(S^3-\bar{J})$ are normally generated by the meridians of the respective copy of $S^3-\bar{J}$. These meridians are identified to push-offs of the curves $\alpha$ and $\beta$ respectively in $M_{J}$ and we saw in ~(\ref{eq:6.2}) that $j_*(\alpha)\in \pi^{(2)}_r$ and $j_*(\beta)\in \pi^{(1)}$. Thus $j_*(\ell_\alpha)$ and $j_*(\ell_\beta)$ lie in $\pi^{(3)}_r$ and so the inclusion map $V\to V\cup E$ induces an isomorphism on $\pi_1$ modulo $\pi_1(-)^{(3)}_r$. Similarly the map $\pi_1(V\cup E)\to \pi_1(V\cup E\cup \mathcal{R})$  is a surjection whose kernel is the normal closure of the curve $\alpha\subset M_{R}$ (by property 1 of (\ref{rem:Rfacts})). But this curve $\alpha$ is isotopic in $E$ to a push-off of the curve $\alpha\subset M_J$ and
$j_*(\alpha)\in \pi^{(2)}_r$. Thus the inclusion $V\cup E\to W$ induces an isomorphism on $\pi_1$ modulo $\pi_1(-)^{(2)}_r$. Combining these two isomorphisms yields ~(\ref{eq:6.3}).

Combining (\ref{eq:6.2}) and (\ref{eq:6.3}), we have that $j_*(\beta) \neq 1\in \tilde{\pi}^{(1)}/\tilde{\pi}^{(2)}_r$. Since this group is torsion-free abelian, $j_*(\beta)$ generates an infinite cyclic subgroup of $\tilde{\pi}^{(1)}/\tilde{\pi}^{(2)}_r$. Therefore the inclusion of the meridian of $M^\beta_{\bar{J}}$ into $W$ is an element of infinite order in $\tilde{\pi}^{(1)}/\tilde{\pi}^{(2)}_r$. We claim that the kernel of $\phi_{\beta}:G \to \G=\tilde{\pi}/\tilde{\pi}^{(3)}_r$ is contained in $G^{(1)}$. For suppose $x \in \ker{\phi_{\beta}}$ and $x=\mu^{m}y$
 where $\mu$ is a meridian of $M^{\beta}_{\bar{J}}$ and $y \in G^{(1)}$. Since $x \in \ker{\phi_{\beta}}$, clearly $x$ is in the kernel of the composition $$\psi_{\beta}: G \overset{\phi_{\beta}}{\to} \tilde{\pi}/\tilde{\pi}^{(3)}_r \to \tilde{\pi}/\tilde{\pi}^{(2)}_r.$$ Moreover, by (\ref{con2}), $\phi_{\beta}(G^{(1)})\subset\tilde{\pi}_{r}^{(2)}$. Therefore $\mu^{m}$ is in the kernel of $\psi_{\beta}$, but this contradicts the fact that $\mu$ is an element of infinite order in $\tilde{\pi}^{(1)}/\tilde{\pi}^{(2)}_r$. Thus the kernel of $\phi_{\beta}:G \to \G$ is contained in $G^{(1)}$. It remains to describe $P$ and to show that $P\subset P^\perp$ with respect to the Blanchfield form on $\bar{J}$.

We consider the coefficient system $\psi:\tilde{\pi}\to \tilde{\pi}/\tilde{\pi}^{(2)}_r\equiv\Lambda$. By (\ref{con2}), $\psi$ restricted to $\pi_1(M^\beta_{\bar{J}})$ factors through the abelianization. Thus (as we shall discuss in more detail in Section~\ref{Blanchfieldforms})
$$
H_1(M^\beta_{\bar{J}};\mathbb{Q}\Lambda)\cong H_1(M^\beta_{\bar{J}};\mathbb{Q}[t,t^{-1}]) \otimes_{\mathbb{Q}[t,t^{-1}]}\mathbb{Q}\Lambda.
$$
Consider the composition:
\begin{equation}\label{eq:kernel}
\mathcal{A}_0(\bar{J})\overset{i}{\hookrightarrow}\mathcal{A}_0(\bar{J}) \otimes_{\mathbb{Q}[t,t^{-1}]}\mathbb{Q}\Lambda\overset{\cong}{\to}H_1(M^\beta_{\bar{J}};\mathbb{Q}\Lambda)\overset{j_*}\to H_1(W;\mathbb{Q}\Lambda),
\end{equation}
We claim that $i=id\otimes 1$ is a map of $\mathbb{Q}[t,t^{-1}]$-modules. Recall that $\mathcal{A}_0(\bar{J}) \otimes_{\mathbb{Q}[t,t^{-1}]}\mathbb{Q}\Lambda$ is a right $\mathbb{Q}\Lambda$-module and hence can be considered as a right $\mathbb{Q}[t,t^{-1}]$-module using the embedding $\mathbb{Z}\hookrightarrow \Lambda$ where, let's say, $t\to \mu$. If $x\in \mathcal{A}_0(\bar{J})$ and $p(t)\in \mathbb{Q}[t,t^{-1}]$ then:
$$
id\otimes 1(xp(t))=xp(t)\otimes 1=x\otimes (p(\mu)\cdot 1)=x\otimes (1\cdot p(\mu)))=(x\otimes 1)*p(t)=(id\otimes 1(x))*p(t).
$$
Thus $id\otimes 1$ is a right $\mathbb{Q}[t,t^{-1}]$-module map.
Define $P_1$ to be the kernel of the composition in ~\ref{eq:kernel}. We claim that $P_1$ is the submodule $P$ of the Alexander module referred to above.
Consider the following commutative diagram where $i$ is injective.
\begin{equation*}
\begin{CD}
\pi_1(M^\beta_{\bar{J}})^{(1)}      @>\equiv>>    \pi_1(M^\beta_{\bar{J}})^{(1)}  @>j_*>> \tilde{\pi}^{(2)}_r @>>>
\tilde{\pi}^{(2)}_r/\tilde{\pi}^{(3)}_r \\
  @VVV   @VVV        @VVV       @VViV\\
\mathcal{A}_0(\bar{J})     @>>>  H_1(M^\beta_{\bar{J}};\mathbb{Q}\Lambda)    @>j_*>> H_1(W;\mathbb{Q}\Lambda) @>\cong>>
  (\tilde{\pi}^{(2)}_r/[\tilde{\pi}^{(2)}_r,\tilde{\pi}^{(2)}_r])\otimes_\mathbb{Z} \mathbb{Q}\\
\end{CD}
\end{equation*}
Given an element in $P_1$, choose a representative of this element in $G^{(1)}=\pi_1(M^\beta_{\bar{J}})^{(1)}$. By the diagram above, since $i$ is injective, it follows that the representative of this element must map into $\tilde{\pi}^{(3)}_r$. In other words, the representative is in the kernel of $\phi_\beta$. This establishes that $\ker{\phi_\beta}= \ker{(G^{(1)}\to G^{(1)}/G^{(2)}\to \mathcal{A}_0/P_1)}$. Thus $P_{1}=P$.

Recall the classical fact (used above) that if $V$ is the slice disk complement for a slice knot $J$, then the kernel $P_0$ of $j_*:H_1(M_J;\mathbb{Q}[t,t^{-1}]) \to H_1(V;\mathbb{Q}[t,t^{-1}])$ is self-annihilating with respect to the classical Blanchfield form on $J$, i.e. $P_0=P_0^\perp$. We would like to extend this to show that the kernel $P_1$ satisfies $P_1\subset P_1^\perp$ with respect to the classical Blanchfield form on $\bar{J}$.

But $M^\beta_{\bar{J}}\to W$ is quite different than the classical situation. First, $M^\beta_{\bar{J}}$ is not the only boundary component of $W$. Secondly, the map $M^\beta_{\bar{J}}\to W$ is the zero map on $H_1(-)$! Thirdly, $\mathbb{Q}\Lambda$ is not a PID. Nonetheless, after defining a new category of cobordisms in Section~\ref{sec:bordism} we are able to prove the required facts using higher-order Blanchfield linking forms.

The desired result to finish the proof at hand is Theorem~\ref{thm:nontriviality} with $k=2$. The verification that $W$ is a $(2)$-bordism is straight-forward. We have not included all the details in this proof because we have not yet covered all the results necessary to prove the theorem, but thought it valuable to see the ideas and to provide motivation for the new category of cobordisms we define in the next section.

\end{proof}

\section{n-Bordisms and Rational n-Bordisms} \label{sec:bordism}

Our proof makes essential use of a much weaker notion than the $(n)$-solvability of Cochran-Orr-Teichner. In this section we define this notion and establish its key properties.

Recall that \cite[Section 8]{COT} introduced a filtration of the concordance classes of knots $\mathcal{C}$
$$
\cdots \subseteq \mathcal{F}_{n} \subseteq \cdots \subseteq
\mathcal{F}_1\subseteq \mathcal{F}_{0.5} \subseteq \mathcal{F}_{0} \subseteq \mathcal{C}.
$$
where the elements of $\mathcal{F}_{n}$ and $\mathcal{F}_{n.5}$ are called \emph{$(n)$-solvable knots} and \emph{$(n.5)$-solvable knots} respectively. This is a filtration by \emph{subgroups} of the knot concordance group. A slice knot $K$ has the property that its zero surgery $M_K$ bounds a $4$-manifold $W$ (namely the exterior of the slicing disk) such that $H_1(M_K)\to H_1(W)$ is an isomorphism and $H_2(W)=0$. An \emph{$(n)$-solvable} knot is, loosely speaking, one such that $M_K$ bounds a $4$-manifold $W$ such that $H_1(M_K)\to H_1(W)$ is an isomorphism and the intersection form on $H_2(W)$ ``looks'' hyperbolic modulo the $n^{th}$-term of the derived series of $\pi_1(W)$. The manifold $W$ is called an \emph{$(n)$-solution for $\partial W$}. These notions are defined below in the context of our new notion.

We will define a weaker notion where the condition that $H_1(\partial W)\to H_1(W)$ be an isomorphism is dropped. Somewhat surprisingly the key results still hold for this much weaker notion. Of lesser importance, we also drop the condition on the connectivity of $\partial W$, which had already been done in ~\cite{CK} in restricted cases.

For a compact connected oriented topological 4-manifold $W$, let $W^{(n)}$ denote the covering
space of $W$ corresponding to the $n$-th derived subgroup of $\pi_1(W)$. The deck translation group of this cover is the solvable group $\pi_1(W)/\pi_1(W)^{(n)}$. Then $H_2(W^{(n)};\mathbb{Q})$ can be endowed with the structure of a right $\mathbb{Q}[\pi_1(W)/\pi_1(W)^{(n)}]$-module. This agrees with the homology group with twisted coefficients $H_2(W;\mathbb{Q}[\pi_1(W)/\pi_1(W)^{(n)}])$. There is an equivariant
intersection form
$$
\lambda_n : H_2(W^{(n)};\mathbb{Q}) \times H_2(W^{(n)};\mathbb{Q}) \lra
\mathbb{Q}[\pi_1(W)/\pi_1(W)^{(n)}]
$$
 \cite[Chapter 5]{Wa}\cite[Section 7]{COT}. The usual intersection form is the case $n=0$. In general, these
intersection forms are singular. Let $I_n \equiv$ image($j_* : H_2(\partial W^{(n)};\mathbb{Q}) \to H_2(W^{(n)};\mathbb{Q})$). Then this intersection form factors
through
$$
\ov{\lambda_n} : H_2(W^{(n)};\mathbb{Q})/I_n \times H_2(W^{(n)};\mathbb{Q})/I_n \lra \mathbb{Q}[\pi_1(W)/\pi_1(W)^{(n)}].
$$
We define a  \textbf{rational $\mathbf{(n)}$-Lagrangian}, $L$,  of $W$ to be a
submodule of $H_2(W;\mathbb{Q}[\pi_1(W)/\pi_1(W)^{(n)}])$ on which
$\lambda_n$ vanishes identically and which maps onto a $\frac12$-rank
subspace of $H_2(W;\mathbb{Q})/I_0$ under the covering map. An
\textbf{$\mathbf{(n)}$-surface} is a based and immersed surface
in $W$ that can be lifted to $W^{(n)}$. Observe that any class in
$H_2(W^{(n)})$ can be represented by (the lift of) an $(n)$-surface and that
$\lambda_n$ can be calculated by counting intersection points in
$W$ among representative $(n)$-surfaces weighted appropriately by
signs and by elements of $\pi_1(W)/\pi_1(W)^{(n)}$ (see \cite[Section 7]{COT}). We say a rational
$(n)$-Lagrangian $L$ admits \textbf{rational $\mathbf{(m)}$-duals} (for $m\le n$) if $L$
is generated by (lifts of) $(n)$-surfaces $\ell_1,\ell_2,\ldots,\ell_g$ and
there exist $(m)$-surfaces $d_1,d_2,\ldots, d_g$ such that $H_2(W;\mathbb{Q})/I_0$
has rank $2g$ and $\lambda_m(\ell_i,d_j)=\delta_{i,j}$.

If $W$ is a spin manifold then we can replace all the occurrences of $\mathbb{Q}$ above by $\Z$ and consider the equivariant intersection form $\lambda_n^{\mathbb{Z}}$ on $H_2(W^{(n)};\mathbb{Z})$ as well as the equivariant \emph{self-intersection form $\mu_n$}. This leads to the definitions of an \textbf{$\mathbf{(n)}$-Lagrangian} and \textbf{$\mathbf{(m)}$-duals} for $W$, where $\mu_n$ is required to vanish identically on an $(n)$-Lagrangian and we also require that the $(n)$-Lagrangian maps onto a $\frac12$-rank \emph{summand} of $H_2(W;\mathbb{Z})/I_0$. In the presence of $(n)$-duals, this forces the usual intersection form on $H_2(W;\mathbb{Z})/I_0$ to be hyperbolic \cite[Remark 7.6]{COT}. All of the above notions were first defined in ~\cite{COT}. In ~\cite{CK} this was extended to the case that $\partial W$ is disconnected but required each boundary component $M_i$ to satisfy $H_1(M_i)\cong H_1(W)\cong \mathbb{Z}$.

A crucial part of the definition is the ``size'' (cardinality) of a rational $(n)$-Lagrangian, which is dictated by the rank of $H_2(W;\mathbb{Q})/I_0$. Under the assumption that
$$
H_1(\partial W;\mathbb{Q})\to H_1(W;\mathbb{Q})
$$
is an isomorphism, it follows that the dual map
$$
H_3(W,M;\mathbb{Q})\to H_2(\partial W;\mathbb{Q})
$$
is an isomorphism and hence that $I_0=0$. Thus in this special case (which was the case treated in (~\cite{COT}, ~\cite{CK}) and ~\cite{Ha2}), the size of rational $(n)$-solutions is dictated merely by the rank of $H_2(W;\mathbb{Q})$. This assumption will not hold in our applications. We need the more general situation.

\begin{defn}
\label{defn:rationalnbordism} Let $n$ be a nonnegative integer. A compact, connected oriented topological 4-manifold $W$ with $\partial W = M$ is a \textbf{rational $\mathbf{(n)}$-bordism} for $M$ if $W$ admits a rational $(n)$-Lagrangian with rational $(n)$-duals. Then we say that $M$ is \textbf{rationally $\mathbf{(n)}$-bordant} via $W$ and that $W$ is a \textbf{rational $\mathbf{(n)}$-bordism} for $M$. If $W$ is spin then we say that $W$ is an \textbf{$\mathbf{(n)}$-bordism} for $M$ if $W$ admits an $(n)$-Lagrangian with $(n)$-duals. Then we say that $M$ is \textbf{$\mathbf{(n)}$-bordant} via $W$ and that $W$ is an \textbf{$\mathbf{(n)}$-bordism} for $M$.
\end{defn}

\begin{defn}
\label{defn:n.5bordism} Let $n$ be a nonnegative integer. A compact, connected oriented 4-manifold $W$ with
$\partial W = M$ is a \textbf{rational $\mathbf{(n.5)}$-bordism} for $M$ if $W$ admits a rational $(n+1)$-Lagrangian with rational $(n)$-duals. Then we say that $M$ is \textbf{rationally $\mathbf{(n.5)}$-bordant} via $W$. If $W$ is spin then we say that $W$ is an \textbf{$\mathbf{(n.5)}$-bordism} for $M$ if $W$ admits an $(n+1)$-Lagrangian with $(n)$-duals. Then we say that $M$ is \textbf{$\mathbf{(n.5)}$-bordant} via $W$ and that $W$ is an \textbf{$\mathbf{(n.5)}$-bordism} for $M$.
\end{defn}

We recover Cochran-Orr-Teichner's notion of solvability by imposing the following additional restrictions.

\begin{defn}
\label{defn:nsolution}~\cite[Section 8]{COT} A $4$-manifold $W$ is an \textbf{$\mathbf{(n)}$-solution} (respectively an \textbf{$\mathbf{(n.5)}$-solution}) for $\partial W=M$ and $M$ is called \textbf{$\mathbf{(n)}$-solvable} (respectively \textbf{$\mathbf{(n.5)}$-solvable}) if
\begin{itemize}
\item [1.] $W$ is an \textbf{$\mathbf{(n)}$-bordism} (respectively an \textbf{$\mathbf{(n.5)}$-bordism}),
\item [2.] $\partial W$ is connected and non-empty, and
\item [3.] $H_1(\partial W;\mathbb{Z})\to H_1(W;\mathbb{Z})$ is an isomorphism.
\end{itemize}
\end{defn}
There is an analogous definition for  \textbf{rationally $\mathbf{(n)}$-solvable} (respectively \textbf{rationally $\mathbf{(n.5)}$-solvable}).

\begin{defn} For $h$ a non-negative integer or half-integer, a knot or link is called \textbf{$\mathbf{(h)}$-bordant} (respectively \textbf{rationally $\mathbf{(h)}$-bordant}, \textbf{$\mathbf{(h)}$-solvable}, \textbf{rationally $\mathbf{(h)}$-solvable}) if its zero surgery manifold $M_K$ admits an $(h)$-bordism (respectively a rational $(h)$-bordism, an $(h)$-solution, a rational $(h)$-solution).
\end{defn}
\begin{rem}
\label{rem:n-solvable}
\begin{enumerate}
\item Any $(h)$-bordism is  a rational $(h)$-bordism.
\item Any $(h)$-solution is  a rational $(h)$-solution.
\item Any $(h)$-solution is  an $(h)$-bordism.
\item Any rational $(h)$-solution is  a rational $(h)$-bordism.
\item Any $(n)$-bordism (respectively rational $(n)$-bordism) is  an $(m)$-bordism (respectively rational $(m)$-bordism) for any $m<n$.
\item If $L$ is slice in a topological (rational) homology $4$-ball then the complement of a set of slice disks is a (rational) $(n)$-solution for any integer or half-integer $n$.  This follows since $H_2(W;\mathbb{Z})=0$, and therefore the Lagrangian may be taken to be the zero submodule.
\end{enumerate}
\end{rem}

\begin{rem}\label{rem:0solv} One can see that any knot is is rationally 0-solvable as follows. Since $\Omega_3(S^1)=0$, $M_K$ is the boundary of some smooth $4$-manifold $W$ with $\pi_1(W)\cong \mathbb{Z}$ generated by the meridian (after surgery). The signature of $W$ can be assumed to be zero by connect-summing with copies of $\pm \mathbb{C}P(2)$. One can see that any Arf invariant zero knot is 0-solvable in the topological category by the same argument, using the fact that $\Omega_3^{Spin}(S^1)\cong \mathbb{Z}_2$ as detected by the Arf invariant and connect-summing with copies of Freedman's $\pm E_8$ manifold. For the argument in the smooth category see the explicit construction in ~\cite[Section 5]{COT2}.
\end{rem}

Certain $\rho$-invariants obstruct solvability.

\begin{thm}\label{thm:sliceobstr}(Cochran-Orr-Teichner~\cite[Theorem 4.2]{COT}) If a knot $K$ is rationally $(n.5)$-solvable via $W$ and  $\phi:\pi_1(M_K)\to \G$ is a PTFA coefficient system that extends to $\pi_1(W)$ and such that $\G^{(n+1)}=1$, then $\rho(M_K,\phi)=0$.
\end{thm}

\begin{prop}\label{prop:firstordervanish} If $K$ is topologically slice in a rational homology $4$-ball (or more generally if $K$ is rationally $(1.5)$-solvable) then one of the first-order signatures of $K$ is zero.
\end{prop}

\begin{proof}[Proof of Proposition~\ref{prop:firstordervanish}] Let $V$ be a rational $(1.5)$-solution for $M_K$, $G=\pi_1(M_K)$, $\pi=\pi_1(V)$ and $\phi:\pi\to \pi/\pi^{(2)}_r$. By  ~\cite[Theorem 4.2]{COT} $\rho(M_K,\phi)=0$. Clearly the restriction of $\phi$ to $G$ factors through $G/G^{(2)}$. Now, by ~\cite[Theorem 4.4]{COT} (see also our Theorem~\ref{thm:nontriviality}), if $P$ denotes the kernel of the map
$$
\mathcal{A}_0(K)\overset{i_*}{\to} H_1(M_K;\mathbb{Q}[\pi/\pi^{(1)}_r])\overset{j_*}\to H_1(V;\mathbb{Q}[\pi/\pi^{(1)}_r]),
$$
then $P\subset P^\perp$ with respect to the classical Blanchfield form of $K$. If $V$ is the exterior of a slice disk in a homology $4$-ball, this is merely the classical result that $P=P^\perp$. It follows that $\rho(M_K,\phi)$ is one of the first-order signatures of $K$. The details in verifying this final claim are entirely similar to those in the proof of Theorem~\ref{thm:generalJ2notslice}.
\end{proof}

There is an extension of Theorem~\ref{thm:sliceobstr} to the much broader category of $(n.5)$-null bordisms, but we shall not need it in this paper. However, a slightly weaker result will follow readily from results that we will need.

\begin{thm}\label{thm:sigvanishes} Suppose $W$ is a rational $(n+1)$-bordism and
$\phi:\pi_1(W)\ra\G$ is a non-trivial coefficient system where
$\G$ is a PTFA group with $\G^{(n+1)}=1$. Suppose for each component $M_i$ of $\partial W$ for which $\phi$ restricted to $\pi_1(M_i)$ is nontrivial, that $\text{rank}_{\mathbb{Z}\Lambda}H_1(M_i;\mathbb{Z}\G)=\beta_1(M_i)-1$. Then
$$
\rho(\partial W,\phi)= 0.
$$
\end{thm}

For the proof we need the following technical result that will also be crucial in Section \ref{Blanchfieldforms}. Recall our notation $\mathcal{K}\Lambda$ for the (skew) quotient field of fractions of $\mathbb{Z}\Lambda$. An \emph{Ore localization} of an Ore domain $\mathbb{Z}\Lambda$ is $\mathcal{R}=\mathbb{Z}\Lambda[S^{-1}]$ for some right-Ore set $S$ ~\cite{Ste}.

\begin{lem}\label{lem:exact} Suppose $W$ is a rational $(k)$-bordism and
$\phi:\pi_1(W)\ra\Lambda$ is a non-trivial coefficient system where
$\Lambda$ is a PTFA group with $\Lambda^{(k)}=1$. Let $\mathcal{R}$ be an Ore localization of $\mathbb{Z}\Lambda$ so $\mathbb{Z}\Lambda\subset\mathcal{R}\subset \mathcal{K}\Lambda$. Suppose for each component $M_i$ of $\partial W$ for which $\phi$ restricted to $\pi_1(M_i)$ is nontrivial, that $\text{rank}_{\mathbb{Z}\Lambda}H_1(M_i;\mathbb{Z}\Lambda)=\beta_1(M_i)-1$. Then
\begin{itemize}
\item [1.] The $\mathbb{Q}$-rank of $(H_2(W)/j_*(H_2(\partial W))$ is equal to the $\mathcal{K}\Lambda$-rank of $H_2(W;\mathcal{R})/I$ where
$$
I=\text{image}(j_*(H_2(\partial W;\mathcal{R})\to H_2(W;\mathcal{R}))).
$$
and
\item [2.]
$$
TH_2(W,\partial W;\mathcal{R})\xrightarrow{\partial}TH_1(\partial W;\mathcal{R})\xrightarrow{j_{\ast}} TH_1(W;\mathcal{R})
$$
is exact, where $T\mathcal{M}$ denotes the $\mathcal{R}$-torsion submodule of the $\mathcal{R}$-module $\mathcal{M}$.
\end{itemize}
\end{lem}

We should point out that the rank hypothesis on $H_1(M_i;\mathbb{Z}\Lambda)$ is \emph{always satisfied} if $\beta_1(M)=1$ (by ~\cite[Proposition 2.11]{COT}), which will always be the case in this paper. The more general result is needed to study links.

\begin{proof}[Proof of Lemma~\ref{lem:exact}] First we establish the rank claim. We can assume that $\partial W$ is not empty. Let $\beta_i$ denote the $i^{th}$-Betti number. By duality
$$
\beta_3(W)=\beta_1(W,\partial W)~~ \text{and} ~~\beta_2(W)=\beta_2(W,\partial W)
$$
Using these facts, by examining the long exact sequence of the pair for reduced homology with $\mathbb{Q}$-coefficients
$$
H_2(\partial W)\overset{j^2_*}\to H_2(W)\to H_2(W,\partial W)\overset{\partial_*}\lra H_1(\partial W)\overset{j_*}\to H_1(W)\to H_1(W,\partial W)\to \tilde{H_0}(\partial W)\to 0
$$
and setting alternating sums of ranks equal to zero, we see that
$$
-\text{rank}_{\mathbb{Q}}\text{im}j_*^2=-\beta_1(\partial W)+\beta_1(W)-\beta_3(W)+\beta_0(\partial W)-1.
$$
Now let $2m= rank_\mathbb{Q}(H_2(W)/j_*^2(H_2(\partial W))$. Then from the above we have
$$
2m=\beta_2(W)-\text{rank}_{\mathbb{Q}}\text{im}j_*^2=\beta_2(W)-\beta_1(\partial W)+\beta_1(W)-\beta_3(W)+\beta_0(\partial W)-1,
$$
or
$$
2m=\chi(W)+2\beta_1(W)-2-\beta_1(\partial W)+\beta_0(\partial W)
$$
where $\chi$ is the Euler characteristic.

We claim that $2m$ is also the $\mathcal{K}\Lambda$-rank of $H_2(W;\mathcal{R})/I$. First we show that this rank is at most $2m$. To see this let
$$
b_i(W)=~\rk_{\mathcal{K}\Lambda}H_i(W;\mathcal{K}\Lambda)\equiv \text{rank}_{\mathbb{Z}\Lambda}H_i(W;\mathbb{Z}\Lambda).
$$
Again, by duality
$$
b_3(W)=b_1(W,\partial W)~~ \text{and} ~~b_2(W)=b_2(W,\partial W)
$$
Since $W$ is connected and the coefficient system on $W$ is non-trivial, $b_0(W)=b_4(W)=0$ by ~\cite[Proposition 2.9]{COT}. Using these facts, by examining the long exact sequence of the pair for homology with $\mathcal{K}\Lambda$-coefficients
$$
H_2(\partial W)\overset{j^2_*}\to H_2(W)\to H_2(W,\partial W)\overset{\partial_*}\lra H_1(\partial W)\overset{j_*}\to H_1(W)\to H_1(W,\partial W)\to H_0(\partial W)\to 0,
$$
we see as above that
$$
-\text{rank}_{\mathcal{K}\Lambda}\text{im}j_*^2=-b_1(\partial W)+b_1(W)-b_3(W)+b_0(\partial W).
$$
so the rank of $H_2(W;\mathcal{R})/I$ is
$$
b_2(W)-\text{rank}_{\mathcal{K}\Lambda}\text{im}j_*^2=b_2(W)-b_1(\partial W)+b_1(W)-b_3(W)+b_0(\partial W).
$$
Since the Euler characteristic can be calculated with $\mathcal{K}\Lambda$-coefficients this can be written as
$$
\text{rank}_{\mathcal{K}\Lambda}(H_2(W;\mathcal{R})/I)=\chi(W)+2b_1(W)-b_1(\partial W)+b_0(\partial W).
$$
Combining this with our previous computation
$$
\text{rank}_{\mathcal{K}\Lambda}(H_2(W;\mathcal{R})/I)-2m=2(b_1(W)-\beta_1(W)+1)+\beta_1(\partial W)-b_1(\partial W)-\beta_0(\partial W)+b_0(\partial W).
$$
We claim that the quantity on the right-hand side of this equality is at most zero. By ~\cite[Proposition 2.11]{COT}
$$
b_1(W)\leq \beta_1(W)-1
$$
so the quantity in parentheses is non-positive. Thus it will suffice to show that
$$
(\beta_1(M_i)-b_1(M_i))-(\beta_0(M_i)-b_0(M_i))= 0
$$
for each component $M_i$ of $\partial W$.
If $M_i$ is a boundary component on which $\phi$ restricts to be trivial, then this is clear since then the $\mathbb{Q}$-ranks agree with the $\mathcal{K}\Lambda$-ranks. Otherwise $\beta_0(M_i)=1$, $b_0(M_i)=0$ by ~\cite[Proposition 2.9]{COT} and
$$
b_1(M_i)=\beta_1(M_i)-1
$$
by our hypothesis. Thus we have established that
$$
\text{rank}_{\mathcal{K}\Lambda}(H_2(W;\mathcal{R})/I)\leq 2m.
$$
We shall soon see that this rank is at least $2m$, hence equals $2m$.

\begin{rem}\label{rem:boundedrho} Note that we have actually shown more. Even with no rank assumptions on the boundary, we have shown that
$$
\text{rank}_{\mathcal{K}\Lambda}(H_2(W;\mathcal{R})/I)\leq 2m+\sum_{\phi_i\neq 0}(\beta_1(M_i)-1-b_1(M_i)).
$$
This remark will be used in a later paper.
\end{rem}
Recall that the cardinality of a rational $(k)$-Lagrangian for $W$ is, by definition, $m$.
Let $\{\ell_1,\ell_2,\ldots, \ell_m\}$ generate a rational
$(k)$-Lagrangian for $W$ and $\{d_1,d_2,\ldots, d_m\}$ its
$(k)$-duals in $H_2(W;\mathbb{Q}[\pi_1(W)/\pi_1(W)^{(k)}])$. Since $\Lambda^{(k)}=1$, $\phi$ factors through
$\phi' : \pi_1(W)/\pi_1(W)^{(k)} \lra \Lambda$. We denote by $\ell_i'$ and
$d_i'$ the images of $\ell_i$ and $d_i$ in $H_2(W;\mathcal{R})$. By
naturality of intersection forms, the intersection form $\lambda$
defined on $H_2(W;\mathcal{R})$ vanishes on the module generated by
$\{\ell_1',\ell_2',\ldots, \ell_m'\}$ and the $d_i'$ are still duals.  Recall that the intersection form factors through
$$
\ov{\lambda} : H_2(W;\mathcal{R})/I \times H_2(W;\mathcal{R})/I \lra \mathcal{R}.
$$
Let $\mathcal{R}^m\oplus \mathcal{R}^m$ be the
free module on $\{\ell_i',d_i'\}$ and let $(-)^*$ denote $Hom_\mathcal{R}(-,\mathcal{R})$. The following composition
$$
\mathcal{R}^m\oplus \mathcal{R}^m \xrightarrow{j_*} (H_2(W;\mathcal{R})/I) \xrightarrow{\ov\lambda}
(H_2(W;\mathcal{R})/I)^* \xrightarrow{j^*} (\mathcal{R}^m\oplus \mathcal{R}^m)^*
$$
is then the definition of $\ov{\lambda}$ (restricted to this free module) and so is represented by a block matrix
$$
\left(\begin{matrix} 0 & I\cr I & X
\end{matrix}\right),
$$
for some $X$.
This matrix has an inverse which is
$$
\left(\begin{matrix} -X & I\cr I & 0
\end{matrix}\right).
$$
Thus the composition is an isomorphism. This implies that both $j^*$ and $j^*\circ \ov\lambda$ are epimorphisms. It follows immediately that
$$
\text{rank}_{\mathcal{K}\Lambda}(H_2(W;\mathcal{R})/I)\geq 2m
$$
and hence equality must hold. This concludes the proof of the first claim of the lemma.

Continuing, the rank of $(H_2(W;\mathcal{R})/I)^*$ must also be $2m$, and hence the kernel of the epimorphism $j^{\ast}$ is the torsion submodule of $(H_2(W;\mathcal{R})/I)^*$. But the latter is torsion-free since $\mathcal{R}$ is a domain. Hence $j^{\ast}$ is an isomorphism and $(H_2(W;\mathcal{R})/I)^*$ is free of rank $2m$. It follows that $\ov\lambda$ is surjective. Now consider the commutative diagram below with
$\mathcal{R}$-coefficients.
$$
\begin{diagram}\label{diagram:Blanch}\dgARROWLENGTH=1.2em
\node{H_2(W)/I} \arrow{e,t}{\pi_{\ast}}
\arrow{sse,b}{\lambda} \arrow[2]{s,l}{\ov\lambda} \node{H_2(W,\partial W)} \arrow{e,t}{\partial_*}
\arrow{s,r}{P.D.}\node{H_1(\partial W)}
\arrow{e,t}{j_*} \arrow{s,r}{P.D.}\node{H_1(W)}\\
\node[2]{H^2(W)} \arrow{s,r}{\kappa}\node{H^2(\partial W)}\arrow{s,r}{\kappa}\\
\node{(H_2(W)/I)^*}\arrow{e,t}{q_*}\node{(H_2(W))^*}\arrow{e,t}{j^*}\node{(H_2(\partial W))^*}
\end{diagram}
$$

Given $p\in TH_1(\partial W;\mathcal{R})$ such that $j_*(p)=0$, choose $x$ such that
$\partial_* x=p$. Since $p$ is torsion, $\kappa\circ
\operatorname{P.D.}(p)=0$ and so $\kappa\circ
\operatorname{P.D.}(x)=q_*(z)$ for some $z\in (H_2(W)/I)^*$. Since $\ov\lambda$ is surjective, we can choose $y$ such that element $\ov\lambda(y)=z$. Then $\partial_*(x-\pi_{\ast}(y))=p$ and
$x-\pi_{\ast}(y)$ lies in the  kernel of
$\kappa\circ\operatorname{P.D.}$., hence is torsion. Thus we have shown that every torsion element of ker$j_*$ is in the image of an element of $TH_2(W,\partial W;\mathcal{R})$. This concludes the proof of the Lemma~\ref{lem:exact}.

\end{proof}

\begin{proof}[Proof of Theorem~\ref{thm:sigvanishes}] Note that the ordinary signature of any $(n+1)$-bordism vanishes since the $(n+1)$-Lagrangian projects to a Lagrangian of the ordinary intersection form. Let $\tilde{I}$ denote the image of the map
$$
H_2(\partial W;\mathcal{K}\G)\overset{j_*}{\lra}H_2(W;\mathcal{K}\G).
$$
By property $(1)$ of Proposition~\ref{prop:rho invariants}, it suffices to show that there is a one-half rank submodule, $\mathcal{L}$ of $H_2(W;\mathcal{K}\G)/\tilde {I}$ on which $\tilde\lambda$ vanishes. By the first part of Lemma~\ref{lem:exact}, applied with $\Lambda=\G$, $k=n+1$ and $\mathcal{R}=\mathcal{K}\G$, we see that we need to find an $\mathcal{L}$ whose rank is one half of
$$
\text{rank}_{\mathbb{Q}}(H_2(W;\mathbb{Q})/I_0)
$$
where $I_0$ is the image of $H_2(\partial W;\mathbb{Q})$. Let $\{\ell_1,\ell_2,\ldots, \ell_m\}$ generate a rational
$(n+1)$-Lagrangian for $W$ and $\{d_1,d_2,\ldots, d_m\}$ its
$(n+1)$-duals in $H_2(W;\mathbb{Q}[\pi_1(W)/\pi_1(W)^{(n+1)}])$. Recall that the cardinality of $m$ of a generating set for a Lagrangian is such that
$$
m=1/2\text{rank}_\mathbb{Q}(H_2(W;\mathbb{Q})/I_0).
$$
Since $\G^{(n+1)}=1$, $\phi$ factors through
$\phi' : \pi_1(W)/\pi_1(W)^{(n+1)} \lra \G$. We denote by $\tilde{\ell}_i$ and
$\tilde{d}_i$ the images of $\ell_i$ and $d_i$ in $H_2(W;\mathcal{K}\G)/\tilde{I}$. By
naturality of intersection forms, the (nonsingular) intersection form $\tilde{\lambda}$
induced on $H_2(W;\mathcal{K}\G)/\tilde{I}$ vanishes on the submodule, $\mathcal{L}$, generated by
$\{\tilde{\ell}_1,\tilde{\ell}_2,\ldots, \tilde{\ell}_m\}$. Moreover the $\tilde{d}_i$ are still duals. Since duals exist,
$$
\text{rank}_{\mathcal{K}\G}\mathcal{L}=m
$$
as required.
\end{proof}

\section{Higher-Order Blanchfield forms and n-Bordisms}\label{Blanchfieldforms}

We have seen in Lemma~\ref{lem:additivity} that an infection will have an effect on a $\rho$-invariant only if the infection circle $\eta$ survives under the map defining the coefficient system. For example if one creates a knot $J$ by infecting a slice knot $R$ along a curve $\eta$ that dies in $\pi_1(B^4-\Delta)$ for some slice disk $\Delta$ for $R$, then this infection will have no effect on the $\rho$-invariants associated to any coefficient system that extends over $B^4-\Delta$. Indeed the resulting knot is known to be topologically slice ~\cite{CFT}. Therefore it is important to prove \emph{injectivity} theorems concerning $\pi_1(S^3-R)\to\pi_1(B^4-\Delta)$, that is to locate elements of $\pi_1(S^3-R)$ that survive under such inclusions. Moreover the curve $\eta$ must usually lie in $\pi_1(S^3-R)^{(n)}$. For then it is known that $J$ will be rationally $n$-solvable and we seek to show that it is not $(n.5)$-solvable. Therefore, loosely speaking, we need to be able to prove that $\eta$ survives under the map
$$
j_*:\pi_1(S^3-R)^{(n)}/\pi_1(S^3-R)^{(n+1)}\to \pi_1(B^4-\Delta)^{(n)}/\pi_1(B^4-\Delta)^{(n+1)}.
$$
For $n=1$ this is a question about ordinary Alexander modules and was solved by Casson-Gordon and Gilmer using linking forms on finite branched covers. In general this seems a daunting task. (Note that this is impossible if $\pi_1(B^4-\Delta)$ is solvable, which occurs, for example, for the standard slice disk for the ribbon knot $R$ of Figure~\ref{fig:ribbonCG}(e.g. see ~\cite{FrT})). To see that higher-order Alexander modules are relevant to this task, observe that the latter quotient is the abelianization of $\pi_1(B^4-\Delta)^{(n)}$ and thus can be interpreted as $H_1(W_n)$ where $W_n$ is the (solvable) covering space of $B^4-\Delta$ corresponding to the subgroup $\pi_1(B^4-\Delta)^{(n)}$. Such modules were named \emph{higher-order Alexander modules} in ~\cite{COT}~\cite{C}~\cite{Ha1}. We will employ higher-order Blanchfield linking forms on higher-order Alexander modules to find restrictions on the kernels of such maps. The logic of the technique is entirely analogous to the classical case ($n=1$): Any two curves $\eta_0, \eta_1$, say, that lie in the kernel of $j_*$ must satisfy $\mathcal{B}\ell(\eta_0,\eta_0)=\mathcal{B}\ell(\eta_0,\eta_1)=\mathcal{B}\ell(\eta_1,\eta_1)=0$ with respect to a higher order linking form $\mathcal{B}\ell$. Our new insight is that, if the curves lie in a submanifold $S^3-K\hookrightarrow S^3-J$, a situation that arises whenever $J$ is formed from $R$ by infection using a knot $K$, then the values (above) of the higher-order Blanchfield form of $J$ can be expressed in terms of the values of the classical Blanchfield form of $K$!

Higher-order Alexander modules and higher-order linking forms for classical knot exteriors and for closed $3$-manifolds with $\beta_1(M)=1$ were introduced in ~\cite[Theorem 2.13]{COT} and further developed in ~\cite{C} and ~\cite{Lei1}. These were defined on the so called higher-order Alexander modules $TH_1(M;\mathcal{R})$, where $TH_1(M;\mathcal{R})$ denotes the $\mathcal{R}$-torsion submodule.

\begin{thm}\label{thm:blanchfieldexist}~\cite[Theorem 2.13]{COT} Suppose $M$ is a closed, connected, oriented $3$-manifold with $\beta_1(M)=1$ and $\phi:\pi_1(M)\to \Lambda$ is a PTFA coefficient system. Suppose $\mathcal{R}$ is a classical Ore localization of the Ore domain $\mathbb{Z}\Lambda$ (so $\mathbb{Z}\Lambda\subset\mathcal{R}\subset \mathcal{K}\Lambda$). Then there is a linking form:
$$
\mathcal{B}l^M_{\mathcal{R}}: TH_1(M;\mathcal{R})\to (TH_1(M;\mathcal{R}))^{\#}\equiv \overline{Hom_{\mathcal{R}}(TH_1(M;\mathcal{R}), \mathcal{K}\Lambda/\mathcal{R})}.
$$
\end{thm}

\begin{rem} It is crucial to our techniques that we work with such Blanchfield forms without localizing the coefficient systems. When we speak of the \emph{unlocalized} Blanchfield form we mean that $\mathcal{R}=\mathbb{Z}\Lambda$ or $\mathcal{R}=\mathbb{Q}\Lambda$. It is in this aspect that our work deviates from that of ~\cite{COT}~\cite{COT2}~\cite{CK}. This was investigated in ~\cite{Lei1}~\cite{Lei3}.  In this generality, $TH_1(M;\mathcal{R})$ need not have homological dimension one nor even be finitely-generated, and these linking forms are \emph{singular}. A non-localized Blanchfield form for knots also played the crucial role in ~\cite{FrT}.
\end{rem}


There is another key result of ~\cite{COT} concerning solvability whose generalization to null-bordism will be a crucial new ingredient in our proofs. Once again, the rank hypothesis is automatically satisfied if $\beta_1(M_i)=1$.

\begin{thm}\label{thm:selfannihil} Suppose $W$ is a rational $(k)$-null-bordism and
$\phi:\pi_1(W)\ra\Lambda$ is a non-trivial coefficient system where
$\Lambda$ is a PTFA group with $\Lambda^{(k)}=1$. Let $\mathcal{R}$ be an Ore localization of $\mathbb{Z}\Lambda$ so $\mathbb{Z}\Lambda\subset\mathcal{R}\subset \mathcal{K}\Lambda$. Suppose that, for each component $M_i$ of $\partial W$ for which $\phi$ restricted to $\pi_1(M_i)$ is nontrivial, that $\text{rank}_{\mathbb{Z}\Lambda}H_1(M_i;\mathbb{Z}\Lambda)=\beta_1(M_i)-1$. Then if $P$ is the kernel of the inclusion-induced map
$$
TH_1(\partial W;\mathcal{R})\xrightarrow{j_{\ast}} TH_1(W;\mathcal{R}),
$$
then $P\subset P^\perp$ with respect to the Blanchfield form on $TH_1(\partial W;\mathcal{R})$.
\end{thm}

\begin{proof}[Proof of Theorem~\ref{thm:selfannihil}] We need the following which was asserted in the proof of ~\cite[Theorem 4.4]{COT}. A careful proof in more generality is given in ~\cite{CHL4} (See also ~\cite[Lemmas 3.2, 3.3]{Cha3}).

\begin{lem}\label{lem:fourmanBlanch} There is a Blanchfield form, $\mathcal{B}l^{rel}$,
$$
\mathcal{B}l^{rel}_{\mathcal{R}}:~TH_2(W,\partial W;\mathcal{R})\to TH_1(W)^{\#}
$$
such that the following diagram, with coefficients in $\mathcal{R}$ unless specified otherwise, is commutative up to sign:
\begin{equation}\label{diag:compatible2}
\begin{diagram}\dgARROWLENGTH=1em
\node{TH_2(W,\partial W;\mathcal{R})} \arrow{e,t}{\partial_*}
\arrow{s,r}{\mathcal{B}l^{rel}_\mathcal{R}}\node{TH_1(\partial W;\mathcal{R})} \arrow{s,r}{\mathcal{B}l^{\partial W}_\mathcal{R}}\\
\node{TH_1(W;\mathcal{R})^{\#}} \arrow{e,t}{\tilde{j_*}}\node{TH_1(\partial W;\mathcal{R})^{\#}}
\end{diagram}
\end{equation}
\end{lem}

Now suppose $P=\text{kernel}j_*\subset TH_1(\partial W;\mathcal{R})$. Suppose $x\in P$ and $y\in P$. According to Lemma~\ref{lem:exact}, we have $x=\partial_*(\tilde{x})$ for some $\tilde{x}\in TH_2(W,\partial W)$. Thus by Diagram~\ref{diag:compatible2},
$$
\mathcal{B}l^{\partial W}_\mathcal{R}(x)(y)=\mathcal{B}l^{\partial W}_\mathcal{R}(\partial_*\tilde{x})(y)=\tilde{j_*}(\mathcal{B}l^{rel}_\mathcal{R}(\tilde{x})(y)= \mathcal{B}l^{rel}_\mathcal{R}(\partial_*\tilde{x})(j_*(y))=0
$$
since $j_*(y)=0$. Hence $P\subset P^\perp$ with respect to the Blanchfield form on $TH_1(\partial W;\mathcal{R})$.

This concludes the proof of Theorem~\ref{thm:selfannihil}.
\end{proof}

In many important situations the induced coefficient system $\phi:\pi_1(M_K)\to \Lambda$ factors through, $\mathbb{Z}$, the abelianization. In this case the higher-order Alexander module of $M_K$ and the higher-order Blanchfield form $\mathcal{B}l^K_\Lambda$ are merely the classical Blanchfield form on the classical Alexander module, ``tensored up''. What is meant by this is the following. Supposing that $\phi$ is both nontrivial and factors through the abelianization, the induced map $\text{image}(\phi)\equiv\mathbb{Z}\hookrightarrow \Lambda$ is an embedding so it induces embeddings
$$
\phi:\mathbb{Q}[t,t^{-1}]\hookrightarrow \mathbb{Q}\Lambda,~\text{and}~~ \phi:\mathbb{Q}(t)\hookrightarrow \mathcal{K}\Lambda.
$$
Moreover there is an isomorphism
$$
H_1(M_K;\mathbb{Q}\Lambda)\cong H_1(M_K;\mathbb{Q}[t,t^{-1}])\otimes_{\mathbb{Q}[t,t^{-1}]}\mathbb{Q}\Lambda \cong \mathcal{A}_0(K)\otimes_{\mathbb{Q}[t,t^{-1}]}\mathbb{Q}\Lambda,
$$
where $\mathcal{A}_0(K)$ is the classical (rational) Alexander module of $K$ and where $\mathbb{Q}\Lambda$ is a $\mathbb{Q}[t,t^{-1}]$-module via the map $t\to \phi(\alpha)$ ~\cite[Theorem 8.2]{C}.
We further claim:
\begin{lem}\label{lem:embedding} $\phi$ induces an embedding
$$
\ov\phi:\mathbb{Q}(t)/\mathbb{Q}[t,t^{-1}]\hookrightarrow \mathcal{K}\Lambda/\mathbb{Q}\Lambda.
$$
\end{lem}
\begin{proof}[Proof of Lemma~\ref{lem:embedding}] Consider the monomorphism of groups $\mathbb{Z}\overset{\phi}\hookrightarrow \Lambda$, where we will abuse notation by setting $t\equiv \phi(t)$. In other words we will consider that $\mathbb{Z}\subset \Lambda$. Then the Lemma is equivalent to
$$
\mathbb{Q}(t)\cap \mathbb{Q}\Lambda\subset\mathbb{Q}[t,t^{-1}].
$$
Suppose
$$
p(t)/r(t)=x\in \mathbb{Q}\Lambda,
$$
where $r(t)\neq 0$. We seek to show that $x\in \mathbb{Q}[t,t^{-1}]$. Consider the equation
\begin{equation}\label{eq:embedding}
p(t)=x r(t)
\end{equation}
in $\mathbb{Q}\Lambda$. The key point is that since $\mathbb{Z}\subset \Lambda$, $\mathbb{Q}\Lambda$ is free as a right $\mathbb{Q}[t,t^{-1}]$-module on the left cosets of $\mathbb{Z}$ in $\Lambda$, i.e.
$$
\mathbb{Q}\Lambda\cong \oplus_{\text{ cosets}}\mathbb{Q}[t,t^{-1}].
$$
Thus for each coset representative $\gamma$ we can speak of the $\gamma$ coordinate of $x$, $x_\gamma$, which is the polynomial in $\mathbb{Q}[t,t^{-1}]$ occurring in the above decomposition of $x$. We can decompose $x$ as
$$
x=\Sigma_\gamma \gamma x_\gamma; \Rightarrow x r(t)=\Sigma_\gamma \gamma (x_\gamma r(t)).
$$
Equation~\ref{eq:embedding} is equivalent to a system of equations, one for each coset representative. For each $\gamma\neq e$ this equation is:
$$
0=x_\gamma(t)r(t)
$$
implying that $x_\gamma(t)=0$. Thus $x\in \mathbb{Q}[t,t^{-1}]$.
\end{proof}

Continuing, then we also have
\begin{equation}\label{eq:tensorup}
\mathcal{B}l_{\Lambda}^K(x\otimes 1,y\otimes 1)=\ov\phi(\mathcal{B}l^K_0(x,y))
\end{equation}
for any $x,y\in \mathcal{A}_0(K)$, where $\mathcal{B}l_0^K$ is the classical Blanchfield form on the rational Alexander module of $K$ ~\cite[Proposition 3.6]{Lei3}~\cite[Theorem 4.7]{Lei1} (see also ~\cite[Section 5.2.2]{Cha2}).

 The following is perhaps the key technical tool of the paper, that we use to establish certain ``injectivity'' as discussed in the first paragraph of this section. For the reader who is just concerned with proving that knots and links are not slice, replace the hypothesis below that ``$W$ is a rational $(k)$-solution for $M_L$'' with the hypothesis that ``$L$ is a slice link and $W$ is the exterior in $B^4$ of a set of slice disks for $L$''. Such an exterior is a rational $(k)$-solution for any $k$.

\begin{thm}\label{thm:nontriviality} Suppose $W$ is a rational $(k)$-bordism one of whose boundary components is $M_K$, $\Lambda$ is a PTFA group such that $\Lambda^{(k)}=1$, and $\psi:\pi_1(W)\to \Lambda$ is a coefficient system whose restriction to $\pi_1(M_K)$ is denoted $\phi$. Suppose that $\phi$ factors non-trivially through $\mathbb{Z}$. Let $P$ be the kernel of the composition
$$
\mathcal{A}_0(K)\overset{id\otimes 1}\lra  \mathcal{A}_0(K) \otimes_{\mathbb{Q}[t,t^{-1}]}\mathbb{Q}\Lambda \overset{i_*}{\to} H_1(M_K;\mathbb{Q}\Lambda)\overset{j_*}\to H_1(W;\mathbb{Q}\Lambda).
$$
Then $P\subset P^\perp$ with respect to $\mathcal{B}l_0$, the classical Blanchfield linking form on the rational Alexander module, $\mathcal{A}_0(K)$, of $K$.
\end{thm}

\begin{proof}[Proof of Theorem~\ref{thm:nontriviality}] Suppose $x,y\in P$ as in the statement. Let $\mathcal{R}=\mathbb{Q}\Lambda$, $M=M_K$ and let $P$ be the submodule of $H_1(M;\mathbb{Q}\Lambda)$ generated by $\{i_*(x\otimes 1),i_*(y\otimes 1)\}$. Then $P\subset \text{kernel}~j_*$. Apply Theorem~\ref{thm:selfannihil} to conclude that
$$
\mathcal{B}l^{K}_{\Lambda}(i_*(x\otimes 1)),(i_*(y\otimes 1))=0.
$$
By ~\ref{eq:tensorup},
$$
\ov\phi(\mathcal{B}l_0^{K}(x,y))=0.
$$
Since $\ov\phi$ is a monomorphism by hypothesis, it follows that $\mathcal{B}l_0^K(x,y)=0$. Thus $P\subset P^\perp$ with respect to the classical Blanchfield form on $K$. This concludes the proof of Theorem~\ref{thm:nontriviality}.
\end{proof}

\section{Constructions of $(n)$-solvable knots}\label{sec:nsolvable}

In preparation for our proof that $\mathcal{F}_n/\mathcal{F}_{n.5}$ has infinite rank, we will exhibit large classes of knots that are $(n)$-solvable, including the knots $J_n(K)$, for any $J_0=K$ of Figure~\ref{fig:family}. Specifically we show that any knot obtained by starting with an Arf invariant zero knot and applying $n$ successive operators $R_i^{\eta_{ij}}$, where $R$ is a slice knot and the $\eta_{ij}$ are in the commutator subgroup, is $(n)$-solvable. In the proof of our main theorem we will need an $(n)$-solution with some special features, which we produce here.

\begin{thm}\label{thm:nsolvable} If $R_i$, $1\leq i\leq n$, are slice knots and $\eta_{ij}\in \pi_1(S^3-R_i)^{(1)}$ (where $\{\eta_{i1},...,\eta_{im_i}\}$ is a trivial link in $S^3$) then , abbreviating the operator $R_i^{\eta_{ij}}$ by $R_i$.
$$
R_n\circ\dots\circ R_2\circ R_1(\mathcal{F}_{0})\subset \mathcal{F}_{n}.
$$
More precisely, for any Arf invariant zero knot $K$, if we abbreviate $R_n\circ\dots\circ R_2\circ R_1(K)$ by $\mathcal{J}_n$ then the zero surgery on $\mathcal{J}_n$, denoted $M_n$, bounds an $(n)$-solution $Z_n$ with the following additional properties:
\begin{itemize}
\item [1.] $\pi_1(\partial Z_n)\to \pi_1(Z_n)$ is surjective;
\item [2.] for any PTFA coefficient system $\phi:\pi_1(Z_n)\to \G$ where $\G^{(n+1)}=1$
$$
\rho(M_n,\phi)=\sigma^{(2)}_\G(Z_n)-\sigma(Z_n)=c_\phi\rho_0(K)
$$
where $c_\phi$ is a non-negative integer bounded above by the product $m_1m_2...m_n$.
\end{itemize}
\end{thm}

\begin{cor}\label{cor:nsolvable} For any Arf invariant zero knot $K=J_0$, each $J_n$ as in Figure~\ref{fig:family} is $(n)$-solvable. Moreover the zero surgery on $J_n$ bounds an $(n)$-solution $Z_n$ with the following additional properties:
\begin{itemize}
\item [1.] $\pi_1(\partial Z_n)\to \pi_1(Z_n)$ is surjective;
\item [2.] for any PTFA coefficient system $\phi:\pi_1(Z_n)\to \G$ where $\G^{(n+1)}=1$
$$
\rho(M_n(K),\phi)=\sigma^{(2)}_\G(Z_n)-\sigma(Z_n)=c_\phi\rho_0(K)
$$
where $c_\phi$ is an integer such that $0\leq c_\phi\leq 2^n$.
\end{itemize}
\end{cor}
\begin{proof}[Proof of Theorem~\ref{thm:nsolvable}] The proof is by induction on $n$. Suppose $n=0$ so $\mathcal{J}_n=K$ and $M_n=M_K$. Any Arf invariant zero knot $K$ admits a $(0)$-solution $Z_0$ such that $\pi_1\cong \mathbb{Z}$ so property $1$ holds (see Remark~\ref{rem:0solv} or ~\cite[Section 5]{COT2}). Then, since $\G$ is abelian if $n=0$, $\phi$ factors through $\mathbb{Z}$ and so $\rho(M_K,\phi)$ is either zero or equal to $\rho_0(K)$ (see part $4$ of Proposition~\ref{prop:rho invariants}). Thus the Theorem holds for $n=0$.

Now suppose that $Z_{n-1}$ exists satisfying the properties $1$ and $2$. We construct $Z_n$ as follows. Recall that, by definition, $\mathcal{J}_{n}=R_n(\mathcal{J}_{n-1})$ is obtained from $R_n$ by $m_n$ infections along the circles $\{\eta_{n1},...,\eta_{nm_n}\}$ using the knot $\mathcal{J}_{n-1}$ as the infecting knot in each case. Recall also from Lemma~\ref{lem:mickeyfacts} that there was a corresponding cobordism $E$ with $m_n+2$ boundary components: ~$M_{n}$, $M_{R_n}$ and $m_n$ copies of $M_{n-1}$ as shown in Figure~\ref{fig:mickey}. Beginning with $E$, cap off the $m_n$ boundary components with $m_n$ copies of $Z_{n-1}$ and cap off $M_{R_n}$ with $\mathcal{R}=B^4-\Delta$, the exterior of any ribbon disk $\Delta$ for $R_n$ as shown schematically in Figure~\ref{fig:zn}. The resulting manifold has a single copy of $M_n$ as boundary and is denoted $Z_n$.

\begin{figure}[htbp]
\setlength{\unitlength}{1pt}
\begin{picture}(180,200)
\put(0,0){\includegraphics{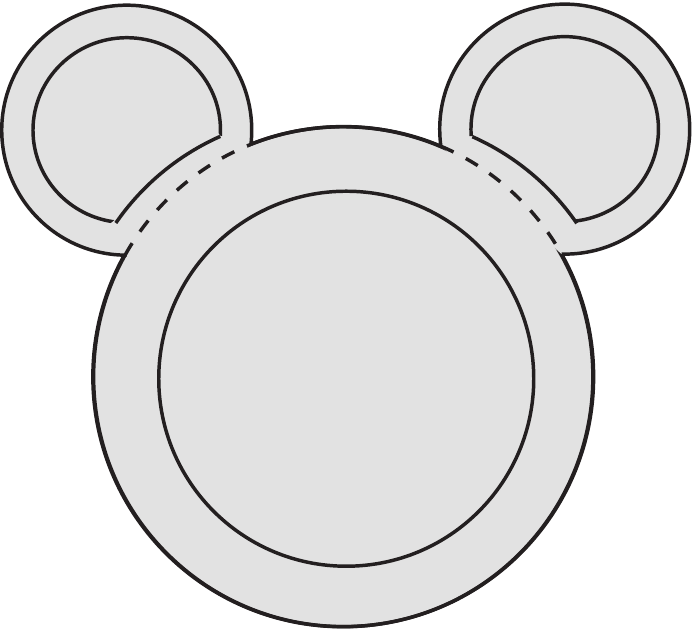}}
\put(27,145){$Z_{n-1}$}
\put(154,145){$Z_{n-1}$}
\put(96,66){$\mathcal{R}$}
\put(87,165){$M_{n-1}$}
\put(87,153){$\dots\dots$}
\put(80,168){\vector(-1,0){28}}
\put(123,168){\vector(1,0){25}}
\put(184,55){$M_{n}$}
\put(184,58){\vector(-1,0){12}}
\end{picture}
\caption{$Z_{n}$}\label{fig:zn}
\end{figure}

Property $1$ of Theorem~\ref{thm:nsolvable} follows from property $1$ for $Z_{n-1}(K)$ together with property $1$ of Lemma~\ref{lem:mickeyfacts}.

\textbf{$Z_n(K)$ is an $(n)$-solution}:

This will follow from a simple analysis of $H_2(Z_n;\mathbb{Z})$. We will drop the $\mathbb{Z}$ from the notation here for simplicity.

Recall from Lemma~\ref{lem:mickeyfacts} that
$$
H_2(E)\cong H_2(M_R)\oplus_{j=1}^{m_n} H_2(M_{n-1}).
$$
Since $Z_{n-1}$ is an $(n-1)$-solution, $H_1(M_{n-1})\to H_1(Z_{n-1})$ is an isomorphism. It follows from duality that $H_2(M_{n-1})\to H_2(Z_{n-1})$ is the zero map (a capped-off Seifert surface for $\mathcal{J}_{n-1}$ is a generator of the former and arises as the inverse image of a regular value under a map to a circle. Extend this map to $Z_{n-1}$ and pull back to get a bounding $3$-manifold). Therefore the Mayer-Vietoris sequence implies that
$$
H_2(E\cup_{j=1}^{m_n} Z_{n-1})\cong H_2(M_R)\oplus_{j=1}^{m_n} H_2(Z_{n-1}).
$$
The same facts apply to $M_R=\partial \mathcal{R}$ by Remark~\ref{rem:Rfacts} so
$$
H_2(Z_n)=H_2(E\cup_{j=1}^{m_n} Z_{n-1}\cup \mathcal{R})\cong \oplus_{j=1}^{m_n} H_2(Z_{n-1}),
$$
since $H_2(\mathcal{R})=0$.

Now, let $\{\ell_1^j,\dots,\ell_g^j\}$ be a collection of $(n-1)$-surfaces generating an $(n-1)$-Lagrangian for the $j^{th}$ copy, $Z_{n-1}^j$, of the $(n-1)$-solution $Z_{n-1}$ and $\{d_1^j,\dots,d_g^j\}$ a collection of $(n-1)$-surfaces that are $(n-1)$-duals. By our analysis of $H_2$, these collections, taken together for $1\leq j\leq m_n$, represent a basis for $H_2(Z_n)$ and so have the required \emph{cardinality} to generate an $n$-Lagrangian with $(n)$-duals for $Z_n$. By property $(1)$ of Theorem~\ref{thm:nsolvable} $\pi_1(Z_{n-1}^j)$ is normally generated by the meridian of the $j^{th}$ copy of $\mathcal{J}_{n-1}$. By definition of infection this meridian is equated to $\eta_{nj}$ in $E$. Since the $\eta_{nj}$ lie in the commutator subgroup of $\pi_1(M_{n-1})$ we see that $\pi_1(Z_{n-1}^j)$ maps into $\pi_1(Z_n)^{(1)}$. Thus $\pi_1(Z_{n-1}^j)^{(n-1})$ maps into $\pi_1(Z_n)^{(n)}$. Therefore the above $(n-1)$-surfaces for $Z_{n-1}^j$ are actually $(n)$-surfaces are for $Z_n$. By functoriality of the intersection form with twisted coefficients the union of these surfaces, over all $j$, also has the required intersection properties to generate an $n$-Lagrangian with $(n)$-duals for $Z_n$. Hence $Z_n$ is in fact an $(n)$-solution as was claimed.

\textbf{Property $2$ of Theorem~\ref{thm:nsolvable} for $Z_n$}:

\noindent Assume that $\phi:\pi_1(Z)\to \G$ where $\G^{(n+1)}=1$.  Recall that both $\sigma$ and $\sigma^{(2)}_\G$ are additive. By Lemma~\ref{lem:mickeysig}, both signatures vanish for $E$. By Theorem~\ref{thm:oldsliceobstr} both signatures vanish for $\mathcal{R}$. Therefore
$$
\sigma^{(2)}_\G(Z_n)-\sigma(Z_n)=\sum_{j=1}^{m_n}(\sigma^{(2)}_\G(Z_{n-1},\phi_j)-\sigma(Z_{n-1}))
$$
where $\phi_j$ is the induced coefficient system on the $j^{th}$ copy of $Z_{n-1}$. Let $\G_j$ be the image of $\phi_j$. By property $2$ of Proposition~\ref{prop:rho invariants} to compute $\sigma^{(2)}_\G(Z_{n-1},\phi_j)$ we may consider $\phi_j$ as a map into $\G_j$. We observed above that
each $\pi_1(Z_{n-1}^j)$ maps into $\pi_1(Z_n)^{(1)}$.  These $\G_j$ are subgroups of a PTFA group and hence are PTFA, and since $\G_j\subset \G^{(1)}$, $\G_j^{(n)}=1$.  Thus property $2$ of Theorem~\ref{thm:nsolvable} for $Z_{n-1}$ may be applied to $Z_{n-1}$ and $\phi_j:\pi_1(Z_{n-1}^j)\to\G_j$. Thus
$$
\sigma^{(2)}_\G(Z_n)-\sigma(Z_n)=\sum_{j=1}^{m_n}c_{\phi_j}\rho_0(K)=\rho_0(K)\sum_{j=1}^{m_n}c_{\phi_j}
$$
where $0\leq c_{\phi_j}\leq m_1m_2...m_{n-1}$. Property $2$ for $Z_n$ is thus established.

This concludes the proof of Theorem~\ref{thm:nsolvable}.
\end{proof}

\section{$\FF_n/\FF_{n.5}$ has infinite rank}\label{sec:infrank}

In this section we prove one of our main theorems.

\begin{thm}\label{thm:infgen} For any $n\ge0$, $\FF_n/\FF_{n.5}$ has infinite rank.
\end{thm}

\begin{proof}[Proof of Theorem~\ref{thm:infgen}] We give a procedure to construct an infinite set of knots in $\FF_n$ that is linearly independent in $\FF_n/\FF_{n.5}$.

\textbf{Step 1.} Find a genus one ribbon knot, $R$, such that $\rho^1(R)\neq 0$.

\noindent Let $R$ be the genus one ribbon knot shown in Figure~\ref{fig:exs_eta2}.
\begin{figure}[htbp]
\begin{picture}(153,165)
\put(0,0){\includegraphics{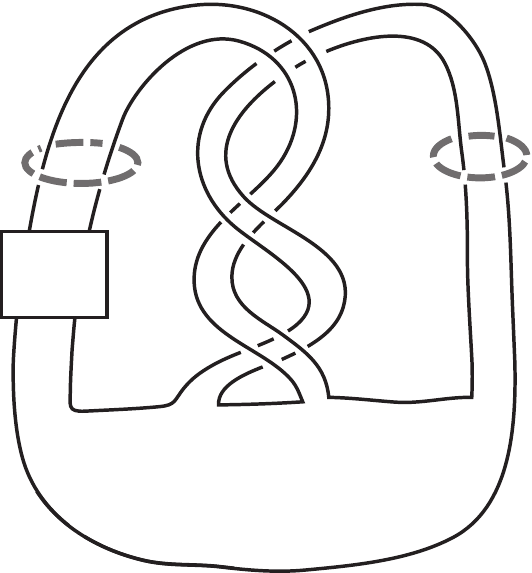}}
\put(12,83){$T^{\ast}$}
\put(-5,120){$\alpha$}
\put(159,121){$\beta$}
\end{picture}
\caption{The ribbon knot $R$}\label{fig:exs_eta2}
\end{figure}
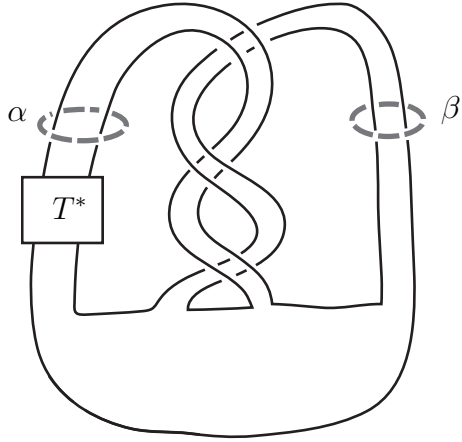
There are two cases:

\textbf{Case I} $\rho^1(9_{46})\neq 0$.

\noindent In this case we define $T*$ to be the unknot, so $R=R_I=9_{46}$. Thus in this case $\rho^1(R)=\rho^1(9_{46})\neq 0$ and Step $1$ is complete.

\textbf{Case II} $\rho^1(9_{46})=0$.

\noindent In this case we define $T^*$ to be the right-handed trefoil knot, $T$. In this case we sometimes refer to $R$ as $R_{II}$, to distinguish the case. As in Example~\ref{ex:first-ordersigs}, we may apply Lemma~\ref{lem:additivity} to calculate
$$
\rho^1(R_{II})=\rho^1(9_{46})+\rho_0(T)=\rho_0(T).
$$
It is an easy calculation that $\rho_0(T)$ is non-zero. Thus $\rho^1(R)\neq 0$ and so Step $1$ holds in each case.

For the rest of the proof, we refer to the ribbon knot $R$ whenever the argument applies to both cases.  We split the argument into the two cases only when necessary.

\vspace{.25in}

\textbf{Step 2.} Find an infinite set $\mathcal{K}$ of Arf invariant zero knots $K^j$ such that no nontrivial rational linear combination of $\{\rho_0(K^j)\}$ is a rational multiple of $\rho^1(R)$. \\

This requirement is stronger than linear independence but is easily accomplished by the following elementary linear algebra. It was shown in ~\cite[Proposition 2.6]{COT2} that there exists an infinite set, $\bar {\mathcal{K}}=\{K^j\}$, of Arf invariant zero knots such that $\{\rho_0(K^j)\}$ is $\mathbb{Q}$-linearly independent. Let $V$ be the $\mathbb{Q}$-vector subspace of $\mathbb{R}$ with $\{v_j=\rho_0(K^j)\}$ as basis. If $\rho^1(R)$ is not in $V$ then set $\mathcal{K}=\bar{\mathcal{K}}$ and we are done. If $\rho^1(R)\in V$ then $\rho^1(R)$ has a unique expression as a nontrivial linear combination of the $v_j$. Let $\mathcal{K}$ be the subset of $\bar{\mathcal{K}}$ obtained by omitting one of the knots $K^{j}$ for which $v_j$ occurs nontrivially in this expression. Then $d$ is not in the span of $\mathcal{K}$ and we are done.

\vspace{.25in}

\textbf{Step 3.} For each fixed $n$ define a family of knots $\{J_n^j~|1\leq j\leq \infty\}\subset \FF_n$.\\

\noindent The families are defined recursively. Fix $j$ and set $J_0^j\equiv K^j$ and let $J_{n+1}^j$ be the knot obtained from the ribbon knot $R$ of Figure~\ref{fig:exs_eta2} by infection along the two band meridians $\{\alpha,\beta\}$ using the knot $J_{n}^j$ in each case, as shown in Figure~\ref{fig:familyJ_n^j}.

\begin{figure}[htbp]
\begin{picture}(160,165)
\put(13,0){\includegraphics{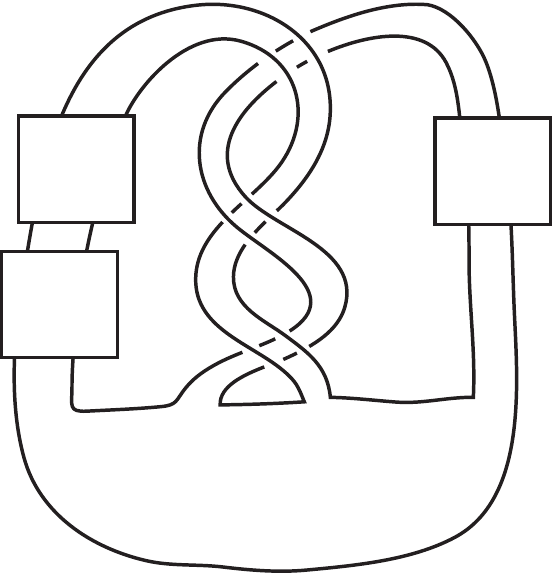}}
\put(25,74){$T^{\ast}$}
\put(21,113){$J_{n}^j$}
\put(142,113){$J_{n}^j$}
\put(-45,94){$J^j_{n+1}=$}
\end{picture}
\caption{The family of $(n)$-solvable knots $J^j_n$}\label{fig:familyJ_n^j}
\end{figure}
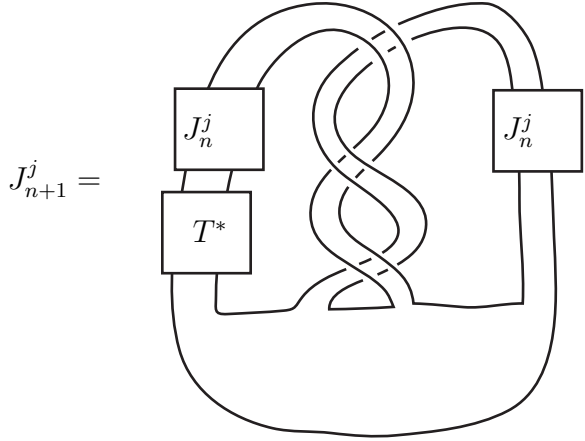
For any $n$ and $j$, $J_{n}^j$ is $(n)$-solvable by Theorem~\ref{thm:nsolvable}, so $J_{n}^j\in \FF_n$. This completes Step 3.

\vspace{.25in}

\textbf{Step $4$. No nontrivial linear combination of the knots $\{J_n^j~|~ 1\leq j\leq \infty\}$ is rationally $(n.5)$-solvable.}

The proof occupies the remainder of this section.
Here $n$ is fixed. We proceed by contradiction. Suppose that $\tilde J\equiv \#_{j=1}^\infty m_jJ^j_n$ (a finite sum) were rationally $(n.5)$-solvable. By re-indexing, without loss of generality we may assume that $m_1>0$. Under this assumption we shall construct a family of $4$-manifolds $W_i$ and reach a quick contradiction. Throughout we abbreviate the zero-framed surgery $M_{J^j_n}$ by $M^j_n$.

\begin{prop}\label{prop:familyofmickeys} Under the assumption that $\tilde J$ is rationally $(n.5)$-solvable, for each $0\leq i\leq n$ there exists a $4$-manifold $W_i$ with the following properties. Letting $\pi=\pi_1(W_i)$,

\begin{itemize}
\item [(1)] $W_i$ is a rational $(n)$-bordism where, for $i<n$, $\partial W_i=M^1_{n-i}$ and $\partial W_n=M^1_{0}\coprod M^1_{0}\coprod M_R$;
\item [(2)] Under the inclusion(s) $j:M_{n-i}^1\subset\partial W_i\to W_i$,
$$
j_*(\pi_1(M^1_{n-i}))\subset \pi^{(i)};
$$
and for each $i$ (at least one of the copies of) $M_{n-i}^1\subset \partial W_i$
$$
j_*(\pi_1(M^1_{n-i}))\cong \mathbb{Z}\subset \pi^{(i)}/\pi^{(i+1)}_r;
$$
and under the inclusion $j:M_R\subset\partial W_n\to W_n$
$$
j_*(\pi_1(M_R))\cong \mathbb{Z}\subset \pi^{(n-1)}/\pi_r^{(n)};
$$
\item [(3)] For any PTFA coefficient system $\phi:\pi_1(W_i)\to\G$ with $\G^{(n+1)}_r=1$
$$
\rho(\partial W_i,\phi)\equiv\sigma_\G^{(2)}(W_i,\phi)-\sigma(W_i)=-\sum_{j}C^j\rho_0(K^j)
$$
for some integers $C^j$ (depending on $\phi$) where $C^1\geq 0$.
\end{itemize}
\end{prop}

Before proving Proposition~\ref{prop:familyofmickeys}, we use it to finish the proof of Step $4$ and complete the proof of Theorem~\ref{thm:infgen}. Consider $W_{n}$ from Proposition~\ref{prop:familyofmickeys} with boundary $M^1_0\coprod M^1_0 \coprod M_R$. Recall that $M^1_0=M_{J^1_0}=M_{K^1}$. Let $\pi=\pi_1(W_n)$ and consider $\phi:\pi\to\pi/\pi^{(n+1)}_r$. Then by property $(3)$ of Proposition~\ref{prop:familyofmickeys} for $i=n$
\begin{equation}\label{eq:1}
\rho(M_{K^1},\phi_\alpha)+ \rho(M_{K^1},\phi_\beta)+\rho(M_R,\phi_R)=\rho(\partial W_n,\phi)=-\sum_{j}C^j\rho_0(K^j)
\end{equation}
where $C^1\geq 0$. By property $(2)$ of Proposition~\ref{prop:familyofmickeys}
$$
j_*(\pi_1(M_{K^1})\subset \pi^{(n)},
$$
implying that the restrictions $\phi_\alpha$ and $\phi_\beta$ factor through the respective abelianizations. Additionally by property $(2)$, at least one of these coefficient systems is non-trivial. Hence by $(2)-(4)$ of Proposition~\ref{prop:rho invariants}
$$
\rho(M_{K^1},\phi_\alpha)+ \rho(M_{K^1},\phi_\beta)=\epsilon\rho_0(K^1)
$$
where $\epsilon$ equals either $1$ or $2$. Thus we can simplify ~\ref{eq:1} to yield
\begin{equation}\label{eq:second}
(\epsilon+C^1)\rho_0(K^1)+\sum_{j>1}C^j\rho_0(K^j)=-\rho(M_R,\phi_R)
\end{equation}
where $\epsilon+C^1\geq 1$.
 Also by property $(2)$ of Proposition~\ref{prop:familyofmickeys},
\begin{equation}\label{eq:M_R}
j_*(\pi_1(M_R))\subset \pi^{(n-1)},
\end{equation}
so $\phi_R$ factors through $G/G^{(2)}$ where $G=\pi_1(M_R)$. We claim that kernel($\phi_R)\subset G^{(1)}$. For suppose that $x\in \text{kernel}(\phi_R)$ and $x=\mu^my$ where $\mu$ is a meridian of $R$ and $y\in G^{(1)}$. Then certainly $x$ is in the kernel of the composition
$$
\psi:G\overset{\phi_R}\longrightarrow \pi^{(n-1)}/\pi^{(n+1)}_r\to\pi^{(n-1)}/\pi^{(n)}_r.
$$
Moreover, by ~\ref{eq:M_R}, $\phi_R(G^{(1)})\subset \pi^{(n)}_r$ so $G^{(1)}$ is in the kernel of $\psi$. Therefore $\mu^m\in \ker\psi$ and the image of $\psi$ has order at most $m$. If $m\neq 0$ this contradicts the last clause of property $2$ of Proposition~\ref{prop:familyofmickeys}. Thus $m=0$ and kernel($\phi_R)\subset G^{(1)}$. Therefore $\psi$ is determined by the kernel, $P$, of
$$
\bar\phi_R:G^{(1)}/G^{(2)}\cong \mathcal{A}_0(R)\to \text{image}(\phi_R).
$$
Since $P$ is normal in $G/G^{(2)}$, it is preserved under  conjugation by a meridional element, implying that $P$ is a \emph{submodule} of $\mathcal{A}_0(R)$.  Since the Alexander polynomial of $R$ is $(2t-1)(t-2)$, the product of two irreducible coprime factors, $\mathcal{A}_0(R)$ admits precisely $4$ submodules: $P_1=\mathcal{A}_0(R)$, $P_0=0, P_\alpha=<\alpha>$ and $P_\beta=<\beta>$. In the first case $\bar\phi_R$ is the zero map so $\phi_R$ factors through $\mathbb{Z}$ and $\rho(M_R,\phi)=\rho_0(R)=0$. Otherwise $\rho(M_R,\phi_R)$ is what we have called a first order signature of $R$. To analyze the remaining $3$ possibilities for $\rho(M_R,\phi_R)$ it is simplest to take the viewpoint that $R$ is obtained from $9_{46}$ by one infection along $\alpha$ using the knot $T^*$. We can then analyze the $3$ possible first-order signatures as in Example~\ref{ex:first-ordersigs},
$$
\rho(M_R,P_\alpha)=\rho(9_{46},P_\alpha)=0,
$$
$$
\rho(M_R,P_\beta)=\rho(9_{46},P_\beta)+\rho_0(T*)=\rho_0(T^*)
$$
$$
\rho(M_R,P_0)=\rho^1(R).
$$
In Case I, $\rho_0(T^*)=0$. Thus in Case I
$$
\rho(M_R,\phi_R)\in \{0,0,0,\rho^1(R)\}
$$
according to the $4$ possibilities for $P$. In the Case II, $\rho^1(9_{46})=0$ so
$$
\rho^1(R)=\rho^1(R_{II})=\rho^1(9_{46})+\rho_0(T^*).
$$
\noindent Thus in all cases we can say that
$$
\rho(M_R,\phi_R)\in \{0,\rho^1(R)\}.
$$
\noindent Combining this with ~\ref{eq:second} we have
$$
(\epsilon+C^1)\rho_0(K^1)+\sum_{j>1}C^j\rho_0(K^j)=C_0\rho^1(R)
$$
where $C_0\in \{0,1\}$. Since $\epsilon+C^1>0$  we have expressed a non-trivial  linear combination of $\{\rho_0(K^j)\}$ as a multiple of $\rho^1(R)$ contradicting our choice of $\{K^j\}$.

This contradiction finishes the proof of Theorem~\ref{thm:infgen} modulo the proof of Proposition~\ref{prop:familyofmickeys}.
\end{proof}

\begin{proof}[Proof of Proposition~\ref{prop:familyofmickeys}] We give a recursive definition of $W_i$.

First we define $W_0$. Let $V$ be a rational $(n.5)$-solution for $\tilde{J}$. Let $C$ be the standard cobordism from $M_{\tilde J}$ to the disjoint union of $m_j$ copies of $M_n^j$. Specifically
$$
\partial C= -M_{\tilde{J}}\coprod_jm_jM_n^j,
$$
where if $m_j<0$ we mean $|m_j|$ copies of $-M_n^j$. This cobordism is discussed in detail in ~\cite[p.113-116]{COT2}. Alternatively, note that $\tilde{J}$ may be constructed by starting from $J_n^1$ and infecting along different meridians a total of $((m_1-1)+\sum _{j=2}|m_j|)$ times using the knot sign$(m_j)J_n^j$ a total of $|m_j|$ times ($m_1-1$ times if $j=1$). Thus $C$ can be viewed as an example of the cobordism $E$ defined in Figure~\ref{fig:mickey}. Identify $C$ with $V$ along $M_{\tilde{J}}$.  Then cap off all of its boundary components except one copy of $M_n^1$ using copies of the special $(n)$-solutions $\pm Z_n^j$ as provided by Theorem~\ref{thm:nsolvable}. The latter shall be called \textbf{$\mathcal{Z}$-caps}. Here there is a technical point concerning orientations: if $m_j>0$ then to the boundary component $M_n^j$ we must glue a copy of $-Z^j_n$ (and vice-versa). It was important in the proof that we remember that since $m_1>0$, all the occurrences of $j=1$ $\mathcal{Z}$-caps are copies of $-Z_n^1$ rather than $Z_n^1$. Let the result be denoted $W_0$ as shown schematically in Figure~\ref{fig:w0}.
\begin{figure}[htbp]
\setlength{\unitlength}{1pt}
\begin{picture}(108,106)
\put(0,0){\includegraphics{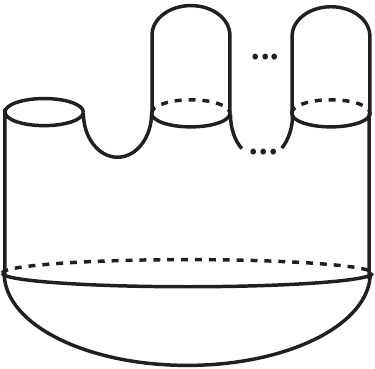}}
\put(50,42){$C$}
\put(50,10){$V$}
\put(89,85){$Z_{n}^{j_k}$}
\put(48,85){$Z_{n}^{j_1}$}
\put(-28,24){$M_{\tilde{J}}$}
\put(-10,27){\vector(1,0){9}}
\put(-28,71){$M_{n}^{1}$}
\put(-10,74){\vector(1,0){9}}
\end{picture}
\caption{$W_{0}$}\label{fig:w0}
\end{figure}

We will now show that $W_0$ is a rational $(n)$-solution for $M_n^1$ (hence a rational $(n)$-bordism). Since $V$ is a rational $(n.5)$-solution for $M_{\tilde J}$, the inclusion-induced map
$$
j_*:~H_1(M_{\tilde J};\mathbb{Q})\to H_1(V;\mathbb{Q})
$$
is an isomorphism. It follows from duality that
$$
j_*:~H_2(M_{\tilde J};\mathbb{Q})\to H_2(V;\mathbb{Q})
$$
is the zero map. Therefore if we examine the Mayer-Vietoris sequence with $\mathbb{Q}$-coefficients,
$$
H_2(M_{\tilde{J}}) \overset{i_*}{\lra} H_2(C)\oplus H_2(V)\overset{\pi_*}{\lra} H_2(C\cup V)\to H_1(M_{\tilde J})\overset{(i_*,j_*)}{\longrightarrow}H_1(C)\oplus H_1(V),
$$
we see that $\pi_*$ induces an isomorphism
$$
(H_2(C)/i_*(H_2(M_{\tilde J})))\oplus H_2(V)\cong H_2(C\cup V).
$$
The integral homology of $C$ was analyzed in ~\cite[p. 113-114]{COT2} and also in Lemma~\ref{lem:mickeyfacts}. From the latter we know that $H_1(C;\mathbb{Q})\cong \mathbb{Q}$, generated by any one of the meridians of any of the knots, and that $H_2(C;\mathbb{Q})$ is $\oplus_j H_2(M^j_n;\mathbb{Q})^{|m_j|}$. In particular $H_2(C)$ arises from its ``top'' boundary. Also the generator of $i_*(H_2(M_{\tilde J}))$ is merely the sum of the generators of the $H_2(M_n^j;\mathbb{Q})$ summands. Thus
$$
H_2(C\cup V;\mathbb{Q})\cong (\mathbb{Q}^m/<1,...,1>) \oplus H_2(V;\mathbb{Q})
$$
where $m=\sum_j |m_j|$ and the generators of the $\mathbb{Q}^m$ come from the ``top'' boundary components of $C$. Moreover $H_1(C\cup V;\mathbb{Q})\cong \mathbb{Q}$ generated by any one of the meridians. Since the $Z_n^j$ are $(n)$-solutions, $H_1(M_n^j)\to H_1(Z_n^j)$ is an isomorphism and by duality $H_2(M_n^j)\to H_2(Z_n^j)$ is the zero map. Thus adding a $\mathcal{Z}$-cap to $C\cup V$ has no effect on $H_1$; while the effect on $H_2$ of adding a $\mathcal{Z}$-cap to $C\cup V$ is to kill the class carried by $\partial Z_n^j=M^j_n$ and to add $H_2(Z_n^j)$. Thus combining these facts we have that
$$
H_2(W_0;\mathbb{Q})/j_*(H_2(\partial W_0;\mathbb{Q}))\cong H_2(W_0;\mathbb{Q})\cong H_2(V;\mathbb{Q})\oplus_{\mathbb{\mathcal{Z}}-caps}H_2(Z_n^j).
$$
and
$$
H_1(W_0;\mathbb{Q})\cong H_1(M_n^1;\mathbb{Q})\cong \mathbb{Q}.
$$
Now, continuing with the verification that $W_0$ is a rational $(n)$-solution, recall that $V$ is a rational $(n.5)$-solution hence a rational $(n)$-solution. Let $\{\ell_1,\dots,\ell_g\}$ be a collection of $n$-surfaces generating a rational $n$-Lagrangian for $V$ and $\{d_1,\dots,d_g\}$ be a collection of $(n)$-surfaces that are the rational $(n+1)$-duals. Since $\pi_1(V)^{(n)}$ maps into $\pi_1(W_0)^{(n)}$, these surfaces are also $(n)$-surfaces for $W_0$. Each $\mathcal{Z}$-cap $Z_n^j$ is a rational $(n)$-solution, so let $\{\ell_1',\dots,\ell_{g'}'\}$ and $\{d_1',\dots,d_{g'}'\}$ denote collections of $(n)$-surfaces generating a rational $(n)$-Lagrangian and rational $(n)$-duals for $Z_n^j$. These surfaces are also $(n)$-surfaces for $W_0$.
By our analysis of $H_2(W_0)$ above, the unions of these collections, for $V$ and for each $\mathcal{Z}$-cap, have the required cardinality to generate a rational $(n)$-Lagrangian with rational $(n)$-duals for $W_0$.
By naturality of the intersection form with twisted coefficients these surfaces also have the required intersection properties to generate a rational $n$-Lagrangian with rational $(n)$-duals for $W_0$. Hence $W_0$ is in fact a rational $(n)$-solution for $M_n^1$ as was claimed. This establishes property $1$ of Proposition~\ref{prop:familyofmickeys} for $W_0$.

In the case $i=0$ , property $2$ is merely the statement that the inclusion $H_1(M_n^1;\mathbb{Q})\to H_1(W_0;\mathbb{Q})$ is injective, which we have already observed is true.

The proof of Proposition~\ref{prop:familyofmickeys} is easier if we inductively prove a more robust version of property $3$.
\vspace{.25in}

\noindent\textbf{Property $3^\prime$}: $W_0\subset W_i$ and for any PTFA coefficient system $\phi:\pi_1(W_i)\to\G$ with $\G^{(n+1)}=1$
$$
\rho(\partial W_i,\phi)=\sigma^{(2)}(W_i,\phi)-\sigma(W_i)=-\sum_{\mathcal{Z}-caps}sign(|m_j|)(\sigma^{(2)}(Z_n^j,\phi_j)-\sigma(Z_n^j))
=-\sum_{j}C^j\rho_0(K^j)
$$
for some integers $C^j$ (depending on $\phi$) where $C^1\geq 0$.

\vspace{.25in}

Theorem~\ref{thm:nsolvable} establishes the last equality in property $3'$ since $m_1>0$. Since $W_0=V\cup C\cup \mathcal{Z}$-caps, by additivity of signatures, property $3'$ will hold for $W_0$ if the difference of signatures vanishes for $V$ and for $C$. Since $V$ is an $(n.5)$-solution, the vanishing for $V$ follows directly from Theorem~\ref{thm:sliceobstr}. The vanishing of the signatures for $C$ follows from Lemma~\ref{lem:mickeysig} (see also ~\cite[Lemma 4.2]{COT2}). Thus we have constructed $W_0$ and Proposition~\ref{prop:familyofmickeys} holds for $i=0$.

Now assume that $W_i$, $i\leq n-1$, has been constructed satisfying the properties above. Before defining $W_{i+1}$ we derive some important facts about $W_i$.  By property $(1)$, $\partial W_i=M_{n-i}^1$ and $W_i$ is a rational $(n)$-bordism and hence a rational $(i+1)$-bordism since $i+1\leq n$. Let $\pi=\pi_1(W_i)$, $\Lambda=\pi/\pi^{(i+1)}_r$ and let $\phi:\pi\to \Lambda$ be the canonical surjection. We will apply Theorem~\ref{thm:selfannihil} to $W_i$ with $k=i+1$. Property $2$ for $W_i$ ensures that $\phi$ restricted to $M_{n-i}^1$ is non-trivial. We conclude that the kernel, $\bar P$ of the composition
$$
H_1(M_{n-i}^1;\mathbb{Q}\Lambda)\overset{j_*}\to H_1(W_i;\mathbb{Q}\Lambda),
$$
satisfies $\bar P\subset \bar P^\perp$ with respect to the  Blanchfield form $\mathcal{B}\ell_\Lambda$ on $M^1_{n-i}$. But property $2$ also ensures that $\phi$ restricted to $\pi_1(M_{n-i}^1)$ factors through the abelianization. Hence by Theorem~\ref{thm:nontriviality} the kernel $P$ of the composition
$$
\mathcal{A}_0(J^1_{n-i})\overset{i}{\hookrightarrow}\mathcal{A}_0(J^1_{n-i}) \otimes\mathbb{Q}\Lambda\overset{\cong}{\to} H_1(M^1_{n-i};\mathbb{Q}\Lambda)\overset{j_*}\to H_1(W_i;\mathbb{Q}\Lambda),
$$
satisfies $P\subset P^\perp$ with respect to the classical Blanchfield form on $J^1_{n-i}$. Recall that, by definition, $J^1_{n-i}$ is obtained from $R$ by two infections along the circles labelled $\alpha$ and $\beta$ as in Figure~\ref{fig:exs_eta2}. These two circles form a generating set $\{[\alpha],[\beta]\}$ for $\mathcal{A}_0(J^j_{n-i})$ (which is isomorphic to $\mathcal{A}_0(R)$ and hence nontrivial). From this we can conclude that at least one of these generators is \emph{not} in $P$ since the classical Blanchfield form of any knot is nonsingular. Now consider the commutative diagram below. Recall that $H_1(W_i;\mathbb{Q}\Lambda)$ is identifiable as the ordinary rational homology of the covering space of $W$ whose fundamental group is the kernel of $\phi:\pi\to \Lambda$. Since this kernel is precisely $\pi^{(i+1)}_r$, we have that
$$
H_1(W_i;\mathbb{Q}\Lambda)\cong (\pi^{(i+1)}_r/[\pi^{(i+1)}_r,\pi^{(i+1)}_r])\otimes_\mathbb{Z} \mathbb{Q}
$$
as indicated in the diagram below. The vertical map $j$ is injective as shown in Section~\ref{J2}. Furthermore, since $\{\alpha,\beta\}\subset \pi_1(M^1_{n-i})^{(1)}$,  by property $2$
$$
\pi_1(M^1_{n-i})^{(1)}\subset \pi^{(i+1)}.
$$
Since the composition in the bottom row sends one of $\{[\alpha],[\beta]\}$ to non-zero, the composition in the top row sends at least one of $\{\alpha,\beta\} $ to non-zero.
\begin{equation*}
\begin{CD}
\pi_1(M^1_{n-i})^{(1)}      @>\cong>>    \pi_1(M^1_{n-i})^{(1)}  @>j_*>>   \pi^{(i+1)}  @>>>
\pi^{(i+1)}_r/\pi^{(i+2)}_r \\
  @VVV   @VVV        @VVV       @VVjV\\
\mathcal{A}_0(J^1_{n-i})     @>i>>  H_1(M^1_{n-i};\mathbb{Q}\Lambda)    @>j_*>> H_1(W_i;\mathbb{Q}\Lambda) @>\cong>>
  (\pi^{(i+1)}_r/[\pi^{(i+1)}_r,\pi^{(i+1)}_r])\otimes_\mathbb{Z} \mathbb{Q}\\
\end{CD}
\end{equation*}

Therefore we have established these crucial facts about $W_i$:
\begin{itemize}
\item [\textbf{Fact $1$}:] Each of $\{\alpha,\beta\}$ maps into $\pi_1(W_i)^{(i+1)}$
\item [\textbf{Fact $2$}:] The kernel, $\tilde{P}$,  of the the composition (top row of the diagram above)
$$
\pi_1(M^1_{n-i})^{(1)}\to \pi_1(W_i)^{(i+1)}_r/\pi_1(W_i)^{(i+2)}_r
$$
is of the form $\pi^{-1}(P)$ for some submodule $P\subset \mathcal{A}_0(J^1_{n-i})$ such that $P\subset P^\perp$ with respect to the classical Blanchfield form and at least one of $\{\alpha,\beta\}$ maps non-trivially under this map.
\end{itemize}
Now we claim further that
\begin{itemize}
\item [\textbf{Fact $3$}:] If $i\leq n-2$, precisely one of $\{\alpha,\beta\}$ maps non-trivially under the map in Fact $2$.
\item [\textbf{Fact $4$}:] If $i\leq n-2$, without loss of generality we may assume that $\beta$ maps non-trivially and $\alpha$ maps trivially under the above map.
\end{itemize}
To establish Facts $3$ and $4$ assume $i\leq n-2$ and consider the coefficient system
$$
\phi:\pi_1(M_{n-i}^1)\to \G=\pi_1(W_i)/\pi_1(W_i)_r^{(i+2)}.
$$
Note that $\G^{(n+1)}=1$ since $i+2\leq n+1$. By property $3'$ for $W_i$,
$$
\rho(M_{n-i}^1,\phi)=\sigma^{(2)}(W_i,\phi)-\sigma(W_i)=\sum_{\mathcal{Z}-caps}\pm(\sigma^{(2)}(Z_n^j,\phi_j)-\sigma(Z_n^j)).
$$
But the $Z_n^j$ are $(n)$-solutions and thus are $(i+1.5)$-solutions since $i+1.5\leq n$. Hence, by Theorem~\ref{thm:sliceobstr}, all these signature defects are zero. Thus
$$
\rho(M_{n-i}^1,\phi)=0.
$$
Moreover by property $(2)$ for $W_i$, $\phi(\pi_1(M_{n-i}^1))\subset \pi_1(W_i)^{(i)}$. Therefore $\phi$ restricted to $\pi_1(M^1_{n-i})$ factors through $\pi_1(M^1_{n-1})/\pi_1(M^1_{n-1})^{(2)}$. In other words (using Fact 2 above) one of the first-order signatures of $J^1_{n-i}$ is zero. Assume that both $\alpha$ and $\beta$ mapped nontrivially. Since the Alexander module of $J^1_{n-i}$ is isomorphic to that of the $9_{46}$ knot, by Example~\ref{ex:first-ordersigs}, this first-order signature would necessarily be what we have called $\rho^1(J^1_{n-1})$. Since $J^1_{n-1}$ is obtained from $R$ by two infections using $J^1_{n-i-1}$ as the infecting knot, as in Example~\ref{ex:first-ordersigs},
$$
0=\rho(M_{n-i}^1,\phi)=\rho^1(J^1_{n-1})=\rho^1(R)+\rho_0(J^1_{n-i-1})+\rho_0(J^1_{n-i-1}).
$$
However, by choice $\rho^1(R)\neq 0$ and, since $n-i-1\geq 1$, $J^1_{n-i-1}$ is $(0.5)$-solvable by Theorem~\ref{thm:nsolvable} and so $\rho_0(J^1_{n-i-1})=0$ by Theorem~\ref{thm:sliceobstr}. This contradiction implies Fact $3$. To show Fact $4$, suppose $R=R_{II}$, $\alpha$ maps nontrivially and $\beta$ maps trivially. Viewing $J^1_{n-1}$ as obtained from $9_{46}$ by two infections using $J^1_{n-i-1}$ as one infecting knot and $T\#~J^1_{n-i-1}$ as the other, then by
Example~\ref{ex:first-ordersigs}
$$
0=\rho(M_{n-i}^1,\phi)=\rho_0(J^1_{n-i-1})+\rho_0(T\#~J^1_{n-i-1})=\rho_0(T)\neq 0.
$$
This contradiction implies Fact $4$ in this case. In the case that $R=R_I$ then  $J^1_{n-1}$ is symmetric with respect to $\alpha$ and $\beta$ so we may assume Fact $4$ by relabelling if necessary.

Finally we can give the construction of $W_{i+1}$. Refer to Figure~\ref{fig:Wi}. Since $J^1_{n-1}$ is obtained from $R$ by two infections using $J^1_{n-i-1}$ as the infecting knot, there is a corresponding cobordism $E$ with $4$ boundary components $M_{n-i}$, $M_{R}$ and two copies of $M^1_{n-i-1}$. Glue this to $W_i$ along $M_{n-i}^1$, as in the top-most portion of Figure~\ref{fig:Wi}. If $i=n-1$ then we set $W_n=W_{n-1}\cup E$ and we are done. Note that $\partial W_n$ consists of two copies of $M^1_0$ and one copy of $M_R$ as required by property $1$. Now consider the case that $i\leq n-2$ (in which case Fact $4$ holds). Notice that $R$ is a ribbon knot and admits a ribbon disk $\Delta$ that is obtained by ``cutting the $\alpha$ band''. Let $\mathcal{R}$ denote the exterior in $B^4$ of this ribbon disk.
\begin{figure}[htbp]
\setlength{\unitlength}{1pt}
\begin{picture}(363,263)
\put(1,0){\includegraphics{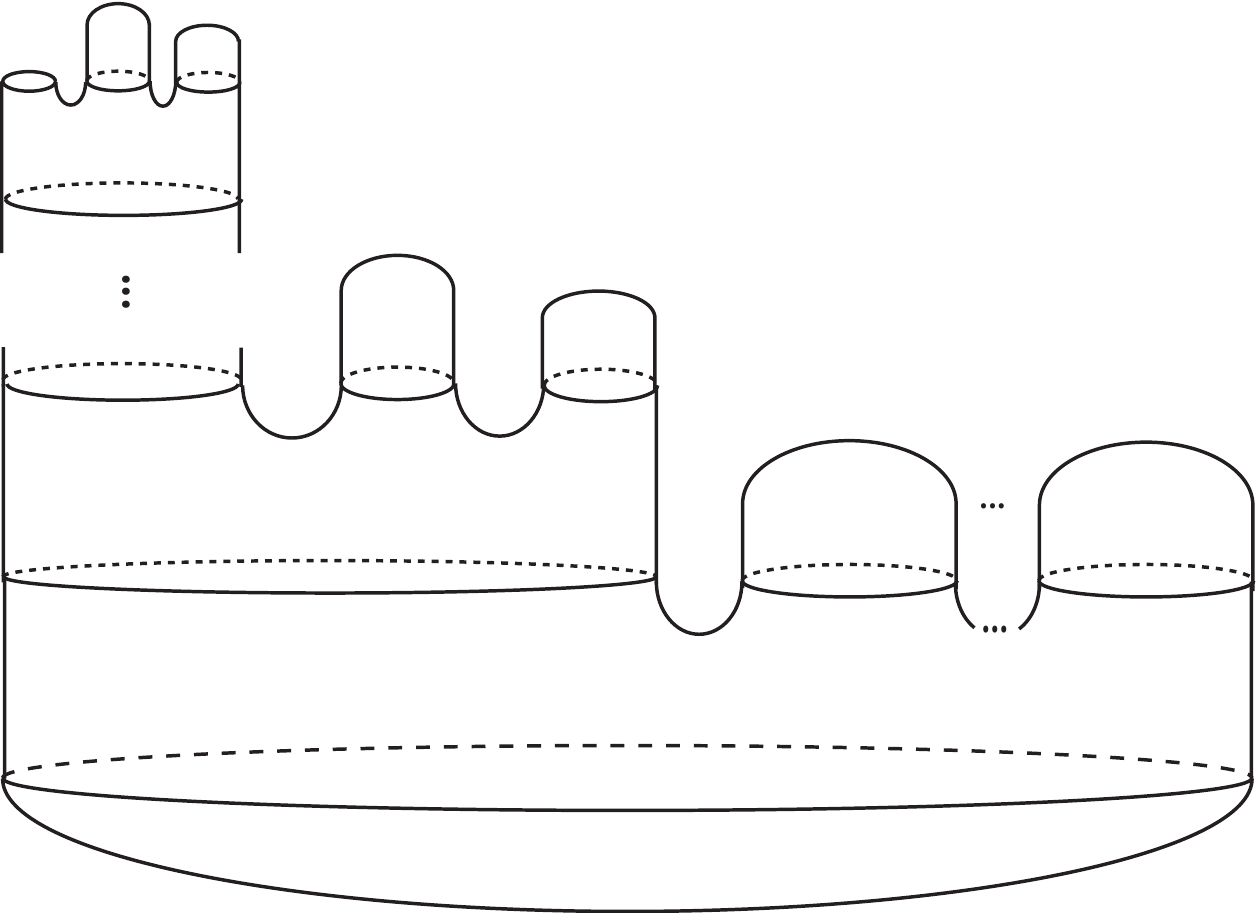}}
\put(183,14){$V$}
\put(183,61){$C$}
\put(99,115){$E$}
\put(241,115){$Z_{n}^{j_{1}}$}
\put(325,115){$Z_{n}^{j_{k}}$}
\put(110,167){$\mathcal{N}_{0}$}
\put(170,163){$\mathcal{R}$}
\put(57,246){$\mathcal{R}$}
\put(30,248){$\mathcal{N}_{i}$}
\put(-28,35){$M_{\tilde{J}}$}
\put(-12,38){\vector(1,0){12}}
\put(-30,93){$M_{n}^{1}$}
\put(-12,97){\vector(1,0){12}}
\put(-30,149){$M_{n-1}^{1}$}
\put(-12,153){\vector(1,0){12}}
\put(-35,236){$M_{n-i-1}^{1}$}
\put(-17,240){\vector(1,0){17}}
\put(-30,202){$M_{n-i}^{1}$}
\put(-12,206){\vector(1,0){12}}
\put(31,218){$E$}
\end{picture}
\caption{$W_{i+1}$}\label{fig:Wi}
\end{figure}
\noindent Cap off the $M_R$ boundary component of $W_i\cup E$ using $\mathcal{R}$. This will be called an $\mathcal{R}$-cap. Recall that $M^1_{n-i-1}$ is the boundary of a special $(n-i-1)$-solution $Z^1_{n-i-1}$ as in Theorem~\ref{thm:nsolvable}. Use this to cap off the $M^1_{n-i-1}$ boundary component of $W_i\cup E\cup \mathcal{R}$ that corresponds to the infection along $\alpha$. Call this a \textbf{null-cap} and denote it by $\mathcal{N}_{i}$. The resulting manifold is $W_{i+1}$. Note that $\partial W_{i+1}=M^1_{n-i-1}$ as required by property $1$ of Proposition~\ref{prop:familyofmickeys}. This completes the definition of $W_{i+1}$.

Now we set about verifying the properties $(1)-(3')$ for $W_{i+1}$.

\textbf{Property (1)}: \textbf{$W_{i+1}$ is a rational  $(n)$-bordism}.

This will follow from an inductive analysis of $H_2(W_{i+1};\mathbb{Q})$. For the following argument we assume $\mathbb{Q}$ coefficients unless specified. We establish that $H_2(W_{i+1};\mathbb{Q})/I_0$ comes from the second homology of $V$, the $\mathcal{Z}-caps$ and the null caps.

\begin{lem}\label{lem:H2W} For each $i\leq n-1$
$$
H_2(W_{i+1})\cong H_2(V)\oplus_{\mathbb{\mathcal{Z}}-caps}H_2(Z_n^j)\oplus_{j=0}^{i}H_2(\mathcal{N}_j)\oplus H_2(\partial W_{i+1}).
$$
\end{lem}
\begin{proof}[Proof of Lemma~\ref{lem:H2W}] The proof is by induction. Recall that we have already established this for $W_0$. Assume it is true for $W_i$:
$$
H_2(W_{i})\cong H_2(V)\oplus_{\mathbb{\mathcal{Z}}-caps}H_2(Z_n^j)\oplus_{j=0}^{i-1}H_2(\mathcal{N}_j)\oplus H_2(M^1_{n-i}).
$$
In the passage from $W_i$ to $W_{i+1}$ the first step was to adjoin $E$ along $M^1_{n-i}$. Consider the sequence
$$
H_2(M^1_{n-i})\to H_2(E)\oplus H_2(W_{i})\overset{\pi_*}{\to} H_2(E\cup W_{i})\to H_1(M^1_{n-i})\overset{(i_*,j_*)}{\longrightarrow}H_1(E)\oplus H_1(W_{i}).
$$
Recall that we have analyzed the homology of $E$ in Lemma~\ref{lem:mickeyfacts} and found that $i_*$ is an isomorphism on $H_1$ (so $\pi_*$ above is onto); and that $H_2(E)\cong \mathbb{Q}^3$, with basis consisting of generators for the two copies of $H_2(M^1_{n-i-1})$ and one from either $H_2(M_{R})$ or $H_2(M^1_{n-i})$ (suitable generators for the latter become equated in $H_2(E)$). Thus
\begin{equation}\label{eq:zcap}
H_2(W_i\cup E )\cong H_2(V)\oplus_{\mathbb{\mathcal{Z}}-caps}H_2(Z_n^j)\oplus_{j=0}^{i-1}H_2(\mathcal{N}_j)\oplus H_2(M_R)\oplus H_2(M^1_{n-i-1}) \oplus H_2(M^1_{n-i-1}).
\end{equation}
If $i=n-1$, then $W_i\cup E\cong W_{i+1}$ so ~\ref{eq:zcap} implies Lemma~\ref{lem:H2W}. Thus we may suppose that $i\leq n-2$. Recalling Remark~\ref{rem:Rfacts}, note that both $\mathcal{R}$ and the null-cap $\mathcal{N}_{i}$ have the property that the inclusion map from their boundary induces an isomorphism on $H_1$ and induces the zero map on $H_2$. Thus, as we saw in the analysis of the homology of $W_0$, the effect on $H_2$ of adding these is to kill the generators corresponding to their boundaries and to add $H_2(\mathcal{N}_{i})$ and $H_2(\mathcal{R})$. Combining these facts we have established the Lemma for $i+1$, finishing the inductive proof of Lemma~\ref{lem:H2W}.
\end{proof}

We continue with the verification that $W_{i+1}$ is a rational $(n)$-bordism. Recall that $V$ and the $\mathcal{Z}$-caps $Z^j_n$ are rational $(n)$-solutions and that the null caps $\{\mathcal{N}_0,..,\mathcal{N}_{i}\}$ are copies of $\{Z^1_{n-1},...,Z^1_{n-i-1}\}$ which are, respectively, $(n-1),...,(n-i-1)$-solutions. Taking the union of their respective Lagrangians and duals gives collections that have the required \emph{cardinality}, by Lemma~\ref{lem:H2W} above, to generate a rational $n$-Lagrangian with rational $(n)$-duals for $W_{i+1}$. We must first verify that all these surfaces are indeed $n$-surfaces for $W_{i+1}$. This is immediate for those arising from the rational $(n)$-solutions but we must check the case of the null caps. By induction the null caps at level less than $i+1$ were already part of $W_i$ and their Lagrangians and duals were already checked to be $(n)$-surfaces for $W_i$ and hence they will be for $W_{i+1}$. Thus we need only consider the $(n-i-1)$ Lagrangian and duals for the $(n-i-1)$-solution $\mathcal{N}_{i}$. A null-cap $\mathcal{N}_{i}$ exists only in case $i\leq n-2$ (by construction). Recall that
$$
\pi_1(M^1_{n-i-1})\to\pi_1(\mathcal{N}_{i})
$$
is surjective, by Theorem~\ref{thm:nsolvable}, and $\pi_1(M^1_{n-i-1})$ is normally generated by its meridian. This meridian is isotopic in $E$ to a push-off of $\alpha$ in $M^1_{n-i}$ (by property $4$ of Lemma~\ref{lem:mickeyfacts}). But by Fact $1$ above, $\alpha$ maps into $\pi_1(W_i)^{(i+1)}$. Thus \emph{any} element of $\pi_1(\mathcal{N}_{i})$ lies in $\pi_1(W_{i+1})^{(i+1)}$ so
$$
\pi_1(\mathcal{N}_{i})^{(n-i-1)}\subset\pi_1(W_{i+1})^{(n)}.
$$
Therefore the $(n-i-1)$ Lagrangian and duals for $\mathcal{N}_{i}$ are actually $(n)$-surfaces for $W_{i+1}$. Moreover by Fact $4$ above, $\alpha$ maps into $\pi_1(W_i)^{(i+2)}_r$. Thus \\

\textbf{Fact 5:} $\pi_1(\mathcal{N}_{i})^{(n-i-1)}\subset\pi_1(W_{i+1})^{(n+1)}_r$,
\\

\noindent a fact that we record for later use. Again, by naturality of the intersection form, the union of the surfaces generating the Lagrangians and duals for $V$ and all the caps have the required intersection properties to generate a rational $n$-Lagrangian with rational $(n)$-duals for $W_{i+1}$. Thus $W_{i+1}$ is a rational $(n)$-bordism as claimed.

This completes the verification of the property $(1)$ of Proposition~\ref{prop:familyofmickeys} for $W_{i+1}$.

\textbf{Property ($2$) for $W_{i+1}$}:

Consider a component $M^1_{n-i-1}$ of $\partial W_{i+1}$. Recall that $\pi_1(M^1_{n-i-1})$ is normally generated by the meridian and this meridian is isotopic in $E$ to a push-off $\beta$ (if $i=n-1$ it could be either $\beta$ or $\alpha$) in $M^1_{n-i}=\partial W_i$. Since $\beta$ (and $\alpha$) lies in the commutator subgroup of $\pi_1(M^1_{n-i})$,
$$
j_*(\beta)\in \pi_1(W_i)^{(i+1)}
$$
by property $(2)$ for $W_{i}$ (similarly for $\alpha$). Thus
$$
j_*(\pi_1(M^1_{n-i-1}))\subset \pi_1(W_{i+1})^{(i+1)}
$$
establishing the first part of property $(2)$ for $W_{i+1}$. To prove the second part we need to show that
$j_*(\beta)$ (for $i=n-1$ one of $j_*(\beta)$ or $j_*(\alpha)$) is non-zero in $\pi_1(W_{i+1})^{(i+1)}/\pi_1(W_{i+1})^{(i+2)}_r$. Fact $4$ (if $i=n-1$ use Fact $2$) provides precisely this except for $\pi_1(W_i)$ instead of $\pi_1(W_{i+1})$. Thus it suffices to show that inclusion induces an isomorphism
\begin{equation}\label{eq:isomorph}
\pi_1(W_{i})/\pi_1(W_{i})^{(i+2)}_r\cong \pi_1(W_{i+1})/\pi_1(W_{i+1})^{(i+2)}_r.
\end{equation}
The map $\pi_1(W_i))\to \pi_1(W_i\cup E)$ is a surjection whose kernel is the normal closure of the longitude $\ell$ of the copy of $S^3-J^1_{n-i-1}\subset M^1_{n-i}$ (property $(1)$ of Lemma~\ref{lem:mickeyfacts}). The group $\pi_1(S^3-J^1_{n-i-1})$ is normally generated by the meridian of this copy of $S^3-J^1_{n-i-1}$. This meridian is identified to a push-off of the curve $\alpha$ and we have seen that $j_*(\alpha)\in \pi_1(W_i)^{(i+2)}$. Thus $j_*(\ell)\in \pi_1(W_i)^{(i+2)}_r$ and so the inclusion map $W_i\to W_i\cup E$ induces an isomorphism on $\pi_1$ modulo $\pi_1(-)^{(i+2)}_r$. If $i=n-1$ this establishes ~\ref{eq:isomorph}. Now suppose $i\leq n-2$. Similarly the map $\pi_1(W_i))\to \pi_1(W_i\cup E\cup \mathcal{R})$ is a surjection whose kernel is the normal closure of the curve $\alpha\subset M_R$ (by property $(1)$ above for $\mathcal{R}$). But this curve $\alpha$ is isotopic in $E$ to the curve $\alpha\subset M^1_{n-i}$ (by Lemma~\ref{lem:mickeyfacts}) and
$j_*(\alpha)\in \pi_1(W_i)^{(i+2)}$. Thus inclusion $W_i\to W_i\cup E\cup \mathcal{\mathcal{R}}$ induces an isomorphism on $\pi_1$ modulo $\pi_1(-)^{(i+2)}_r$. Finally, the same type of argument applies to $\mathcal{N}_{i}$ using property $(1)$ of Theorem~\ref{thm:nsolvable}, that $\pi_1(M^1_{n-i-1})$ is normally generated by its meridian and that his meridian is isotopic in $E$ to a push-off of $\alpha$ in $M^1_{n-i}$. This completes the verification of property ($2$) for $W_{i+1}$.

\textbf{Property ($3'$) for $W_{i+1}$}:

Since, for $i\neq n-1$, $W_{i+1}=W_i\cup E\cup \mathcal{R}\cup \mathcal{N}_{i}$, and $W_n=W_{n-1}\cup E$, and property $3'$ holds for $W_i$ (using the induced coefficient system), it will suffice to prove that the signature defect is zero on $E$, $\mathcal{R}$ and $\mathcal{N}_{i}$. The first is given by Lemma~\ref{lem:mickeysig} and the second holds since $\mathcal{R}$ is a slice disk complement hence an $(n.5)$-solution. Finally, we established above in Fact $5$ that
$$
j_*(\pi_1(\mathcal{N}_{i}))\subset\pi_1(W_{i+1})^{(i+2)}_r
$$
so the coefficient system induced on $\mathcal{N}_{i}$ by $\phi:\pi_1(W_{i+1})\to\G$ is trivial since $\G^{(n+2)}_r=1$. This concludes the verification of property ($3'$) for $W_{i+1}$.

This concludes the inductive proof of Proposition~\ref{prop:familyofmickeys}.
\end{proof}

\begin{ex}\label{ex:concretefamily} The families of Figure~\ref{fig:familyJ_n^j} have the disadvantage that we are unable to specify $T^*$ due to our inability to calculate $\rho^1(9_{46})$. The family of knots, $J_n^j$ of Figure~\ref{fig:betterfamilyJ_n^j} (ignore the dotted arc) overcomes this problem, giving a specific infinite family of $(n)$-solvable knots that is linearly independent modulo $\FF_{n.5}$. Here $T$ is the right-hand trefoil knot and $J_0^j=K^j$ is the family of knots used in Step $2$ of the above proof. Each $J^j_n$ ($n>0$) is obtained by two infections on the $8_9$ knot, which is itself a ribbon knot (a ribbon move is shown by the dotted arc) ~\cite{Lam}.
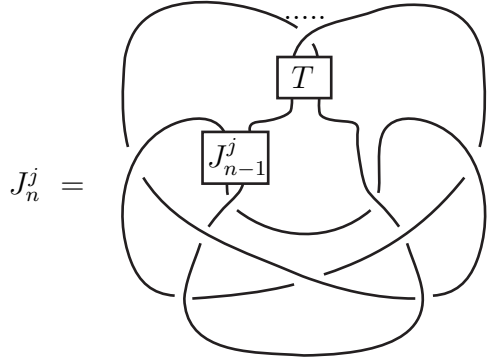
\begin{figure}[htbp]
\begin{picture}(160,140)
\put(13,0){\includegraphics{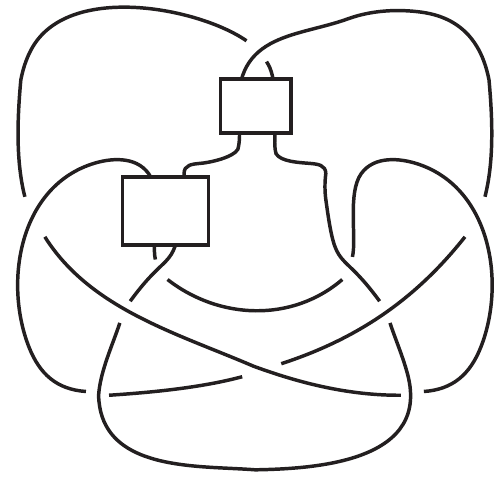}}
\put(79,136){$.....$}
\put(82,110){$T$}
\put(50,80){$J_{n-1}^j$}
\put(-25,70){$J_{n}^j~~=$}
\end{picture}
\caption{A family of $(n)$-solvable knots $J^j_n$}\label{fig:betterfamilyJ_n^j}
\end{figure}

The proof that no linear combination is rationally $(n.5)$-solvable is the same as that above, but for this family there are several major simplifications. Let $R$ be the knot obtained in Figure~\ref{fig:betterfamilyJ_n^j} by setting $J_{n-1}^j=U$. Then $R$ is a ribbon knot (a ribbon move is again shown by the dotted arc). For any $j$, $J_{n}^j$ obtained from $R$ by an infection using $J_{n-1}^j$ so inductively is $(n)$-solvable by Theorem~\ref{thm:nsolvable}. Hence $J_{n}^j\in \FF_n$. Moreover $\rho^1(R)=\rho_0(T)\neq 0$ by the calculations of ~\cite[Examples 4.4, 4.6]{CHL4}, and so $R$ satisfies Step $1$ of the above proof. Moreover $J_n^j$ is obtained from $R$ by a single infection on a curve, $\alpha$, that \textbf{generates} the cyclic module $\mathcal{A}_0(R)$, so $\alpha$ does not lie in \emph{any} submodule $P$ where $P\subset P^{\perp}$. This eliminates the various dichotomies between $\alpha$ and $\beta$ in Step $4$ of the above proof. These observations simplify the flow of Step $4$ of the above proof.
\end{ex}

\vspace{.25in}

 \section{The family $J_n$}\label{familyJ}

In this section we prove our second main theorem which shows that the family of knots $J_n$ of Figure~\ref{fig:family} contains many non-slice knots, even though, for each $n>1$, all classical invariants as well as those of Casson-Gordon vanish for $J_n$. Recall that $J_0=K$ and $J_n=J_n(K)$ is obtained from $J_0$ by applying the ``operator'' $R$ $n$ yielding the inductive definition of Figure~\ref{fig:family}. The proof that we give
actually applies to large classes of knots obtained by $n$-times iterated generalized doubling, which we record in the form of a more general theorem at the end of the section.

The main theorem of this section is:

\begin{thm}\hspace{12pt}\label{thm:main}
\begin{itemize}
\item [1.] There is a constant $C$ such that, if $|\rho_0(K)|> C$, then for each $n\ge0$, $J_n(K)$ is of infinite order in the topological concordance group. Moreover if, additionally, Arf($K)=0$, then $J_n(K)$ is of infinite order in $\FF_n/\FF_{n.5}$.
\item [2.] If $\rho^1(9_{46})\neq 0$ and some $J_n(K)$ is a slice knot (or even rationally $(n.5)$-solvable) then $\rho_0(K)\in \{0,\rho^1(9_{46})\}$.
\end{itemize}
\end{thm}

\begin{rem} We would conjecture that: If $J_n(K)$ is a slice knot then $K$ is algebraically slice. This is unknown even for $n=1$. Part $2$ of the theorem is evidence for this conjecture. We have not been able to calculate the real number $\rho^1(9_{46})$. Recall that for $J_2$ we were able to prove a much stronger theorem, Theorem~\ref{thm:J2notslice}.
\end{rem}

\begin{cor}\label{cor:CG} For any $n\ge 1$ there exist knots $J\in\mathcal{F}_{(n-1)}$ for which the knot $R(J)$, shown in Figure~\ref{fig:family3}, is not a slice knot nor even in $\mathcal{F}_{n.5}$.
\end{cor}

\begin{figure}[htbp]
\setlength{\unitlength}{1pt}
\begin{picture}(200,160)
\put(0,0){\includegraphics[height=150pt]{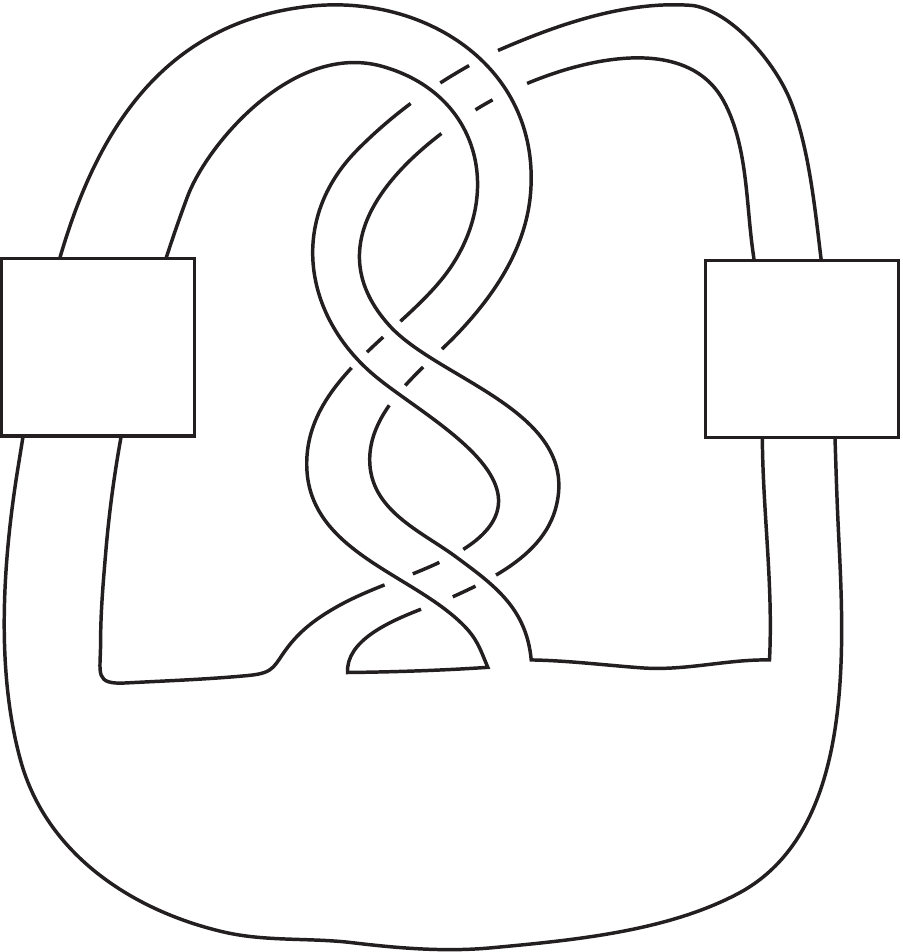}}
\put(-50,70){$R(J)~=$}
\put(9,92){$J$}
\put(124,92){$J$}
\end{picture}
\caption{}\label{fig:family3}
\end{figure}

\begin{proof}[Proof of Corollary~\ref{cor:CG}] Let $J=J_{n-1}(K)$ for some $K$ with $|\rho_0(K)|> C_n$ (for example a connected sum of a suitably large even number of trefoil knots). Then the knot on the right-hand side of Figure~\ref{fig:ribbonCG} is merely $J_n(K)$ which, by Theorem~\ref{thm:main}, is $(n)$-solvable hence in $\mathcal{F}_{(n)}$, but is not slice nor even rationally $(n.5)$-solvable; hence not in $\mathcal{F}_{(n.5)}$. Since $J\in \mathcal{F}_{(n-1)}$, if $n\geq 2$ then $J$ is algebraically slice and if $n\geq 3$ then $J$ has vanishing Casson-Gordon invariants ~\cite[Theorem 9.11]{COT}.
\end{proof}

\begin{proof}[Proof of Theorem \ref{thm:main}] The proof follows closely the lines of the proof of Theorem~\ref{thm:infgen}. Let $R$ be the ribbon knot $9_{46}$ and recall that $J_0=J_0(K)=K$ and that $J_n=J_n(K)$ is obtained from $R$ by infecting twice, along the curves $\alpha$ and $\beta$ (as shown in Figure~\ref{fig:Rdoubling}), using the knot $J_{n-1}$ as the infecting knot in each case, as shown in Figure~\ref{fig:family}. By Theorem~\ref{thm:nsolvable}, $J_n(K)$ is $(n)$-solvable for any Arf invariant zero knot $K$. Let $C$ be the Cheeger-Gromov constant for $M_{R}$. We shall show that if a non-zero multiple of $J_n$ is rationally $(n.5)$-solvable then $|\rho_0(K)|\leq C$. In particular this will demonstrate that if $K$ is chosen so that $|\rho_0(K)|>C$ then $J_n(K)$ is of infinite order in $\FF_n/\FF_{n.5}$ and consequently of infinite order in the smooth and topological concordance groups. We will also show that if $\rho^1(9_{46})\neq 0$ and $J_n$ is rationally $(n.5)$-solvable then $\rho_0(K)\in \{0,\rho^1(9_{46})\}$ (a much stronger result).

Suppose that, for some positive integer $m$, $\tilde J\equiv \#_{i=1}^m J_n$ is rationally $(n.5)$-solvable. Under this assumption we shall construct a family of $4$-manifolds $W_i$, as in the proof of Theorem~\ref{thm:infgen}, and reach the desired results. Again, let $M_{i}$ abbreviate $M_{J_i}$.

\begin{prop}\label{prop:familyofmickeys2} Under the assumption that $\tilde J$ is rationally $(n.5)$-solvable, for each $0\leq i\leq n$ there exists a $4$-manifold $W_i$ with the following properties. Letting $\pi=\pi_1(W_i)$,

\begin{itemize}
\item [(1)] $W_i$ is a rational $(n)$-bordism whose boundary is a disjoint union of $r(i)$ copies of $M_R$ and $r(i)+1$ copies of $M_{n-i}$~;
\item [(2)] Each inclusion $j:M_{n-i}\subset\partial W_i\to W_i$ satisfies
$$
j_*(\pi_1(M_{n-i}))\subset \pi^{(i)};
$$
and
$$
j_*(\pi_1(M_{n-i}))\cong \mathbb{Z}\subset \pi^{(i)}/\pi^{(i+1)}_r;
$$
\item [(3)] For any PTFA coefficient system $\phi:\pi_1(W_i)\to\G$ with $\G^{(n+1)}_r=1$
$$
\rho(\partial W_i,\phi)\equiv\sigma_\G^{(2)}(W_i,\phi)-\sigma(W_i)=-D\rho_0(K)
$$
for some non-negative integer $D$ (depending on $\phi$). If $m=1$ then $D=0$.
\item [(4)] If $\rho^1(9_{46})\neq 0$ then, for $i<n$, $r(i)=0$ whereas $r(n)=0$ or $1$, and if $r(n)=1$ then $\rho(M_R,\pi_1(M_R)\to\pi_1(W_n)\to \pi_1(W_n)/\pi_1(W_n)^{(n+1)}_r)=\rho^1(9_{46})$.
\end{itemize}
\end{prop}

Before proving Proposition~\ref{prop:familyofmickeys2}, we assume it and finish the proof of Theorem~\ref{thm:main}. Consider $W_{n}$ from Proposition~\ref{prop:familyofmickeys2} with boundary $(r(n)+1)M_0 \coprod r(n)M_R$. Recall that $M_0=M_{J_0}=M_{K}$. Let $\pi=\pi_1(W_n)$ and consider $\phi:\pi\to\pi/\pi^{(n+1)}_r$. Let $\phi^R_j$, $1\leq j\leq r(n)$, and $\phi^K_j$, $1\leq j\leq r(n)+1$ denote the restrictions of $\phi$ to the various boundary components of $W_n$. Then by property $(3)$ of Proposition~\ref{prop:familyofmickeys2} for $i=n$
\begin{equation}\label{eq:first}
\sum_{j=1}^{r(n)+1}\rho(M_{K},\phi_j^K)+ \sum_{j=1}^{r(n)}\rho(M_R,\phi^R_j)=-D\rho_0(K)
\end{equation}
where $D\geq 0$. By property $(2)$ of Proposition~\ref{prop:familyofmickeys2}, for each boundary component $M_K$,
$$
j_*(\pi_1(M_{K})\subset \pi^{(n)},
$$
implying that each $\phi_j^K$ factors through the abelianization. Additionally by property $(2)$, each of these coefficient systems is non-trivial. Hence by $(2)-(4)$ of Proposition~\ref{prop:rho invariants}
$$
\rho(M_{K},\phi_j^K)=\rho_0(K)
$$
for each $j$. Thus we can simplify ~\ref{eq:first} yielding
\begin{equation}\label{eq:final}
(r(n)+1+D)\rho_0(K)=-\sum_{j=1}^{r(n)}\rho_0(M_R,\phi^R_j).
\end{equation}
Since $C$ is the Cheeger-Gromov constant of $M_R$, for each $j$
$$
|\rho_0(M_R,\phi^R_j)|\leq C.
$$
Hence
\begin{equation}\label{eq:final2}
|\rho_0(K)|\leq \frac{r(n)}{r(n)+1+D}C\leq C.
\end{equation}
Hence if $|\rho_0(K)|> C$ then $\tilde{J}$ is not is rationally $(n.5)$-solvable, thereby completing the proof of Part $1$ of Theorem~\ref{thm:main}, modulo the proof of Proposition~\ref{prop:familyofmickeys2}.

For Part $2$, specialize to the case that $m=1$ and assume that $\rho^1(9_{46})\neq 0$. Then, by property $4$, either $r(n)=0$ or $r(n)=1$. In the first case, by equation ~\ref{eq:final2}, $\rho(K)=0$. If $r(n)=1$ then, using the last clause of property 4, equation ~\ref{eq:final} becomes
$$
|\rho_0(K)|=\frac{1}{2+D}|\rho(M_R,\phi^R)|=\frac{1}{2+D}|\rho^1(R)|.
$$
Moreover, since $m=1$, $D=0$ by property $3$. This completes the proof Part $2$ of Theorem~\ref{thm:main}, modulo the proof of Proposition~\ref{prop:familyofmickeys2}.
\end{proof}

\begin{proof}[Proof of Proposition~\ref{prop:familyofmickeys2}] We give a recursive definition of $W_i$. First we define $W_0$. This will be identical to a special case of the $W_0$ (where all $m_j=0$ except $m_1$) constructed in the proof of Theorem~\ref{thm:infgen}. Let $V$ be a rational $(n.5)$-solution for $\tilde{J}$. Let $C$ be the standard cobordism from $M_{\tilde J}$ to the disjoint union of $m$ copies of $M_n$. Consider $=C\cup V$ and cap off $m-1$ of its boundary components using copies of the special $(n)$-solutions $-Z_n$ as provided by Theorem~\ref{thm:nsolvable}. These are called \textbf{$\mathcal{Z}$-caps}. Let the result be denoted $W_0$ as shown schematically in Figure~\ref{fig:w0second}. Note that if $m=1$ no $\mathcal{Z}$-caps occur.
\begin{figure}[htbp]
\setlength{\unitlength}{1pt}
\begin{picture}(108,106)
\put(0,0){\includegraphics{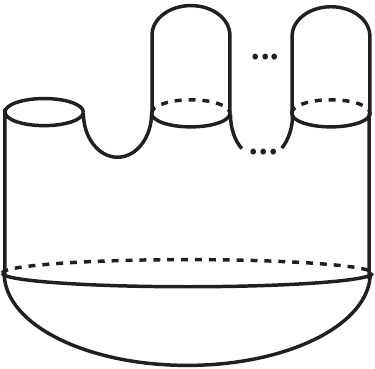}}
\put(50,42){$C$}
\put(50,10){$V$}
\put(85,85){$-Z_{n}$}
\put(44,85){$-Z_{n}$}
\put(-28,24){$M_{\tilde{J}}$}
\put(-10,27){\vector(1,0){9}}
\put(-28,71){$M_{n}$}
\put(-10,74){\vector(1,0){9}}
\end{picture}
\caption{$W_{0}$}\label{fig:w0second}
\end{figure}

\noindent Properties $1$ and $4$ are satisfied with $r(0)=0$. Property $2$ was verified in the proof of Proposition~\ref{prop:familyofmickeys}. Once again, the proof is easier if we inductively prove a more robust version of property $3$.
\vspace{.25in}

\noindent\textbf{Property $3^\prime$}: $W_0\subset W_i$ and for any PTFA coefficient system $\phi:\pi_1(W_i)\to\G$ with $\G^{(n+1)}=1$
$$
\rho(\partial W_i,\phi)=\sigma^{(2)}(W_i,\phi)-\sigma(W_i)=-\sum_{\mathcal{Z}-caps}(\sigma^{(2)}(Z_n,\phi)-\sigma(Z_n))
=-D\rho_0(K)
$$
for some non-negative integer $D$.

\vspace{.25in}
\noindent The last equality in property $3'$ is a consequence of Theorem~\ref{thm:nsolvable}. If $m=1$, it will be clear from the construction that $D=0$ simply because there will be no $\mathcal{Z}$-caps. Property $3'$ for $W_0$ was verified in the proof of Proposition~\ref{prop:familyofmickeys}. Thus we have constructed $W_0$ such that Proposition~\ref{prop:familyofmickeys2} holds for $i=0$.

Now assume that $W_i$, $i\leq n-1$, has been constructed satisfying the properties above.  Before defining $W_{i+1}$, we collect some crucial facts about $W_i$. Consider \emph{any} boundary component $M_{n-i}$ of $W_i$. Recall that, by definition, $J_{n-i}$ is obtained from $R$ by two infections along the circles labelled $\alpha$ and $\beta$. These two circles form a generating set for $\mathcal{A}_0(J_{n-i})$. Even though the boundary of $W_i$ consists of more than just this one copy of $M_{n-i}$, the exact same proof as in Proposition~\ref{prop:familyofmickeys} establishes the following (because Theorem~\ref{thm:nontriviality} applies to \emph{any} boundary component for which the relevant coefficient system is nontrivial).

\begin{itemize}
\item [\textbf{Fact $1$}:] Each of $\{\alpha,\beta\}$ maps into $\pi_1(W_i)^{(i+1)}$
\item [\textbf{Fact $2$}:] The kernel, $\tilde{P}$,  of the map
$$
\pi_1(M_{n-i})^{(1)}\to \pi_1(W_i)^{(i+1)}_r/\pi_1(W_i)^{(i+2)}_r
$$
is of the form $\pi^{-1}(P)$ for some submodule $P$ such that $P\subset P^\perp$ with respect to the classical Blanchfield form and at least one of $\{\alpha,\beta\}$ maps non-trivially under this map.
\end{itemize}

Moreover if $i\leq n-2$ we claim that
\begin{itemize}
\item [\textbf{Fact $3$}:] If $\rho^1(R)\neq 0$ then precisely one of $\{\alpha,\beta\}$ maps non-trivially under the above map.
\end{itemize}

Under the assumptions that $\rho^1(9_{46})\neq 0$ and $i\leq n-2$, the verification \textbf{Fact} $3$ is almost the same as the verification of Fact $3$ from the proof of Proposition~\ref{prop:familyofmickeys}. Specifically, to establish Fact $3$ consider the coefficient system
$$
\phi:\pi_1(M_{n-i})\to \G=\pi_1(W_i)/\pi_1(W_i)_r^{(i+2)}.
$$
Note that $\G^{(n+1)}=1$ since $i+2\leq n+1$. By property $3'$ for $W_i$,
$$
\rho(M_{n-i},\phi)=\sigma^{(2)}(W_i,\phi)-\sigma(W_i)=\sum_{\mathcal{Z}-caps}\sigma^{(2)}(Z_n,\phi_j)-\sigma(Z_n).
$$
But the $Z_n$ are $(n)$-solutions and so are $(i+1.5)$-solutions since $i+1.5\leq n$. Hence, by Theorem~\ref{thm:sliceobstr}, all these signature defects are zero. Thus
$$
\rho(M_{n-i},\phi)=0.
$$
Moreover, by property $2$ for $W_i$, $\phi(\pi_1(M_{n-i}))\subset \pi_1(W_i)^{(i)}$. Therefore $\phi$ restricted to $\pi_1(M_{n-i})$ factors through $\pi_1(M_{n-1})/\pi_1(M_{n-1})^{(2)}$. Just as we argued in the proof of Fact $3$ in the previous section, this implies that one of the first-order signatures of $J_{n-i}$ is zero. Now argue by contradiction. Assuming that \emph{both} $\alpha$ and $\beta$ mapped nontrivially, this first-order signature would be $\rho^1(J_{n-1})$. Since $J_{n-1}$ is obtained from $R$ by two infections using $J_{n-i-1}$ as the infecting knot, as in Example~\ref{ex:first-ordersigs},
$$
0=\rho(M_{n-i},\phi)=\rho^1(J_{n-1})=\rho^1(R)+\rho_0(J_{n-i-1})+\rho_0(J_{n-i-1}).
$$
However, by assumption, $\rho^1(R)=\rho^1(9_{46})\neq 0$ and, since $n-i-1\geq 1$, $J_{n-i-1}$ is $(0.5)$-solvable by Theorem~\ref{thm:nsolvable} and so $\rho_0(J_{n-i-1})=0$ by Theorem~\ref{thm:sliceobstr}. This contradiction implies Fact $3$.

We now give the construction of $W_{i+1}$. Recall that $\partial W_{i}$ consists of copies of $M_{n-i}$ and (possibly) copies of $M_R$. These \textbf{old} copies of $M_R$ will \emph{not} be capped off. However, for each copy of $M_{n-i}$ there are two cases:

\textbf{Case I}: For the specified copy of $M_{n-i}$ in $\partial W_i$, $\alpha$ maps to zero and $\beta$ maps to non-zero under the map
$$
\pi_1(M_{n-i})\overset{j_*}{\lra} \pi_1(W_i)\overset{}{\lra}\pi/\pi_r^{(i+2)}.
$$
By symmetry this will also cover the case when the roles of $\alpha$ and $\beta$ are reversed.

\textbf{Case II}: For the specified copy of $M_{n-i}$ in $\partial W_i$, \textbf{both} of $\{\alpha,\beta\}$ map to non-zero under the map above.

\noindent Since $J_{n-i}$ is obtained from $R$ by two infections using $J_{n-i-1}$ as the infecting knot, there is a corresponding cobordism $E$ with $4$ boundary components: $M_{n-i}$, $M_{R}$ and two copies of $M_{n-i-1}$. $W_{i+1}$ is obtained from $W_i$ by first adjoining, along \emph{each} copy of $M_{n-i}$, a copy of $E$. The newly created copy of $M_R\subset \partial E$ will be called a \textbf{new} copy of $M_R$. Such copies of $E$ lie in either Case I or Case II according to the boundary component $M_{n-i}$ to which they are glued. To an $E$ of Case II nothing more will be added. To an $E$ of Case I further adjoin, to the copy of $M_{n-i-1}\subset \partial E$ whose meridian is equated to $\alpha$, a copy of the special $(n-i-1)$-solution $Z_{n-i-1}$ as constructed in Theorem~\ref{thm:nsolvable}. The latter will again be called a \textbf{null-cap} and the collection of all such be denoted by $\mathcal{N}_i$. Also, for an $E$ of Case I , cap off the new copy of $M_{R}$ with the ribbon disk exterior $\mathcal{R}$ that corresponds to $\alpha$. Such is called an $\mathcal{R}$-cap. This completes the definition of $W_{i+1}$ in all cases. This differs from the proof of Proposition~\ref{prop:familyofmickeys} in only two ways. First, if Case II ever occurs then there will be exposed copies of $M_R$ that will never be capped off, whereas in the proof of Proposition~\ref{prop:familyofmickeys}, Case II only occurred for $i=n-1$, so only $W_n$ had an $M_R$ boundary component (this is because, in the proof of Proposition~\ref{prop:familyofmickeys}, $R$ was chosen specifically so that $\rho^1(R)\neq 0$.) Thirdly, here we cap off the final new copies of $M_R$ (created in going from $W_{n-1}$ to $W_n$) if they arise from a Case I $E$, whereas in the proof of Proposition~\ref{prop:familyofmickeys} we did not (although we could have).

Now we set out to verify properties $1-4$ for $W_{i+1}$. Certainly  $\partial W_{i+1}$  is a disjoint union of some number, say $j(i+1)$, of copies of $M_{n-i-1}$ and some number of copies, say $r(i+1)$, of $M_R$. We seek to show that $j(i+1)=r(i+1)+1$. In this notation, by induction $j(i)=r(i)+1$. When we formed $W_{i+1}$, for \textbf{each} of the boundary components $M_{n-i}$ we adjoined a copy of $E$. This eliminated one boundary component but createsd $3$ new boundary components. In net, \emph{before possibly capping off}, $j(i+1)=j(i)+1$ and $r(i+1)=r(i)+1$. In Case II, nothing more was done. In Case I, one copy of $M_R$ and one copy of $M_{n-i-1}$ was capped off. Thus in any case the difference $j(i+1)-r(i+1)$ is preserved by the addition of new $\mathcal{R}$-caps and null caps. Thus $j(i+1)=r(i+1)+1$ as required by property $1$.

\textbf{Property ($2$) for $W_{i+1}$}:

The proof is essentially identical to that in the proof of Proposition~\ref{prop:familyofmickeys}. Consider a component $M_{n-i-1}\subset\partial W_{i+1}$. Note that $\pi_1(M_{n-i-1})$ is normally generated by the meridian and this meridian is isotopic in $E$ to a push-off of either $\alpha$ or $\beta$ in $M_{n-i}=\partial W_i$. Since both $\alpha$ and $\beta$ lie in the commutator subgroup of $\pi_1(M_{n-i})$,
$$
j_*(\alpha),j_*(\beta)\in \pi_1(W_i)^{(i+1)}
$$
by property $(2)$ for $W_{i}$. Thus
$$
j_*(\pi_1(M_{n-i-1}))\subset \pi_1(W_{i+1})^{(i+1)}
$$
establishing the first part of property $(2)$ for $W_{i+1}$. To prove the second part we need to show that
$j_*(\beta)$ (in Case I) or both $j_*(\alpha)$ and $j_*(\beta)$ (in Case II) are non-zero in $\pi_1(W_{i+1})^{(i+1)}/\pi_1(W_{i+1})^{(i+2)}_r$. Fact $2$ together with the definitions of Case I and II ensure precisely this except for the group $\pi_1(W_i)$ instead of the group $\pi_1(W_{i+1})$. Thus it suffices to show that inclusion induces an isomorphism
\begin{equation}\label{eq:samepi}
\pi_1(W_{i})/\pi_1(W_{i})^{(i+2)}_r\cong \pi_1(W_{i+1})/\pi_1(W_{i+1})^{(i+2)}_r.
\end{equation}
This was already shown in the proof of Proposition~\ref{prop:familyofmickeys}. This completes the verification of property ($2$) for $W_{i+1}$.

\textbf{Property ($3'$) for $W_{i+1}$}:

Since property $3'$ holds for $W_i$ and since $W_{i+1}$ is obtained from $W_i$ by adjoining copies of $E$ and possibly some $\mathcal{R}$-caps and null-caps, it suffices to prove that the signature defect is zero on the extra pieces $E$, $\mathcal{R}$ and $\mathcal{N}_{i}$. But this was established already in Proposition~\ref{prop:familyofmickeys}. This concludes the verification of property ($3'$) for $W_{i+1}$.

\textbf{Property ($4$) for $W_{i+1}$}

Suppose $\rho^1(9_{46})\neq 0$ and $i+1<n$. Inductively we may suppose that $r(i)=0$, that is $\partial W_i=M_{n-i}$. Then $i\leq n-2$ so Fact 3 holds. Consequently in the passage from $W_i$ to $W_{i+1}$ the Case II never occurs. Thus the new copy of $M_R$ and one of the copies of $M_{n-i-1}$ are capped off, so that $\partial W_{i+1}=M_{n-i-1}$. Hence $r(i+1)=0$. If $i+1=n$ then we still may assume inductively that $r(n-1)=0$ and $\partial W_{n-1}=M_{1}$, but now Fact 3 may not hold. Nonetheless, since then only one copy of $E$ is adjoined in going from $W_{n-1}$ to $W_n$, $r(n)$ is either $0$ or $1$. In the latter case $M_R\subset\partial W_n$ and Case II must have occurred. Then the induced coefficient system
$$
\phi_R:\pi_1(M_R)\to \pi_1(W_n)\to \pi_1(W_n)/\pi_1(W_n)^{(n+1)}_r
$$
factors through $\pi_1(M_R)/\pi_1(M_R)^{(2)}$ since $\pi_1(M_R)$ is normally generated by the meridian of $R$, which is isotopic in its copy of $E$ to the meridian of a copy of $M_{1}\subset \partial W_{n-1}$ and by property $2$ this meridian lies in $\pi_1(W_{n-1})^{(n-1)}$ and hence in $\pi_1(W_{n})^{(n-1)}$. We now just need to establish that $\phi_R$ induces an embedding of $\pi_1(M_R)/\pi_1(M_R)^{(2)}$ since this will identify $\rho(M_R,\phi_R)$ as $\rho^1(R)$. Combining Fact $2$ in the case $i=n-1$ and the fact that $M_R$ arose from an $E$ of Case II, the kernel, $\tilde{P}$, of the map
$$
\pi_1(M_{1})^{(1)}\to \pi_1(W_{n-1})^{(n)}_r/\pi_1(W_{n-1})^{(n+1)}_r
$$
is zero. Thus $\pi_1(M_1)/\pi_1(M_1)^{(2)}$ embeds in $\pi_1(W_{n-1})/\pi_1(W_{n-1})^{(n+1)}_r$. By equation ~\ref{eq:samepi} with $i=n-1$, $\pi_1(M_1)/\pi_1(M_1)^{(2)}$ embeds in $\pi_1(W_n)/\pi_1(W_n)^{(n+1)}_r$. But the Alexander modules of $M_1$ and $M_R$ are isomorphic, with the meridian and the curves $\alpha$ and $\beta$ being identified in $E$. This shows that $\phi_R$ induces an embedding of $\pi_1(M_R)/\pi_1(M_R)^{(2)}$.

This concludes the proof of Proposition~\ref{prop:familyofmickeys2}.
\end{proof}

More generally, the proof above proves this more general result about iterated generalized doublings of knots.

\begin{thm}\label{thm:main3} Suppose $R_i$, $1\leq i\leq n$, is a set of (not necessarily distinct) slice knots and, for each fixed $i$, $\{\eta_{i1},...,\eta_{im_i}\}$ is a trivial link of circles in $\pi_1(S^3-R_i)^{(1)}$ such that for some $ij$ and $ik$ (possibly equal) $\mathcal{B}\ell_0^i(\eta_{ij},\eta_{ik})\neq 0$, where $\mathcal{B}\ell_0^i$ is the classical Blanchfield
  form of $R_i$. Then there exists a constant $C$ such that if $K$ is any knot with Arf$(K)=0$ and $|\rho_0(K)|>C$, the result, $R_{n}\circ\dots\circ R_{1}(K)$, of n-times iterated generalized doubling,
   is of infinite order in the smooth and topological concordance groups, and moreover represents an element of infinite order in $\mathcal{F}_{n}/\mathcal{F}_{n.5}$.
\end{thm}

There are situations where the constant can be taken independent of $n$ (as in Theorem~\ref{thm:main}) but we shall not state it in generality.

\begin{proof}[Proof of Theorem~\ref{thm:main3}] Recall that Arf$(K)=0$ if and only $K$ is $(0)$-solvable by Remark~\ref{rem:0solv}. Since the infections are being done along curves that lie in the commutator subgroup, any $n$-times iterated operator applied to such a knot $K$ results in an $(n)$-solvable knot by Theorem~\ref{thm:nsolvable}. Let $\mathcal{J}_n$ denote the result of such an operator. Choose $C'$ to be the maximum of the Cheeger-Gromov constants for the $\{M_{R_i}\}$. Let $m$ be the maximum of the $m_j$. Choose $C$ such that
$$
\left(\frac{m^n-1}{m-1}\right)C'\leq C.
$$
The proof then proceeds exactly like that of part $1$ of Theorem ~\ref{thm:main} above. Suppose that a non-trivial multiple, $\tilde{\mathcal{J}}$, of $\mathcal{J}_n$ were rationally $(n.5)$-solvable. We show that $|\rho_0(K)|\leq C$. We recursively construct $4$-manifolds $W_i$ as in Proposition~\ref{prop:familyofmickeys2}. The primary difference in the argument is that the various cobordisms, $E_i$, that arise have $2+m_i$ boundary components. Letting $\mathcal{J}_i=R_{i}\circ\dots\circ R_{1}(K)$ and $M_{i}=M_{\mathcal{J}_i}$ one establishes recursively:

\begin{prop}\label{prop:familyofmickeys3} Under the assumption that $\tilde{\mathcal{J}}$ is rationally $(n.5)$-solvable, for each $0\leq j\leq n$ there exists a $4$-manifold $W_j$ with the following properties. Letting $\pi=\pi_1(W_j)$,

\begin{itemize}
\item [(1)] $W_j$ is a rational $(n)$-bordism whose boundary is a disjoint union of copies of $M_{R_i}$ (the total number of copies being at most $\frac{m^i-1}{m-1}$), together with a positive number of copies of $M_{n-j}$~;
\item [(2)] Each inclusion $j:M_{n-j}\subset\partial W_j\to W_j$ satisfies
$$
j_*(\pi_1(M_{n-j}))\subset \pi^{(j)};
$$
and
$$
j_*(\pi_1(M_{n-j}))\cong \mathbb{Z}\subset \pi^{(j)}/\pi^{(j+1)}_r;
$$
\item [(3)] For any PTFA coefficient system $\phi:\pi_1(W_j)\to\G$ with $\G^{(n+1)}_r=1$
$$
\rho(\partial W_j,\phi)\equiv\sigma_\G^{(2)}(W_j,\phi)-\sigma(W_j)=-D\rho_0(K)
$$
for some non-negative integer $D$ (depending on $\phi$).
\end{itemize}
\end{prop}
Assuming this and applying it in the case $j=n$ one deduces:
\begin{equation}\label{eq:9.1}
(k+D)\rho_0(K)=-\sum_{i=1}^{r(i)}\rho_0(M_{R_i},\phi^R_i).
\end{equation}
where $k$ is the number of boundary components of $W_n$ that are copies of $M_K$ and $r(i)$ is the number of boundary components of $W_n$ that are copies of $M_{R_i}$. Thus
\begin{equation}\label{eq:9.2}
|\rho_0(K)|\leq \frac{r(1)+...+r(n)}{k+D}C'\leq \left(\frac{m^n-1}{m-1}\right) C' \leq C.
\end{equation}
The crucial point is that $k\geq 1$. Hence if $|\rho_0(K)|> C$ then $\tilde{J}$ is not is rationally $(n.5)$-solvable, thereby completing the proof.

In the construction of the $W_j$, in this general case, there will be no $\mathcal{R}$-caps. The $M_{R_i}$ boundary components that appear at each level are allowed to persist. The key point is that the analogue of Fact $2$ (proof of Proposition~\ref{prop:familyofmickeys2}) applies. This fact, together with our hypothesis on the $\eta_{ij}$, ensures that the kernel $P$ of Fact $2$ cannot contain every $\eta_{ij}$, so that there is always at least one $\eta_{ij}$ that maps nontrivially (as in Fact 2). This translates into the fact that $W_j$ always has at least one boundary component of the form $M_{n-j}$. Then it is an easy combinatorial exercise to see that if one never has any null caps that the number of copies of $M_K$ in $\partial W_n$ is precisely $0+1+m_n+m_nm_{n-1}+m_nm_{n-1}m_{n-2}+...+m_nm_{n-1}...m_2$ which is at most $1+m+m^2+...+m^{n-1}$. This is the maximum number of copies possible. The number is not very important-just the fact that there is a bound independent of $K$. The proof is completed just as in the proof of Proposition~\ref{prop:familyofmickeys2}.
\end{proof}

A nice application of the more general theorem is the following which gives new information about the concordance order of knots that previously could not be distinguished from an order two knot.

\begin{cor}\label{cor:torsion} For any $n>0$ there is a constant $D$ such that if $|\rho_0(J_0)|>D$ then the knot $K_{n}$ of Figure~\ref{fig:torsionfigeight} is of infinite order in the topological and
 smooth concordance groups.
\end{cor}

\begin{proof}[Proof of Corollary~\ref{cor:torsion}] Any odd multiple of $K_{n}$ has Arf invariant one and hence is not a slice knot, nor even $(0)$-solvable. Let $J=\#^{2k}K_{n}$. Let $\tilde{R}$ denote
 the ribbon knot obtained as a connected sum of $2k$ copies of the figure eight knot. Then $J$ is obtained from $\tilde{R}$ by $4k$ infections along a generating set for the Alexander module of $\tilde{R}$, using
  the knot $J_{n-1}$ in each case. Recall that $J_{n-1}=R\circ...\circ R(J_0)$ ($n-1$ times) using the operator $R_{\{\alpha,\beta\}}$ of Figure~\ref{fig:Rdoubling}. Thus
$$
J=\tilde{R}\circ R\circ...\circ R(J_0)
$$
and Theorem~\ref{thm:main3} applies  to show that $J$ is not slice.
\end{proof}

 \bibliographystyle{plain}
\bibliography{mybib5}
\end{document}